\documentclass[10pt,letterpaper]{article}

\usepackage[letterpaper, top=.95in, bottom=1in, left=1.2in, right=1.2in, marginparwidth=2.3cm, marginparsep=1mm, includefoot]{geometry}

\usepackage{latexsym,
amsthm,
color,
amssymb,
url,
bm,
cite,
amssymb,
microtype,
amsmath,
amsfonts,
mathtools}

\usepackage{amssymb}
\usepackage{microtype}
\usepackage{amsmath}
\usepackage{amsfonts}
\usepackage{mathtools}
\usepackage[mathscr]{euscript}
\usepackage{amsthm}
\usepackage{nicefrac}
\usepackage{dsfont}

\usepackage{caption}
\usepackage{setspace}
\usepackage{subcaption}
\usepackage{setspace}

 \usepackage[pdftex, plainpages = false, pdfpagelabels, 
                 bookmarks = true,
                 bookmarksopen = true,
                 bookmarksnumbered = true,
                 breaklinks = true,
                 linktocpage,
                 pagebackref,
                 colorlinks = true,  
                 linkcolor = blue,
                 urlcolor  = blue,
                 citecolor = red,
                 anchorcolor = green,
                 hyperindex = true,
                 hyperfigures
                 ]{hyperref}

\usepackage[inline]{enumitem}
\usepackage[mathscr]{euscript}
\usepackage{amsthm}
\usepackage{nicefrac}
\usepackage{dsfont}
\usepackage{tikz}
\usepackage[ruled,linesnumbered]{algorithm2e}
\usepackage{Figures/kbordermatrix}

\usepackage{nicefrac}
\usepackage[mathscr]{euscript}

\RequirePackage[T1]{fontenc}
\RequirePackage[utf8]{inputenc}
\usepackage[english]{babel}
\usepackage{alphabeta}
\usepackage{graphicx}

\usetikzlibrary{math}

\makeatletter
\input{colordvi}

\usepackage{aliascnt}

\hypersetup{
    colorlinks,
    linkcolor={red!75!black},
    citecolor={green!75!black},
    urlcolor={blue!75!black},
}

\newcommand{\cref}[1]{\autoref{#1}}

\definecolor{BrilliantRose}{rgb}{1.0, 0.33, 0.64}
\usepackage{todonotes}
\newcommand{\rev}[1]{}

\newcommand{\ans}[1]{}

\newcommand{\remove}[1]{}

\newcounter{func}
\newcommand{\funref}[1]{\hyperref[#1]{f_{\ref*{#1}}}} 

\tikzset{black node/.style={draw, circle, fill = black, minimum size = 5pt, inner sep = 0pt}}
\tikzset{white node/.style={draw, circlternary_treese, fill = white, minimum size = 5pt, inner sep = 0pt}}
\tikzset{normal/.style = {draw=none, fill = none}}
\tikzset{lean/.style = {draw=none, rectangle, fill = none, minimum size = 0pt, inner sep = 0pt}}
\usetikzlibrary{decorations.pathreplacing}
\usetikzlibrary{arrows.meta}

\usetikzlibrary{shapes}
\tikzset{diam/.style={draw, diamond, fill = black, minimum size = 7pt, inner sep = 0pt}}

\newcommand{\Acal}{\mathcal{A}}

\newcommand{\Ccal}{\mathcal{C}}

\newcommand{\Hcal}{\mathcal{H}}

\newcommand{\Lcal}{\mathcal{L}}

\newcommand{\Ocal}{\mathcal{O}}

\newcommand{\Scal}{\mathcal{S}}
\newcommand{\Tcal}{\mathcal{T}}

\newcommand{\Xcal}{\mathcal{X}}

\newcommand{\Nbbb}{\mathbb{N}}

\newcommand{\Rbbb}{\mathbb{R}}

\newcommand{\Zbbb}{\mathbb{Z}}

\RequirePackage{stmaryrd}
\usepackage{textcomp}
\DeclareUnicodeCharacter{2286}{\subseteq}
\DeclareUnicodeCharacter{2192}{\ifmmode\to\else\textrightarrow\fi}
\DeclareUnicodeCharacter{2203}{\ensuremath\exists}
\DeclareUnicodeCharacter{183}{\cdot}
\DeclareUnicodeCharacter{2200}{\forall}
\DeclareUnicodeCharacter{2264}{\leq}
\DeclareUnicodeCharacter{2265}{\geq}
\DeclareUnicodeCharacter{8614}{\mathbin{\mapsto}}
\DeclareUnicodeCharacter{8656}{\Leftarrow}
\DeclareUnicodeCharacter{8657}{\Uparrow}
\DeclareUnicodeCharacter{8658}{\Rightarrow}
\DeclareUnicodeCharacter{8659}{\Downarrow}
\DeclareUnicodeCharacter{8669}{\rightsquigarrow}
\newcommand{\eqdef}{\stackrel{{\scriptsize\rm def}}{=}}
\DeclareUnicodeCharacter{8797}{\eqdef}
\DeclareUnicodeCharacter{8870}{\vdash}
\DeclareUnicodeCharacter{8873}{\Vdash}
\DeclareUnicodeCharacter{22A7}{\models}
\DeclareUnicodeCharacter{9121}{\lceil}
\DeclareUnicodeCharacter{9123}{\lfloor}
\DeclareUnicodeCharacter{9124}{\rceil}
\DeclareUnicodeCharacter{2208}{\in}
\DeclareUnicodeCharacter{9126}{\rfloor}
\DeclareUnicodeCharacter{9655}{\triangleright}
\DeclareUnicodeCharacter{9665}{\triangleleft}
\DeclareUnicodeCharacter{9671}{\diamond}
\DeclareUnicodeCharacter{9675}{\circ}
\DeclareUnicodeCharacter{10178}{\bot}
\DeclareUnicodeCharacter{10214}{} 
\DeclareUnicodeCharacter{10215}{} 
\DeclareUnicodeCharacter{10229}{\longleftarrow}
\DeclareUnicodeCharacter{10230}{\longrightarrow}
\DeclareUnicodeCharacter{10231}{\longleftrightarrow}
\DeclareUnicodeCharacter{10232}{\Longleftarrow}
\DeclareUnicodeCharacter{10233}{\Longrightarrow}
\DeclareUnicodeCharacter{10234}{\Longleftrightarrow}
\DeclareUnicodeCharacter{10236}{\longmapsto}
\DeclareUnicodeCharacter{10238}{\Longmapsto} 
\DeclareUnicodeCharacter{10503}{\Mapsto}    
\DeclareUnicodeCharacter{10971}{\mathrel{\not\hspace{-0.2em}\cap}}
\DeclareUnicodeCharacter{65294}{\ldotp}
\DeclareUnicodeCharacter{65372}{\mid}

\definecolor{DarkTangerine}{rgb}{1.0, 0.66, 0.07}
\definecolor{darkyellow}{rgb}{.7, .6, 0.0}
\definecolor{CornflowerBlue}{rgb}{0.39, 0.58, 0.93}
\definecolor{DarkGoldenrod}{rgb}{0.72, 0.53, 0.04}
\definecolor{BritishRacingGreen}{rgb}{0.0, 0.26, 0.15}
\definecolor{AO}{rgb}{0.0, 0.5, 0.0}
\definecolor{MidnightBlack}{rgb}{0.1,0.1,.34}
\definecolor{MidnightBlue}{rgb}{0.1,0.1,0.43}
\definecolor{Black}{rgb}{0,0, 0}
\definecolor{Blue}{rgb}{0, 0 ,1}
\definecolor{Red}{rgb}{1, 0 ,0}
\definecolor{White}{rgb}{1, 1, 1}
\definecolor{DeepMagenta}{rgb}{0.8, 0.0, 0.8}
\definecolor{grey}{rgb}{.6, .6, .6}
\definecolor{darkgrey}{rgb}{.33, .33, .33}
\definecolor{Mygreen}{rgb}{.0, .7, .0}
\definecolor{Yellow}{rgb}{.55,.55,0}
\definecolor{Mustard}{rgb}{1.0, 0.86, 0.35}
\definecolor{applegreen}{rgb}{0.55, 0.71, 0.0}
\definecolor{darkturquoise}{rgb}{0.0, 0.81, 0.82}
\definecolor{celestialblue}{rgb}{0.29, 0.59, 0.82}
\definecolor{green_yellow}{rgb}{0.68, 1.0, 0.18}
\definecolor{crimsonglory}{rgb}{0.75, 0.0, 0.2}
\definecolor{darkmagenta}{rgb}{0.30, 0.0, 0.30}
\definecolor{magenta}{rgb}{0.50, 0.0, 0.50}
\definecolor{internationalorange}{rgb}{1.0, 0.31, 0.0}
\definecolor{darkorange}{rgb}{1.0, 0.55, 0.0}
\definecolor{ao}{rgb}{0.0, 0.5, 0.0}
\definecolor{awesome}{rgb}{1.0, 0.13, 0.32}
\definecolor{darkcyan}{rgb}{0.0, 0.50, 0.50}
\definecolor{violet}{rgb}{0.93, 0.51, 0.93}
\definecolor{brown}{rgb}{0.65, 0.16, 0.16}
\definecolor{orange}{rgb}{1.0, 0.65, 0.0}
\definecolor{DarkGreen}{rgb}{0,.5,0}
\definecolor{BostonUniversityRed}{rgb}{0.8, 0.0, 0.0}
\definecolor{BrightLavender}{rgb}{0.75, 0.58, 0.89}
\definecolor{DarkLavender}{rgb}{0.45, 0.31, 0.59}
\definecolor{ChromeYellow}{rgb}{1.0, 0.65, 0.0}

\definecolor{DarkGreen}{rgb}{0.547,.5,0}

\newcommand{\red}[1]{#1}
\newcommand{\green}[1]{{\color{Mygreen}#1}}

\definecolor{Red}{rgb}{1, 0 ,0}
\definecolor{Blue}{rgb}{0, 0 ,1}

\newtheorem{theorem}{Theorem}[section]

\newaliascnt{question}{theorem}

\aliascntresetthe{question}

\newaliascnt{lemma}{theorem}
\newtheorem{lemma}[lemma]{Lemma}
\aliascntresetthe{lemma}

\newaliascnt{claim}{theorem}

\aliascntresetthe{claim}

\newaliascnt{invariant}{theorem}

\aliascntresetthe{invariant}

\newaliascnt{proposition}{theorem}
\newtheorem{proposition}[proposition]{Proposition}
\aliascntresetthe{proposition}

\newaliascnt{observation}{theorem}
\newtheorem{observation}[observation]{Observation}
\aliascntresetthe{observation}

\newaliascnt{corollary}{theorem}
\newtheorem{corollary}[corollary]{Corollary}
\aliascntresetthe{corollary}

\newaliascnt{definition}{theorem}
\newtheorem{definition}[definition]{Definition}
\aliascntresetthe{definition}

\newaliascnt{conjecture}{theorem}

\aliascntresetthe{conjecture}

\newaliascnt{counterexample}{theorem}

\aliascntresetthe{counterexample}

\newcommand{\hh}{\end{document}}

\newcommand{\tw}{{\sf tw}\xspace}

\newcommand{\poly}{\text{$\mathsf{poly}$}\xspace}

\newcommand{\Perimeter}{{\sf Perimeter}\xspace}
\newcommand{\Boundary}{{\sf Boundary}\xspace}
\newcommand{\abs}[1]{\ensuremath{\lvert #1 \rvert}}

\newcommand{\FFcal}{\mathcal{M}}

\newcommand*{\Floor}[1]{\left\lfloor #1 \right\rfloor}

\newcommand*{\Choose}[2]{{ #1 \choose #2}}

\usepackage{color}

\title{Excluding surfaces as minors in graphs\thanks{Emails of authors: 
  \texttt{sedthilk@thilikos.info}, \texttt{wiederrecht@kaist.ac.kr}}}

\author{\bigskip\large Dimitrios M. Thilikos\thanks{LIRMM, Univ Montpellier, CNRS, Montpellier, France.}~$^{,}$\thanks{Supported by the French-German Collaboration ANR/DFG Project UTMA (ANR-20-CE92-0027), the ANR project GODASse ANR-24-CE48-4377,  and by the Franco-Norwegian project PHC AURORA 2024 (Projet n°\! 51260WL).}\and\large 
 \and\large 
 Sebastian Wiederrecht\thanks{School of Computing, KAIST, Daejeon, South Korea.}}
\date{}

\begin{document}

\maketitle

\begin{abstract}
\noindent The Graph Minors Structure Theorem (GMST) of Robertson and Seymour states that for every graph $H,$ any $H$-minor-free graph $G$ has a tree-decomposition of bounded adhesion such that the torso of every bag embeds in a surface $\Sigma$ where $H$ does not embed after removing a small number of \textsl{apex vertices} and confining some vertices into a bounded number of \textsl{bounded depth} vortices.
However, the functions involved in the original form of this statement were not explicit.
In an enormous effort Kawarabayashi, Thomas, and Wollan proved a similar statement with explicit (and single-exponential in $|V(H)|$) bounds.
However, their proof replaces the statement ``\textsl{a surface where $H$ does not embed}'' with ``\textsl{a surface of Euler-genus in $\mathcal{O}(|H|^2)$}''.

In this paper we close this gap and prove that the bounds of Kawarabayashi, Thomas, and Wollan can be achieved with a tight bound on the Euler-genus.
Moreover, we provide a more refined version of the GMST focussed exclusively on excluding, instead of a single graph, grid-like graphs that are minor-universal for a given set of surfaces.
This allows us to give a description, in the style of Robertson and Seymour, of graphs excluding a graph of fixed Euler-genus as a minor, rather than focussing on the size of the graph. 
\end{abstract}
\medskip

\noindent \textbf{Keywords:} Graph minor, Surface embedding, Euler Genus, Graph structure theorem.

\thispagestyle{empty}
\newpage

\pagenumbering{arabic}

\section{Introduction}\label{mentalization}

A graph $H$ is a \emph{minor} of a graph $G$ if it can be obtained from $G$ by a sequence of vertex deletions, edge deletions, and edge contractions.
In their seminal series of papers on graph minors, Robertson and Seymour (R\&S)~\cite{robertson2003graph} proved a general structure theorem for graphs that exclude a fixed graph $H$ as a minor.
To state this theorem, we need a couple of definitions.
Let $a$ and $b$ be non-negative integers.
We say that a graph $G$ has an \emph{$a$-almost embedding}
\rev{1. Where $a$ is used in an a-almost embedding?}
\ans{Fixed.}
in a surface $\Sigma$ of \emph{breadth $b$} if there exists a set $A\subseteq V(G),$ called the \emph{apex set}, \red{with $|A|\leq a$} such that
\rev{2. P1, L7: “$G$” should be “$G-A$”.}
\ans{Fixed.} 
$G-A = G_0 \cup G_1 \cup \dots \cup G_b$ where
\begin{enumerate}
  \item $G_0$ can be drawn on $\Sigma$ without crossings such that this embedding has  $b$ distinct faces\footnote{These faces are called the \emph{vortices} of $G.$ See \cref{subsec_sigmadec} for a formal definition and \cref{prefascist} specifically for a definition of the path decomposition (we call it a ``linear decomposition'' later on) for vortices.}
  \rev{3. P1, L9: Formally, a face is a region of the surface after deleting the drawing of the graph. So every face is an open set and cannot contain any vertex. Moreover, you have to assume that no vertices other than the ones in $V(G_0)\cap V(G_i)$ are on the boundary of $F_i$.}
  \ans{We now make clear that we refer to the boundary of the faces.
We also specify that only the vertices of $V(G_0)\cap V(G_i)$  lie on the  boundary of $F_i$.
}
  $F_1,\dots,F_b$ and for every $i\in[b]$ the vertices of $V(G_0)\cap V(G_i)$, \red{and only those},  all lie on \red{the boundary of} $F_i,$ and
  \item for each $i\in\{ 1,\dots,b\},$ $G_i$ is a graph with a path decomposition\footnote{For the definition of tree decompositions see \cref{inphasized}. A \emph{path decomposition} of a graph $G$ is a tree decomposition $(T,\beta)$ for $G$ where $T$ is a path.}
  of width at most $b$ such that the vertices of $V(G_0)\cap V(G_i)$ appear in distinct bags of this decomposition exactly in the order as they appear on \red{the boundary of} $F_i.$
\rev{4. P1, L11: You have to assume the boundary of $F_i$ is a cycle, for otherwise some vertex can appear more than once and “the order as they appear on $F_i$” does not make sense.}
\ans{We now explain that the vertices of $V(G_0)\cap V(G_i)$  lie on the  boundary of $F_i$ in the order they appear in this boundary.
}
\end{enumerate}
A graph $G$ is a \emph{clique sum} of two graphs $G_1$ and $G_2$ if there exist cliques $S_i\subseteq G_i$ for both $i$ such that $|S_1|=|S_2|$ and $G$ can be obtained from $G_1$ and $G_2$ by identifying the vertices of $S_1$ with those of $S_2$ into a clique $S$ and possibly removing some of the edges with both endpoints in $S.$

With these definition, the Graph Minor Structure Theorem of Robertson and Seymour reads as follows.
\rev{9. P1, L27: This result should be proved in somewhere much earlier than [23].}
\ans{To our knowledge, the first time the \textsl{Graph Minor Structure Theorem} is proved in \green{\sf [Neil Robertson and P. D. Seymour. Graph minors. XVI. Excluding
a non-planar graph. J. Combin. Theory Ser. B, 89(1):43–76, 2003.]}.}
\begin{proposition}[Graph Minor Structure Theorem (GMST), \cite{robertson2003graph}]
\label{thm_GMST1}
There exists a function $f\colon\mathbb{N}\to\mathbb{N}$ such that for all graphs $H$ and $G$ either
\begin{enumerate}
  \item $H$ is a minor of $G,$ or
  \item $G$ can be obtained by means of clique sums from graphs that are $f(|V(H)|)$-almost embeddable with breadth at most $|V(H)|^2$ in surfaces \textbf{where $H$ cannot be embedded}. 
 \rev{5.  P1, L21: “in surfaces where $H$ does not embed” should be “in surfaces where $H$ cannot be embedded”.}
 \ans{Fixed.} 
\end{enumerate}
\end{proposition}

\red{The arguments of Robertson and Seymour leading to \cref{thm_GMST1} essentially split all $H$-minor-free graphs into two cases: Those of \textsl{small} treewidth and those of \textsl{large} treewidth.
This is, because they require a large grid minor as the starting point of their construction of an almost-embedding.
That is, an implicit lower bound on $f(|V(H)|)$ can be derived from the minimum order of a grid minor that is required for an application of, for example, the Flat Wall Theorem \cite{Robertson1995GMXIII,Kawarabayashi2018NewProof}.
Note that such lower bounds should be treated as lower bounds to the general \textsl{method} rather than a lower bound on the -- theoretically -- best possible function $f$.}
\rev{6. This sentence does not make sense because the treewidth of H-minor-free graphs can be any
integer, as long as H is non-planar.}
\ans{We reformulated the sentence and explain more directly where lower bound bottlenecks come from. We also stress that we are talking about the best the method can achieve, it might be possible to get better bounds on $f$ in other ways.}
Moreover, a particularly difficult to handle contributions to its upper bound stems from the fact that, in order to prove the ``\textsl{where $H$ cannot be embedded}''-part of \cref{thm_GMST1},
\rev{7. It is incorrect to write “ H does not embed” or “graph that embeds on a surface”. They
should be “a graph is embedded in/on a surface” or “a surface embeds a graph”. Similar errors appear
numerous times in this paper.}
\ans{We adopt the terminology ``a graph is embedded in a surface'' all over the paper.} 
one needs to show that there is some function $g$ such that for every $h,$ any graph that embeds on a surface $\Sigma$ with \textsl{representativity}\footnote{Roughly: the minimum length of a non-contractible cycle.} at least $g(h)$ will contain all graphs on ${h}$ vertices that embed in $\Sigma$ as a minor (see Theorem 4.3 in \cite{robertson2003graph}).
This proof, however, is non-constructive.

In recent years, many steps at making the graph minors theory of R\&S constructive have been made.
Arguably the biggest step here is a new, and fully constructive, proof of a variant of \cref{thm_GMST1} by Kawarabayashi, Thomas, and Wollan.
Their version of the GMST reads as follows.

\begin{proposition}[\!\! \cite{KawarabayashiTW20Quicklyexcluding}]
\label{thm_GMST2}
There exists a function $f\colon\mathbb{N}\to\mathbb{N}$ with $f(k)\in 2^{\mathsf{poly}(k)}$ such that for all graphs $H$ and $G$ either
\begin{enumerate}
  \item $H$ is a minor of $G,$ or
  \item $G$ can be obtained by means of clique sums from graphs that are $f(|V(H)|)$-almost embeddable with breadth at most $|V(H)|^2$ in surfaces \textbf{of Euler genus $\mathcal{O}(|V(H)|^2)$}. 
\end{enumerate}
\end{proposition}

Notice that the bound on the Euler genus of the surfaces involved in \cref{thm_GMST2} \red{is  worse as the two quantities may differ arbitrarily, depending on $H$}.
\rev{10.  P1, L-9: It is not fair to claim that the bound is “signiﬁcantly worse”. The Euler genus in Proposition 1.2 is at the correct order when $E(H) = \Omega (|V(H)|^2)$.}  
\ans{We now write ``is  worse as the two quantities may differ arbitrarily, depending on $H$".}
The reason for this worse bound is the non-constructiveness in the original proof by R\&S \cite{robertson2003graph} and the determination of Kawarabayashi, Thomas, and Wollan to avoid these arguments in order to make their theorem constructive.

\paragraph{Local structure theorems.}
Both \cref{thm_GMST1} and \cref{thm_GMST2}
\rev{11. P2, L1: “Both, Proposition ...” should be “Both Proposition ...”}
\ans{Fixed.}
 are theorems that describe the structure of $H$-minor-free graphs \textsl{globally}.
However, both theorems, as most results in this area, are based on much more technical \textsl{local} versions.
The world ``local'' refers to the structure in an $H$-minor-free graph that is governed by a large \textsl{wall}\footnote{See \cref{subsec_wallsandtangles} for related definitions of walls and tangles} or, in other words, the structure that is controlled by a large-order \textsl{tangle}.
Let $W$ be an $r$-wall for some $r\geq 3.$
A \emph{simplification of $W$} is an $r$-wall $W'$ which is obtained by contracting some edges of $W.$
In this language one can restate \cref{thm_GMST2} in its local form as follows.

\begin{proposition}\label{thm_GMST2local}
There exists functions $f,g\colon\mathbb{N}\to\mathbb{N}$ with $f(k),g(k)\in 2^{\mathsf{poly}(k)}$ such that for all graphs $H$ and all graphs $G$ containing a $g(|V(H)|)$-wall $W$ either
\begin{enumerate}
  \item there is a minor-model of $H$ which is highly connected\footnote{The term ``highly connected'' here is slightly informal. Formally this means that there exists a minor model $\{ X_v\}_{v\in V(H)}$ of $H$ such that each $X_v$ intersects many intersections of rows and columns of $W.$ (See \cref{inphasized} for the definitions of a minor model.)} to $W,$ or
  \item $G$ can be obtained by means of clique sums from graphs that are $H$-minor-free and a graph $G_W$ which is $f(|V(H)|)$-almost embeddable with breadth at most $|V(H)|^2$ in a surface \textbf{of Euler genus $\mathcal{O}(|V(H)|^2)$} such that there is a simplification $W'$ of $W$ with $W'\subseteq G_W.$ 
\end{enumerate}
\end{proposition}
\rev{12. P2, L10: I do not see where Proposition 1.3 is stated in [17]. Similar results do exist in [17], but they
do not include the H-minor-freeness.}
\ans{We now explain that this is a restatement of Theorem 1.3 of \green{\textsf{[Ken-ichi Kawarabayashi, Robin Thomas, and Paul Wollan. Quickly
excluding a non-planar graph. CoRR, abs/2010.12397, 2020]}}.}

\red{The above is a restatement of Theorem 1.3 of \cite{KawarabayashiTW20Quicklyexcluding}).}
The local structure theorem describes a single summand in the clique-sum structure provided by statements like \cref{thm_GMST1} and \cref{thm_GMST2}.
While such a local structure theorem only describes those summands of large treewidth, this is usually sufficient as all summands of small treewidth can be further split into summands of small \textsl{size} (which is the essence of the definition of treewidth).
A simple greedy procedure, first introduced by R\&S \cite{robertson1991graph}, then allows to immediately go from a local structure theorem to a global one\footnote{In case of this paper, such a proof can be found in \cref{locglob}.} while increasing the involved function at most by some polynomial.

\subsection{Our contribution}\label{subsec_contribution}
Recently, Gavoille and Hilaire \cite{gavoille2023minor} proved the existence of a minor-universal graph for all graphs that embed in a fixed surface $\Sigma.$
In particular, the bounds they achieve are polynomial in the Euler genus of $\Sigma,$ that is, they prove that for every integer $\mathsf{h},$ their universal graph of \textsl{order} $\mathsf{poly}(g\cdot h)$ contains every graph on at most $h$ vertices that embeds in $\Sigma$ as a minor.
Here $g$ denotes the Euler genus of $\Sigma.$
We leverage this result to prove the following local version of the GMST.

\begin{theorem}\label{thm_GMST3}
There exists functions $f,g\colon\mathbb{N}\to\mathbb{N}$ with $f(k),g(k)\in 2^{\mathsf{poly}(k)}$ such that for all graphs $H$ and all graphs $G$ containing a $g(|V(H)|)$-wall $W,$ either
\begin{enumerate}
  \item $G$ contains a minor-model of $H$ which is highly connected to $W,$ or
  \item $G$ can be obtained by means of clique sums from graphs that are $H$-minor-free and a graph $G_W$ which is $f(|V(H)|)$-almost embeddable with breadth at most $|V(H)|^2$ in a surface \textbf{where $H$ does not embed} such that there is a simplification $W'$ of $W$ with $W'\subseteq G_W.$ 
\end{enumerate}
\end{theorem}

With this, we achieve a first unification of the tight bound on the Euler genus from \cref{thm_GMST1} and the constructive bound on the function $f$ in \cref{thm_GMST2}.

This is done in two steps. 
First, we define several families of grid-like graphs representing a fixed surface $\Sigma,$ then we show in \cref{sec_dycks} that, up to a function depending linearly on the order of the graph and exponentially on the Euler genus of $\Sigma,$ all of these graphs are equivalent.
See \cref{supplanted} for some examples of such graphs.

\begin{figure}[ht]
  \begin{center}
  \scalebox{.9}{\includegraphics{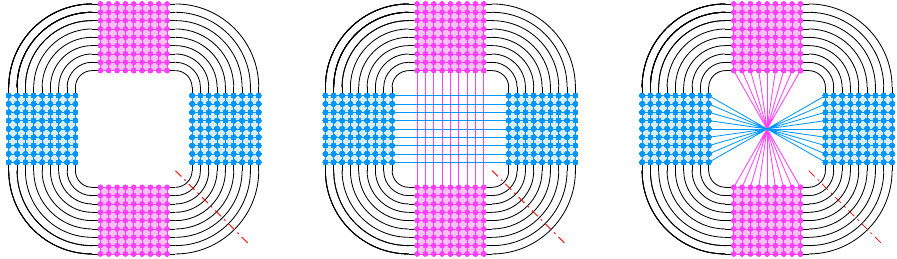}}
  \end{center}
  \caption{The annulus grid $\mathscr{A}_{9},$ the handle grid $\mathscr{H}_{9}$ and the crosscap grid $\mathscr{C}_{9}$ in order from left to right. {Notice that both $\mathscr{H}_{9}$ and $\mathscr{C}_{9}$ contain two 
  $(18\times 9)$ grids as vertex-disjoint subgraphs (depicted in different colors). Moreover, $\mathscr{C}_{9}$ contains a $(18\times 18)$-grid as a spanning subgraph}.}
  \label{supplanted}
\end{figure}

This allows us to apply a recent  result of Gavoille and Hilaire \cite{gavoille2023minor} to show that there is some universal constant $c_{\Sigma}$ such that any of these grid-like graphs of order $c_{\Sigma}t^2$ contains every graph on at most $t$ vertices that embeds in $\Sigma$ as a minor.
This means, we leverage the theorem of Gavoille and Hilaire to prove that their minor-universal graph exists for any given ``signature'' (see \cite{gavoille2023minor} for a definition).
This is important for the following reason.
We know that any surface $\Sigma$ can be obtained from the sphere by adding $\mathsf{h}$ handles and $\mathsf{c}$ crosscaps where
\rev{13. P2, L22-23: The same $h$ is written in diﬀerent fonts.}
\ans{Fixed everywhere.}
\rev{8. P1, L26: The two $h$ in this line have diﬀerent fonts. }
\ans{Fixed everywhere.}
 the number $2\mathsf{h}+\mathsf{c}$ is precisely the Euler genus of $\Sigma.$
By the Classification Theorem of Surfaces, also known as Dyck's Theorem, \cite{Dyck1888Beitrage,Francis99ConwayZIP}, we know that every surface is homeomorphic to a surface that is obtained from the sphere by adding $\mathsf{h}$ handles and at \textsl{at most two} crosscaps.
Gavoille and Hilaire proved that there exists a minor-universal graph for orientable surfaces that is obtained by only adding handles to the sphere, and one for non-orientable surfaces obtained by only adding crosscaps to the sphere.
However, the proof of \cref{thm_GMST2local} by Kawarabayashi, Thomas, and Wollan, while guaranteeing a bound on their total number, extracts some number of handles and crosscaps in arbitrary order.

Let us give a quick definition of these grid-like graphs.
The  \emph{annulus grid} $\mathscr{A}_{k}$ is the $(4k,k)$-cylindrical grid\footnote{An $(n \times  m)$-cylindrical grid is a Cartesian product of a cycle on $n$ vertices and a path on $m$ vertices.} depicted in the left of \cref{supplanted}.
The \emph{handle grid} $\mathscr{H}_{k}$ (resp. \emph{crosscap grid} $\mathscr{C}_{k}$) is obtained adding in $\mathscr{A}_{k}$ edges as indicated in the middle (resp. right) part of \cref{supplanted}. We refer to the added edges as \emph{transactions} of the handle grid $\mathscr{H}_{k}$ or the crosscap grid $\mathscr{C}_{k}.$

Let now $\mathsf{h}\in\mathbb{N}$ and $\mathsf{c}\in[0,2].$
\rev{14. P3, L14: I do not understand the notion $[0,2]$ used here. It is a closed interval and contains real numbers. How can you have $c\in [0,2]$ copies if $c$ is not an integer? If $[0,2]$ refers something nonstandard, it should be deﬁned before it is used.}
\ans{We use $\{0,1,2\}$ here. Later in the the Definitions section we define  $[x,y]\coloneqq\{x,x+1,\ldots,y-1,y\}$.}
We define the graph $\mathscr{D}_{k}^{(\mathsf{h},\mathsf{c})}$ by taking one copy of $\mathscr{A}_{k},$ $\mathsf{h}$ copies of $\mathscr{H}_{k},$ and  $\mathsf{c}\in\red{\{0,1,2\}}$ copies of $\mathscr{C}_{k},$ then ``cut'' them along the dotted \textcolor{BostonUniversityRed}{red}
\rev{15. P3, L15: The word “red” should be written in black.}
\ans{We disagree. We prefer to use colors for better visualization.}
line, as in \cref{supplanted}, and join them together in the cyclic order $\mathscr{A}_{k},\mathscr{H}_{k},\ldots,\mathscr{H}_{k},\mathscr{C}_{k},\ldots,\mathscr{C}_{k},$ as \red{visualized} in \cref{perniciously} \red{(see \cref{sec_dycks} for the formal definitions)}.
\rev{16. P3, L16: A formal deﬁnition of $D^{(h,c)}_{k}$ $k$ should be included. Showing a ﬁgure for the case $k= 8$ is not suﬃcient.} 
\ans{We now explain that the  formal definition is given in Section 3. 
The full definition is technical, therefore for the purposes of the intro, we believe that Figure 2 provides good intuition on the definition of $\mathscr{D}^{({\sf h},{\sf c})}_{k}$.}
We call the graph $\mathscr{D}_{k}^{(\mathsf{h},\mathsf{c})}$ the \emph{Dyck-grid} of \emph{order} $k$ \emph{with} $\mathsf{h}$ \emph{handles and} $\mathsf{c}$ \emph{crosscaps}.
For technical reasons, we make the convention that $\mathscr{D}^{(-1,2)}_k=\mathscr{D}^{(0,0)}_k.$

Notice that, in a similar way, we may take any pair of non-negative integers $\mathsf{h}$ and $\mathsf{c}$ as input and concatenate one copy of $\mathscr{A}_{k}$ together with $\mathsf{h}$ copies of $\mathscr{H}_{k}$ and $\mathsf{c}$ copies of $\mathscr{C}_{k}$ in \textsl{any order}.
Any such graph yields a (parametric) representation of the surface obtained from the sphere by adding $\mathsf{h}$ handles and $\mathsf{c}$ crosscaps.

\begin{figure}[ht]
  \begin{center}
  \scalebox{0.85}{\includegraphics{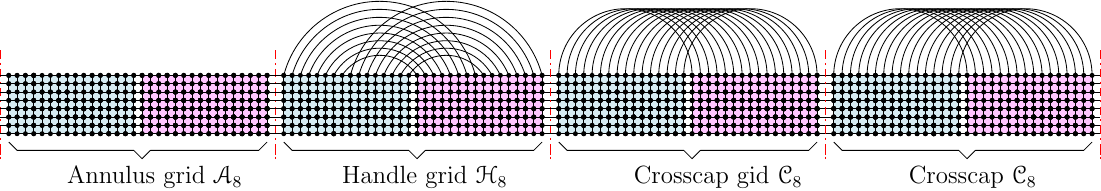}}
  \end{center}
  \caption{The Dyck-grid  of order $8$ with one handle and two crosscaps, i.e., the graph $\mathscr{D}_{8}^{1,2}.$ \red{The \textcolor{red}{red} dashed lines indicates that  ``same hight'' lefmost and rightmost vertices are adjacent (as it is the case in \cref{supplanted}).}}
  \label{perniciously}
\end{figure}%
\rev{17. P3, L17: Does $\Nbbb$ contains $0$? You let ${\sf h}\in \Nbbb$ but also allow ${\sf h}=0$ in $D^{(0,0)}$.}
\ans{We define $\Nbbb$ is the set of all non-negative integers. We added this definition in the beginning of the preliminaries section, together with the ones of $\Zbbb$ and $\Rbbb$.}

We prove that any such object with $\mathsf{h}$ handles and $\mathsf{c}$ crosscaps can be \red{transformed}
\rev{18. L-2 - P4, L2: First, it is unclear for me what “translate” means. Second, you wrote “We prove ...”, so you should indicate what theorem contains this statement and where the proof can be found.
} 
\ans{We now use the more adequate term ``transform'' for this. 
We also explain that we refer to  the \cref{diskussionen} where these 
transformations are explained.} 
 into any other such object with $\mathsf{h}'$ handles and $\mathsf{c}'$ crosscaps while preserving $2\mathsf{h}+\mathsf{c}=2\mathsf{h}'+\mathsf{c}'$ where $\mathsf{c}'\neq 0$ if and only if $\mathsf{c}\neq 0$ and $\mathsf{c}'$ is even if and only of $\mathsf{c}$ is even.
\red{This result is proven in \cref{sec_dycks} (\cref{diskussionen}) and it can be seen as the graph-theoretical analogue of Dyck's Theorem.}
\smallskip

Second, we use an intermediate structure from \cite{KawarabayashiTW20Quicklyexcluding} called a \emph{$\Sigma$-configuration}\footnote{See \cref{configuration} for the definition.} to prove that \cref{thm_GMST2local} may be extended to provide, as a third outcome, one of the Dyck-grids of large order representing some predetermined orientable surface $\Sigma_1$ or some predetermined non-orientable surface $\Sigma_2.$
This further allows us to adjust the second outcome in a way that only allows surfaces whose Euler genus is strictly smaller than the Euler genus of $\Sigma_1$ and $\Sigma_2.$

Using this result, we do not only prove \cref{thm_GMST3}, but we also prove a variant of \cref{thm_GMST2} which, instead of a clique, excludes exclusively two Dyck-grids, one for an orientable surface and one for a non-orientable one.
This theorem reads as follows.

\begin{proposition}\label{thm_GMST4}
There exists a function $f\colon\mathbb{N}^2\to\mathbb{N}$ with $f(k,\mathsf{h}_{1})\in 2^{2^{\mathcal{O}({\mathsf{h}_{1}})}\mathsf{poly}(k)}$ such that for every choice of non-negative integers $k,$ $\mathsf{h}_1,$ $\mathsf{h}_2,$ and $\mathsf{c}_2$ where $\mathsf{h}_2 < \mathsf{h}_1$ and $\mathsf{c}_2\in\{ 1,2\}$
\begin{enumerate}
  \item  $G$ contains $\mathscr{D}^{(\mathsf{h}_1,0)}_k$ or $\mathscr{D}^{(\mathsf{h}_2,\mathsf{c}_2)}_k$ as  a minor, 
\rev{19. P4, L13: The usage of “among” seems incorrect, and I do not understand this sentence. Do you mean both graphs are minors of $G$, or do you mean one of them is a minor of $G$?}
\ans{We rephrased the sentence to so make clear that  $G$ contains $\mathscr{D}^{(\mathsf{h}_1,0)}_k$ or $\mathscr{D}^{(\mathsf{h}_2,\mathsf{c}_2)}_k$ as  a minor. We now avoid the imprecise use of ``among''.}
  or
  \item $G$ can be obtained by means of clique sums from graphs that are $f(k,\mathsf{h}_{1})$-almost embeddable with breadth at most $f(k,\mathsf{h}_{1})$ in surfaces where neither $\mathscr{D}^{(\mathsf{h}_1,0)}_i$ nor $\mathscr{D}^{(\mathsf{h}_2,\mathsf{c}_2)}_i$ embeds for all $i\in\mathbb{N}.$ 
\end{enumerate}
\end{proposition}

This paper is part of a two-paper series where \cref{thm_GMST4} is the key tool towards the main theorem of a second paper in \cite{surfex2} providing a tight decomposition theorem 
for the exclusion of Dyck grids.

\paragraph{An algorithmic side of our results.}
In order to prove both \cref{thm_GMST3} and \cref{thm_GMST4} we prove a much more technical intermediate lemma in the form of \cref{efficiency}.
We believe that \cref{efficiency} will play a pivotal role in future applications.
Moreover, we leverage the constructive nature of the results of \cite{KawarabayashiTW20Quicklyexcluding} to provide an algorithmic version of \cref{efficiency}.
That is, we prove that any of the results in this paper are accompanied by polynomial-time algorithms.

An important part towards the algorithmic side of graph minors is a way to compute the clique-sum structure as guaranteed by \cref{thm_GMST1}, \cref{thm_GMST2}, and \cref{thm_GMST4}.
To make this possible one needs a way, given a witness of large treewidth in form of a so-called \textsl{well-liked set} $X,$ to find in polynomial time a large wall $W$ whose tangle agrees with the tangle of $X.$
While a lot of the necessary theoretical background towards this is already provided by \cite{KawarabayashiTW20Quicklyexcluding}, no explicit algorithm is given in their paper.
For this reason we present in \cref{biographies} a polynomial-time algorithm (given that the size of the desired wall is considered a constant) for this problem.

\section{Preliminaries}
\label{inphasized}

We denote by $\mathbb{Z}$ the set of integers. Also we use  {$\Nbbb$ for  the set of all non-negative integers},  and by $\mathbb{R}$ the set of reals.
Given two integers $a,b\in\mathbb{Z}$ we denote the set $\{z\in\mathbb{Z} \mid a\leq  z\leq  b\}$ by $[a,b].$
In case $a>b$ the set $[a,b]$ is empty. For an integer $p\geq  1,$ we set $[p]=[1,p]$ and $\mathbb{N}_{\geq  p}=\mathbb{N}\setminus [0,p-1].$
\smallskip

All graphs considered in this paper are undirected, finite, and without loops or multiple edges.
We use standard graph-theoretic notation, and we refer the reader to~\cite{diestel2016graph} for any undefined terminology. 
An \emph{annotated graph} is a pair $(G,X),$ where $X\subseteq V(G).$
Also we say hat $X\subseteq V(G)$ is \emph{connected} in $G$ if $G[X]$ is a connected graph.

\red{
\paragraph{Minors and minor models.} A \emph{minor model} of a graph $H$ in $G$ is a collection
$\Xcal=\{X_{v}\mid v\in V(H)\}$ is pairwise disjoint connected subsets of 
$G$ such that for every edge $vu\in E(H)$, $X_{v}\cup X_{u}$ is connected in $G$. A graph $H$ is a \emph{minor} of a graph $G$, denoted by $H≤G$,  if $G$ contains 
a minor model of $H$.
}

\paragraph{Tree-decompositions}
  Let $G$ be a graph. 
  A \emph{tree-decomposition} of $G$ is a tuple $\mathcal{T}=(T,\beta)$ where $T$ is a tree and $\beta\colon V(T)\to 2^{V(G)}$ is a function, whose images are called the \emph{bags} of $\mathcal{T},$ such that 
  \medskip
  \begin{enumerate}
  \item $\bigcup_{t\in V(T)} \beta (t)= V (G),$ 
\item for every $e\in E(G)$ there exists $t\in V(T)$ with $e\subseteq \beta(t),$ and 
\item for every $v\in V(G)$ the set $\{t\in V(T) \mid v\in \beta(t)\}$ induces a subtree of $T.$
\end{enumerate} 
\medskip
We refer to the vertices of $T$ as the \emph{nodes} of the tree-decomposition $\mathcal{T}.$
  
For each $t\in V(T),$ we define the \emph{\red{adhesion sets of} $t$}
\rev{20. P5, L4-6: I do not understand the deﬁnition of the adhesions of $t$. In L4, you deﬁne an adhesion of tto be set. In L6, you deﬁne the adhesion of $t$ to be an integer.}
 \ans{We now use the term  ``adhesion sets'' instead of ``adhesions''.}
as the sets in $$\{ \beta(t)\cap\beta(d) \mid d\text{ adjacent with } t\}$$ and the maximum size of them is called the  \emph{adhesion of $t$}.
  The \emph{adhesion} of $\mathcal{T}$ is the maximum adhesion of a node of $\mathcal{T}.$
  The \emph{torso} of $\mathcal{T}$ on a node $t$ is the graph, denoted by $G_{t},$ obtained by adding edges between every pair of vertices of $\beta(t)$ which belongs to the \red{same  adhesion set} of $t.$  
  The \emph{width} of a $(T,\beta)$ is the value $\max_{t\in V(T)}|\beta(t)|-1.$
  The \emph{treewidth} of $G,$ denoted by $\mathsf{tw}(G),$ is the minimum width over all tree-decompositions of $G.$ 
 
\section{Dyck-walls: Graphs representing surfaces} 
\label{sec_dycks}
In the following we introduce several classes of grid-like graphs that capture the behaviour of surfaces.
These graphs will acts as patterns that allow us to formalise what it means to exclude a surface from a graph.
 
Let $m,n$ be positive integers where $m,n\geq 3.$
Given an $(m,n)$-cylindrical grid, let $C_1,\dots,C_m$ be the $m$ disjoint cycles numbered in such a way that $C_{i}$ separates $C_{i-1}$ from $C_{i+1}$ for all $i\in[2,m-1].$
Moreover, let us number the vertices of $C_i$ as $v^i_1,\dots,v^i_{n}$ such that $v^1_jv^2_j\dots v^{m-1}_jv^m_j$ induces is a path for every $j\in[n].$
  
Consider the $(m,n\cdot 4m)$-cylindrical grid.
By \emph{adding a handle at position $i\in[n]$} we mean adding the edges 
\begin{align*}
  \{ v^1_{4m(i-1)+j}v^1_{4m(i-1)+3m-j+1} \mid j\in[m] \}\cup\{ v^1_{4m(i-1)+m+j}v^1_{4m(i-1)+4m\red{-j+1}} \mid j\in[m] \}.
\end{align*}
We call the paths \red{consisting of these edges} \emph{handle paths}.

By \emph{adding a crosscap at position $i\in[n]$} we mean adding the edges
\rev{21. P5, L18: I guess that $v_{4m(i-1)+4 m+j-1}$ should be $v_{4m(i-1)+4 m-j+1}$.}
\ans{Indeed, we changed it.}
\begin{align*}
  \{ v^1_{4m(i-1)+j}v^1_{4m(i-1)+2m+j} \mid j\in[2m] \}.
\end{align*}
We call the paths 
\red{consisting of these edges} \emph{crosscap paths}.

\paragraph{Mixed surface grids.}
A \emph{mixed surface grid of order $k$ with $\mathsf{h}$ handles}  and $\mathsf{c}$ crosscaps is a graph $H^k_{I_{\text{crscp}},I_{\text{hndl}}}$ constructed as follows.
Let $I_{\text{hndl}},I_{\text{crscp}},$
\rev{24. P5, L-11 - L-8: You mentioned that “a mixed surface grid ... is a graph constructed as follows”. But this paragraph seems no conclusion. Do you mean that $H_{I_c,I_h}$ is a mixed surface grid?}
\ans{Fixed.}
 be a partition of the set $[2,\mathsf{h}+\mathsf{c}+1]$ into two sets where $|I_{\text{hndl}}|=\mathsf{h}$ and $|I_{\text{crscp}}|=\mathsf{c}.$
 \rev{22. P5, L-10: The notation $I_c$ and $I_h$ are problematic. $c$ and $h$ are integers and they can be equal. In that case $I_c$ and $I_h$ denote the same set.}
 \ans{We changed the names of these sets to $I_{\text{hndl}}$ and $I_{\text{crscp}}$ to avoid the issue.}
Consider the $(k,4(\mathsf{h}+\mathsf{c}+1)\cdot k)$-cylindrical grid $H$ and let $H^k_{I_{\text{hndl}},I_{\text{crscp}}}$
\rev{23. P5, L-9: $k$ should appear as a part of $H_{I_c,I_h}$ because it is a variable, especially you use it in Definition 3.1 to define $D^{h,c}_{k}$.}
\ans{Added the $k$.}  
be obtained from $H$ by adding a handle at every position $i\in I_{\text{hndl}}$  and adding a crosscap at every position $i\in I_{\text{crscp}}$.
\rev{25. P5, L-6: When you define “adding a handle” or “adding a crosscap”, you allow subdividing added edges arbitrarily many times. So there are inﬁnitely many possible $H_{[2,h+1],[h+2,c+h+1]}$ based on the number of times you subdivide those edges. But $D^{h,c}_k$ must be defined to be a unique graph.}
\ans{We removed the option to subdivide the additional edges in the definition of mixed surface grids. This should take care of the issue.}  

\paragraph{Parametric graphs.}

\red{A \emph{parametric graph} is a sequence of graphs $\mathscr{G} = \langle \mathscr{G}_t \rangle_{t \in \mathbb{N}}$. In this paper we always additionally assume that, for every $i \leq j$, $\mathscr{G}_i$ is a minor of $\mathscr{G}_j$. Let  $\mathscr{G}=\langle \mathscr{G}_t\rangle_{t\in\mathbb{N}}$ and $\mathscr{G}'=\langle \mathscr{G}_t'\rangle_{t\in\mathbb{N}}$ be two parametric graphs.
We write $\mathscr{G} \lesssim \mathscr{G}',$ if there is some function $f\colon \Nbbb\to\Nbbb$ such that $\mathscr{G}_t≤ \mathscr{G}_{f(t)}'$ for all $t\in\mathbb{N}.$
If $\mathscr{G} \lesssim \mathscr{G}'$ and $\mathscr{G}' \lesssim \mathscr{G}$ we say that $\mathscr{G}$ and $\mathscr{G}'$ are \emph{equivalent}, denoted by $\mathscr{G}\equiv \mathscr{G}'.$ Also if the both functions in the definition of $\equiv$ are both linear, then we say that $\mathscr{G}$ and $\mathscr{G}'$ are \emph{linearly equivalent}.
}
\rev{26. P5, L-5: “Parametric graph” should be deﬁned.}
\ans{A paragraph about parametric graphs has been added before Definition 3.1. }

\begin{definition}[Dyck-grid]  
\label{def_dgrid}
Let $\mathsf{h},\mathsf{c}\in \mathbb{N}.$
The \emph{$(\red{\mathsf{h},\mathsf{c}})$-Dyck-grid} of order $k\in \mathbb{N}$  is the \red{parametric graph 
$\mathscr{D}^{(\mathsf{h},\mathsf{c})}=\langle \mathscr{D}_{k}^{(\mathsf{h},\mathsf{c})}\rangle_{\in\mathbb{N}}$ where}

$$\mathscr{D}_{k}^{(\mathsf{h},\mathsf{c})}=H^k_{[2,\mathsf{h}+1],[\mathsf{h}+2,\mathsf{h}+\mathsf{c}+1]}.$$
We denote the corresponding parametric graph by $\mathscr{D}^{(\mathsf{h},\mathsf{c})}\coloneqq\langle \mathscr{D}^{(\mathsf{h},\mathsf{c})}_k\rangle_{k\in\mathbb{N}}.$ 
A Dyck-grid $\mathscr{D}^{(\mathsf{h},\mathsf{c})}_{k}$ is said to be \emph{orientable} if $c=0$ and \emph{non-orientable} otherwise. We also make the convention that  $\mathscr{D}^{(-1,2)}=\mathscr{D}^{(0,0)}.$
\rev{27. P5, L-4: You deﬁne $D^{-1,2}=D^{0,0}$, However, based on your deﬁnition, the left side is non-orientable but the right side is orientable.}
\ans{There was some confusion on the order handles and crosscaps appear. In the new version we make the convention that handles appear first.}

The \emph{Euler-genus} of the Dyck-grid $\mathscr{D}^{(\mathsf{h},\mathsf{c})}_k$ is $2\mathsf{h}+\mathsf{c}.$
We refer to the cycles $C_{k-b-1},\ldots,C_{k}$ as the $b$ \emph{outermost} cycles of $\mathscr{D}_{k}^{(\mathsf{h},\mathsf{c})}.$
\rev{29. P5, L-1: I do not see why Figure 2 is a drawing of $\mathscr{D}^{1,2}_{8}$ The rows in Figure 2 are paths. 
They should be cycles based on your definition of $D^{1,2}_{8}$}
\ans{Notice that handles and crosscaps are created by adding ``parallel'' edges on the top of a $(m,n\cdot 4m)$-cylindrical grid. In  \cref{perniciously} we present this cylindrical grid ``unfolded'' and we  add in its leftmost and rightmost sides a red dashed line to indicate that the corresponding vertices are connected by edges as it is already indicated by  \cref{supplanted} 
where these edges are drawn (again crossing the dashed red line). We added some more explanation on this in the caption of \cref{perniciously}.}
\end{definition}

Notice that $\mathscr{D}^{0,0}\equiv \mathscr{A}$
\rev{28. P5, L-1: What are $\Acal,\Hcal$ and  $\Ccal$? What does \~{}\  mean?}
\ans{We added a paragraph before \cref{def_dgrid} defining the basic concepts of parametric graphs. We now define and use $\equiv$ instead of $\sim$.}  
and, also, 
$\mathscr{H}\equiv \mathscr{D}^{1,0}$ and  $\mathscr{C}\equiv\mathscr{D}^{0,1}$ \red{(the parametric graphs $\mathscr{A}$, $\mathscr{H}$, and $\mathscr{C}$ are defined in \cref{subsec_contribution}, see \cref{supplanted})}. For a drawing of $\mathscr{D}_{8}^{1,2}$ see of \cref{perniciously}. 
We say that a surface $\Sigma$ \emph{corresponds} to a Dyck-grid $\mathscr{D}^{(\mathsf{h},\mathsf{c})}_k,$ and vice versa, if $\Sigma$ is homeomorphic to a surface that can be obtained by adding $\mathsf{h}$ handles and $\mathsf{c}$ crosscaps to a sphere.
\red{In such cases we sometimes denote $\mathscr{D}^{(\mathsf{h},\mathsf{c})}$ as $\mathscr{D}^{\Sigma}.$}
\medskip
\rev{30. P6, L2: What is $D^{\Sigma}$?}
\ans{We now explain the notation $D^{\Sigma}$.}

The main reason most of the structural theory talks about walls rather than grids is that walls are sub-cubic graphs and as such, if they are contained as minors in a graph, one can find a subdivision as a subgraph.
For a similar reason we define Dyck-walls.
Let $c\in[0,2],$ $\mathsf{h},$ and $t$ be non-negative integers.
The \emph{elementary $(\red{\mathsf{h},\mathsf{c}};t)$-Dyck-wall} is obtained from $\mathscr{D}^{(\mathsf{h},\mathsf{c})}_{2t}$ by deleting the cycles $C_{t+1},\dots,C_{2t}$ and by deleting the edge $v^i_jv^{i+1}_j$ for every odd $i\in[t-1]$ and every odd $j\in[8t]$ and for every even $i\in[t-1]$ and even $j\in[8t].$
Moreover, for each handle and crosscap that was added to create $\mathscr{D}^{(\mathsf{h},\mathsf{c})}_{2t}$ delete every (subdivided) edge \red{added this way and incident to} $v^1_\ell$ where $\ell$ is even.
\rev{31. P6, L8: I do not understand “delete every edge starting on $v_\ell^1$”. Do you want to make $v_\ell^1$ an isolated vertex? Or do you only want to delete edges incident to it but not in the cylinder wall? In the definition of $D^{c,h}_{k}$, edges not in the cylinder wall form paths. Do you only delete the edge in those paths incident to $v_\ell^1$, or do you delete the entire path?}
\ans{We now added some clarification on this.}
That is, we delete every second edge.
An \emph{$(\red{\mathsf{h},\mathsf{c}};t)$-Dyck-wall} is a subdivision of the elementary $(\red{\mathsf{h},\mathsf{c}};t)$-Dyck-wall.

Consider an embedding of a $(\red{\mathsf{h},\mathsf{c}};t)$-Dyck-wall $D$ in a surface $\Sigma$ of Euler-genus $2\mathsf{h}+\mathsf{c}$ \red{where\footnote{Actually this embedding can be done unique by picking $t$ big enough as a function of $\mathsf{h}$ and $\mathsf{c}$ because do the results in \cite{SeymourT96Uniqueness}.} $C_t$ bounds a disk in which no vertex of $D$ is drawn.}

\rev{32. P6, L11: “Then $C_{t}$ bounds a disk” requires a proof. First, it is false unless you assume that tis suﬃciently large. Second, you might need to cite a result stating that the embedding of a highly connected graph is unique. On the other hand, you can skip the proof if you change “Then” by choosing an embedding that satisfies this property.}
\ans{We implemented the suggested fix. We also added a footnote to  \green{\sf [P. D. Seymour \& Robin Thomas titled “Uniqueness of Highly Representative Surface Embeddings” (Journal of Graph Theory 23:4 (1996), pp 337-349) proves a uniqueness result for embeddings with large representativity.]}.}
We call this disk the \emph{simple face} of $D.$
Similarly, there exists a unique face which is incident to $C_1$ and every handle and every crosscap of $D.$
We call this face the \emph{exceptional face} of $D.$
\medskip

The next theorem is the main result of this section. It formalizes that, \red{in terms of  the equivalence  relation ``$\equiv$'' between parametric graphs,}  there is no
\rev{33. P6, L15: “not” should be “no”. The use of “asymptotically” is incorrect. “Asymptotically” is used to described a phenomenon when certain parameter tends to inﬁnite or a certain value. There is no such ingredient in Theorem 3.2.  I guess you mean “approximately”, not “asymptotically”.}
\ans{We now gave the definition of  equivalence between parametric graphs that  
precisely provides the equivalence that we prove here.}
difference between surface grids and Dyck-grids.  


\begin{theorem}\label{diskussionen}
Let $g,\mathsf{h},\mathsf{c}$ be integers such that $g=2\mathsf{h}+\mathsf{c}>0$ and ${\sf c}>0.$
Let ${\sf h}_{0}=\lceil\frac{c}{2}\rceil-1$ and ${\sf c}_{0}={\sf c}-2{\sf h}_{0}$.
\rev{34. P6, L17: What is it? I guess that you just mean $h_0=\lceil \frac{c}{2} \rceil -1$ and $c_0=c-2h_0$.}
\ans{Fixed.}
Then, \red{for every integer $k > 0$, the following hold:}
\rev{35. P6, L18: Then for every integer $k > 0$, the following ...}  
\ans{Fixed.}

\begin{itemize}
\item every mixed surface grid of order $162^{2g}k$ with $\mathsf{h}$ handles and $\mathsf{c}$ crosscaps  contains $\mathscr{D}^{(\mathsf{h}+\mathsf{h}_{0},\mathsf{c}_0)}_k$ as a minor, 
\item  $\mathscr{D}^{(\mathsf{h}+\mathsf{h}_{0},\mathsf{c}_0)}_{162^{2g}k}$ contains \red{every mixed surface} grid of order $k$ with $\mathsf{h}$ handles and $\mathsf{c}$ crosscaps as a minor.
\end{itemize}
\end{theorem} 

\begin{figure}[ht]
  \begin{center}
  \scalebox{.75}{\includegraphics{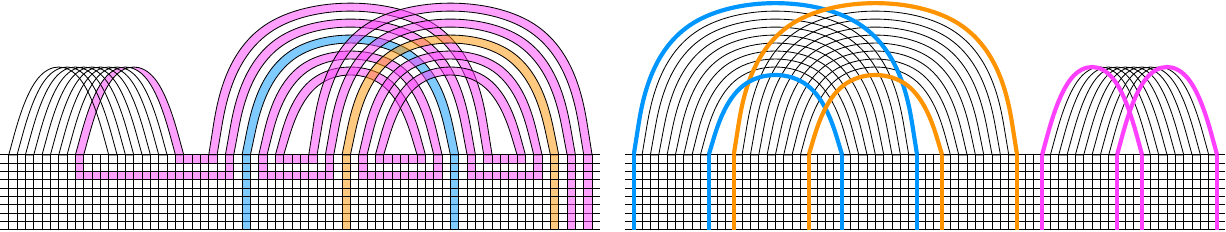}}
  \end{center}
  \caption{Swapping the position of a crosscap and a handle. The colors show how the handle and the crosscap on the right are routed through the crosscap and the handle on the left.}
  \label{dignitaries}
\end{figure}
\rev{36.  I do not read the proof of Theorem 3.2 because the theorem looks obviously true, except the correctness of the quantity $162^{2g}k$ should be carefully verified. The authors should verify the correctness on their own.}
\ans{Checked!}
At this point, we wish to mention that, for orientable Dyck-grids, the following observation can be made.
This observation reflects the fact that a non-orientable surface can never be contained in an orientable one.

\begin{observation}\label{miserliness}
For every $t\in\mathbb{N},$ $t\geq  2,$ and every $k\in\mathbb{N}$ it holds that $\mathscr{D}^{(0,1)}_t\not\leq  \mathscr{D}^{(t-1,0)}_k.$
\rev{37. P6, L-10: What do $\leq$ and  $\not\leq$ mean?}
\ans{In the preliminaries we added the definition of ``minor model'', ``minor'', and the symbol ``$≤$'' for the minor relation.}
\end{observation}

The rest of this section is dedicated to prove \cref{diskussionen}.
In the following, we will say that a handle transaction is of \emph{order $k$} if each of the two transactions involved in the handle are of order $k.$

The next lemma  allows us to swap a handle with a neighboring crosscap {within}  a mixed surface grid by sacrificing a portion of its order.

\begin{lemma}\label{upanishads}
Let $k,\mathsf{h},\mathsf{c}\geq  1$ be integers.
Let $H'$ be the $(9k,4(\mathsf{h}+\mathsf{c}+1)\cdot 9k)$-cylindrical grid, let $I'_{\text{hndl}},I'_{\text{crscp}}$ be a partition of $[2,\mathsf{h}+\mathsf{c}+1]$ where $|I'_{\text{crscp}}|=\mathsf{c}$ and $|I'_{\text{hndl}}|=\mathsf{h}.$
Moreover, let $H$ be the $(k,4(\mathsf{h}+\mathsf{c}+1)\cdot k)$-cylindrical grid.
Finally, let $i\in I'_{\text{crscp}}$ and $i+1\in I'_{\text{hndl}}$ as well as $j\in I'_{\text{hndl}}$ and $j+1\in I'_\mathsf{c}$ and set 
\begin{eqnarray*}
 I_{\text{hndl}}\coloneqq (I'_{\text{hndl}}\setminus \{ i+1\})\cup\{ i\} & I_{\text{crscp}}\coloneqq (I'_{\text{crscp}}\setminus\{ i\})\cup \{ i+1\}, \\
J_{\text{hndl}} \coloneqq (I'_{\text{hndl}}\setminus\{ j\})\cup\{ j+1\} & J_{\text{crscp}}\coloneqq (I'_{\text{crscp}}\setminus\{ j+1\})\cup\{ j\}
\end{eqnarray*}
Then \red{the mixed surface} grid 
\rev{39. P11, L-6: “the mixed surface ...” should be “a mixed surface ...”. Such graphs are not unique based on your definition.}
\ans{In Theorem 3.2 it should say ``every'' mixed surface grid, neither ''a'' nor ''the''. We fixed it.
In general, we have changed the definition to no longer allow for subdivided edges so now, given the correct sets of indices and the order, the graph $H^k_{I,I'}$ is unique.}
${H'}^{k}_{I'_{\text{hndl}},I'_{\text{crscp}}}$ of order $5k$ with $\mathsf{h}$ handles and $\mathsf{c}$ crosscaps  contains as a minor 
\begin{itemize}
\item \red{the mixed surface} grid $H^k_{I_{\text{hndl}},I_{\text{crscp}}}$ of order $k$ with $\mathsf{h}$ handles and $\mathsf{c}$ crosscaps  and 
\item \red{the mixed surface} grid $H^{k}_{J_{\text{crscp}},J_{\text{hndl}}}$ of order $k$ with $\mathsf{h}$ handles and $\mathsf{c}$ crosscaps.
\end{itemize}
\end{lemma}

\begin{proof}
To see that the claim of this lemma is true consider \cref{dignitaries}.
Since the construction is completely symmetric it suffices to consider the case where $i+1\in I'_{\text{hndl}}$ and  $i\in I'_{\text{crscp}}$ and as depicted in the figure.
Notice that the handle transaction at $i+1$ partitions the set $[4\cdot 9k\cdot i+1,8\cdot 9k\cdot i]$ into four intervals, namely $[4\cdot 9k\cdot i+1,4\cdot 9k\cdot i+9k],$ $[4\cdot 9k\cdot i+9k+1,4\cdot 9k\cdot i+2\cdot 9k],$ $[4\cdot 9k\cdot i+2\cdot 9k+1,4\cdot 9k\cdot i+3\cdot 9k],$ and $[4\cdot 9k\cdot i+3\cdot 9k+1,4\cdot 9k\cdot (i+1)].$
The first and third interval form the endpoints of one of the two transactions that form the handle, the second transaction involved in the handle has its endpoints in the second and fourth interval.
We fix the two intervals $[4\cdot 9k\cdot i+32k+1,4\cdot 9k\cdot i+34k]$ and $[4\cdot 9k\cdot i+34k+1,4\cdot 9k\cdot (i+1)]$ to host the starting and endpoints of a crosscap transaction of order $2k$ which is necessary for a mixed surface grid of order $k.$
As depicted in \cref{dignitaries}, we may route the two pairs of $2k$ disjoint paths, from each of the two terminal intervals, through the handle.
By doing so, each handle is traversed two times per pair which takes up $8k$ paths of each of the two transactions of the handle in total.
Thus, as depicted in \cref{dignitaries}, for each of the two transactions involved in the handle we may keep $k$ paths that allow us to maintain a handle transaction of order $k.$
Moreover, the may route the crosscap in such a way that it avoids the cycles $C_{8k+1},\dots,C_{9k}$ since we are only routing $8k$ paths in total.
By combining these cycles with $4k$ paths from the grid-structure at position $0,$ a sub-handle, sub\red{-}crosscap\ans{added a missing dash} of each other handle or crosscap transaction of order $k$ each, and the newly routed handle and crosscap transactions, we obtain the desired mixed surface grid of order $k.$
\end{proof}

\subsection{A graph-minor analogue of Dyck’s lemma}
The next step is to convert three consecutive crosscaps into a handle followed by a crosscap.
Moreover, we show that the reversal of this exchange is possible.
The numbers in the following lemma are not necessarily optimal, they are adjusted to make for a nicer presentation.

We remark that \cref{classification} can be seen as 
a graph-theoretical analogue of  Dyck's theorem, stated in terms of graph minors.

\begin{figure}[ht]
  \begin{center}
  \scalebox{.65}{\includegraphics{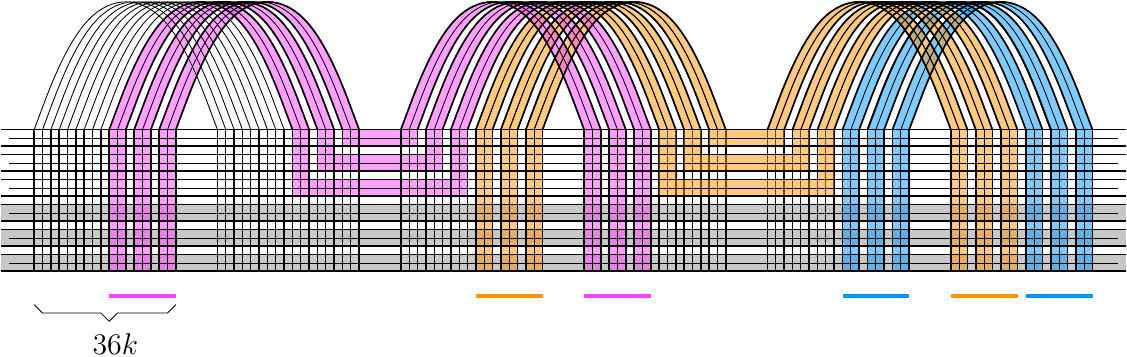}}
  \end{center}
  \caption{Three crosscaps in a row. }
  \label{functionaries}
\end{figure}

\begin{lemma}\label{classification}
Let $k,\mathsf{h},\mathsf{c}\geq  0$ be integers such that $k\geq  1$ and $c\geq  3.$
Let $H'$ be the $(18k,4(\mathsf{h}+\mathsf{c}+1)\cdot 18k)$-cylindrical grid and $H''$ be the $(18k,4(c-2+\mathsf{h}+1+1)\cdot 18k)$-cylindrical grid.
Let $I'_{\text{hndl}},I'_{\text{crscp}}$ be a partition of $[2,\mathsf{h}+\mathsf{c}+1]$ where $|I'_{\text{crscp}}|=c$ and $|I'_{\text{hndl}}|=h$ and there exists some integer $i\in I'_{\text{crscp}}$ such that $i+1,i+2\in I'_{\text{crscp}},$ moreover, let $I''_{\mathsf{c}},I''_{\mathsf{h}}$ be a partition of $[2,\mathsf{h}+\mathsf{c}]$ such that $i\in I''_{\mathsf{h}}$ and $i+1\in I''_{\mathsf{c}}.$
Finally, let $H^1$ be the $(k,4(\mathsf{h}+\mathsf{c})\cdot k)$-cylindrical grid and $H^2$ be the $(k,4(\mathsf{h}+\mathsf{c}+1)\cdot k)$-cylindrical grid.

Then the following hold:
\begin{enumerate}
\item\label{combination}  \red{the mixed surface} grid ${H'}^{18k}_{I'_{\text{hndl}},I'_{\text{crscp}}}$ of order $18k$ with $\mathsf{h}$ handles and $\mathsf{c}$ crosscaps  contains \red{the mixed surface} grid ${H^1}^{k}_{I''_{\mathsf{c}},I''_{\mathsf{h}}}$ of order $k$ with  $\mathsf{h}+1$ handles  and $\mathsf{c}-2$ crosscaps as a minor, and 
\item\label{undertaking} \red{the mixed surface} grid ${H''}^{18k}_{I''_{\mathsf{c}},I''_{\mathsf{h}}}$ of order $18k$ with $\mathsf{h}+1$ handles and $\mathsf{c}-2$ crosscaps contains \red{the mixed surface} grid ${H^{2}}^{k}_{I'_{\text{hndl}},I'_{\text{crscp}}}$ of order $k$ with $\mathsf{h}$ handles and $\mathsf{c}$ crosscaps  as a minor. 
\end{enumerate}
\end{lemma}

\begin{proof}
We divide the proof of this lemma into two steps for the proof of \ref{combination}
(visualized in \cref{functionaries}, \cref{resentative}, and \cref{instinctive})
and one for the proof of the reverse direction \ref{undertaking}
(visualized in \cref{thematically}, \cref{interference}, and \cref{celebrities}).\medskip

\begin{figure}[ht]
  \begin{center}
  \scalebox{.65}{\includegraphics{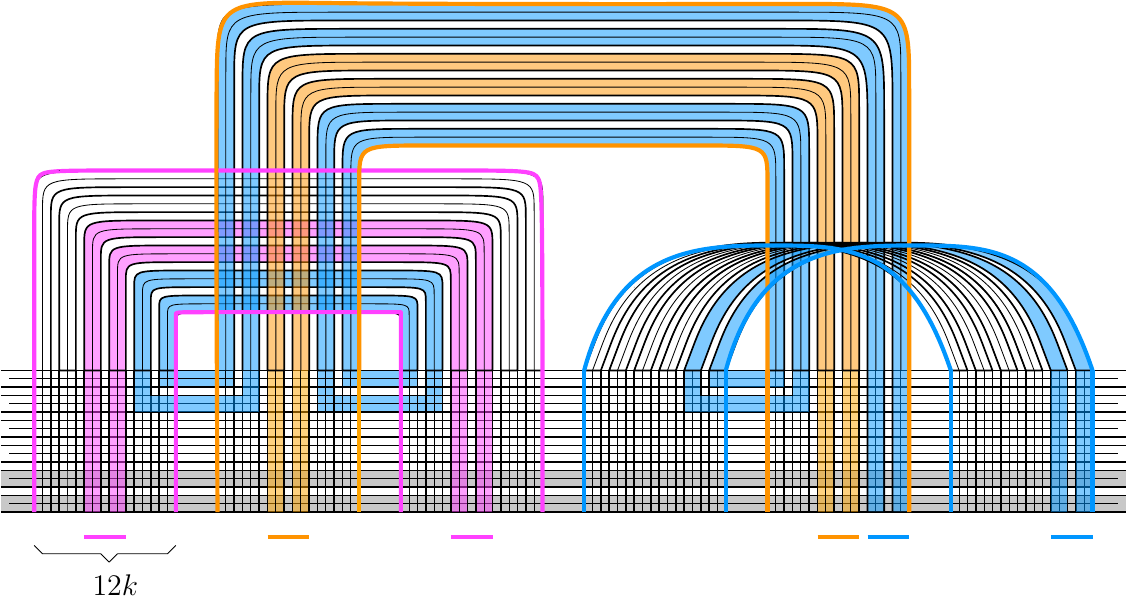}}
  \end{center}
  \caption{The result of the construction from \cref{functionaries}.}
  \label{resentative}
\end{figure}

Proof of \ref{combination}.
We begin by routing three sets of paths through the three consecutive crosscaps.
Notice that each of the three crosscap transactions contains a total of $2\cdot 18k=36k$ paths, however, each of our path-bundles will only contain $12k$ paths.
This allows us to avoid the cycles $C_{12k+1},\dots,C_{18k}$ which we need to form the annulus of the resulting mixed surface grid.
The first bundle of paths starts in the crosscap at position $i$ on the cycle $C_{18k},$ traverses this crosscap, traverses the crosscap at position $i+1$ and then ends again on the cycle $C_{18k}.$
The second bundle starts left (with respect to the ordering of the vertices on $C_{18k}$ as induced by the definition of mixed surface grids) of the endpoints of the first bundle.
It traverses the crosscap at position $i+1$ and then the crosscap at position $i+2.$
Finally, the third bundle starts left of the end of the second bundle in the crosscap at position $i+2,$ traverses only this one crosscap, and ends on the right of the second bundle.
See \cref{functionaries} for an illustration of this first construction.

Now keep these three bundles, each of order $12k,$ keep $12k$ paths of each other crosscap transaction and keep $6k$ paths of each of the two transactions involved in every other handle transaction.
Moreover, keep $12k$ paths from the grid structure at position $0$ together with the cycles $C_{12k+1},\dots,C_{18k}$
The resulting graph almost resembles a mixed surface grid with the exception of positions $i,$ $i+1,$ and $i+2,$ here we have two planar and one crosscap transaction intertwined.
See \cref{resentative} for an illustration of the situation.

\begin{figure}[ht]
  \begin{center}
  \scalebox{.65}{\includegraphics{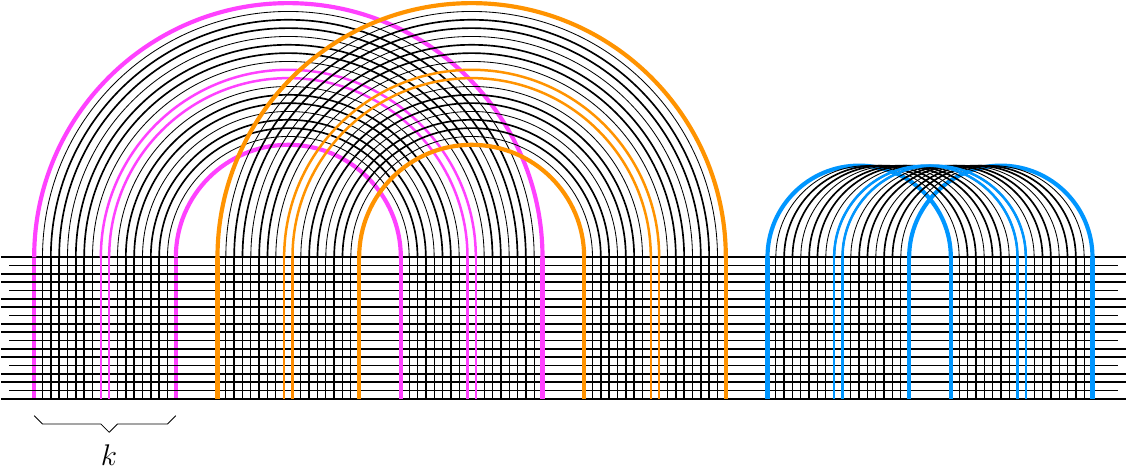}}
  \end{center}
  \caption{The result of the construction from \cref{resentative}. The colors indicate the different bundles of paths used to route the transactions.}
  \label{instinctive}
\end{figure}

The final step is to untangle the crosscap at position $i+2$ from the newly created handle.
Notice that ``position $i+2$'' refers only to the bundle of paths forming the crosscap transactions as we are not in a mixed surface grid any more.
Indeed, by merging the other two crosscaps into a single handle, we have lost a position.
Similar to the proof of \cref{upanishads} we fix two disjoint intervals of endpoints, one {within} the right most part of the crosscap transaction on the cycle $C_{18k},$ and the other in the right most part of the middle transaction that is part of the handle.
Both intervals are of order $2k.$
We then route a crosscap transaction from the left of these intervals to the right.
As depicted in \cref{resentative} we may traverse the handle transaction which allows us to traverse the crosscap exactly once.
Hence, the resulting transaction is, indeed, a crosscap.
Since each of the two transactions involved in the handle are of order $12k,$ and we only spend $4k$ of their paths to route the crosscap, there are enough paths left to choose $k$ paths each in order to form a handle transaction of order $k.$
By restricting all other handles and crosscaps accordingly and keeping the cycles $C_{17k+1},\dots,C_{18k},$ we obtain a mixed surface grid of order $k$ as desired.
See \cref{instinctive} for an illustration.\medskip

\begin{figure}[ht]
  \begin{center}
  \scalebox{.65}{\includegraphics{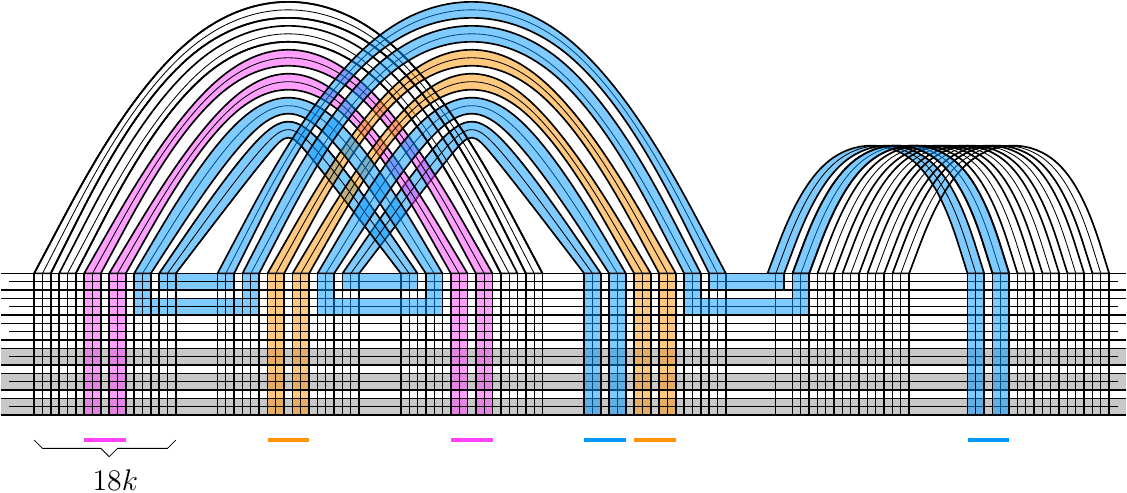}}
  \end{center}
  \caption{A handle and a crosscap in a row together with the routing for the crosscap that allows to recreate the situation from \cref{resentative}.}
  \label{thematically}
\end{figure}

Proof of \ref{undertaking}.
The reverse transformation works along similar lines but is not completely analogous.
For the sake of completeness we include illustrations of the three steps that decompose a handle and a crosscap into three distinct crosscaps.

The first step is to recreate the entangled situation from \cref{resentative} by routing the crosscap backwards through the handle as depicted in \cref{thematically}.
By doing so we divide the order of the original crosscap transaction by three as we need to route $6k$ paths which will later be divided into three disjoint crosscaps, but we also need to keep $6k$ paths from each of the two transactions of order $18k$ that form the handle.
As we only route $6k$ paths in this step, we are able to keep the cycles $C_{6k+1},\dots,C_{18k}$ for the next step.

\begin{figure}[ht]
  \begin{center}
  \scalebox{.65}{\includegraphics{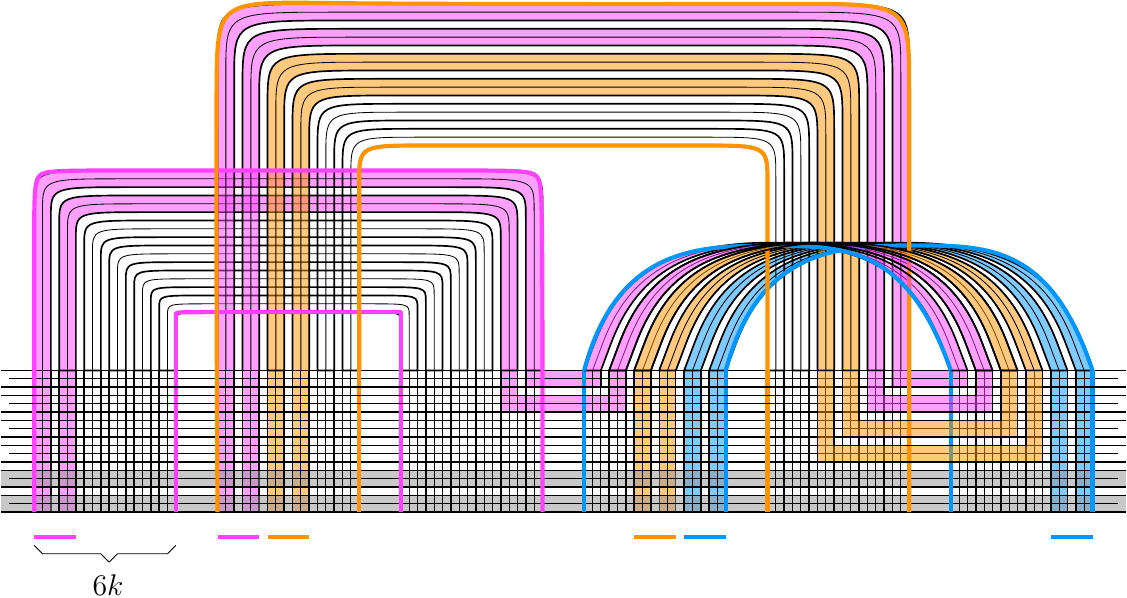}}
  \end{center}
  \caption{The result of the construction from \cref{thematically}.}
  \label{interference}
\end{figure}

With this, we have now established a situation similar to the one from \cref{resentative} as depicted in \cref{interference}.
We may now split the crosscap transaction of order $6k$ into three, each of them of order $2k.$
Finally, we make use of the entangled handle to create three distinct crosscap transactions of order $2k$ as depicted in \cref{celebrities}.

Since we route $4k$ paths through the annulus, we may, in particular, avoid the cycles $C_{17k+1},\dots,C_{18k}$ and thus, by restricting the grid structure at position $0$ and all other handle and crosscap transactions accordingly, we obtain the desired mixed surface grid and our proof is complete.
\end{proof}

\begin{figure}[ht]
  \begin{center}
  \scalebox{.65}{\includegraphics{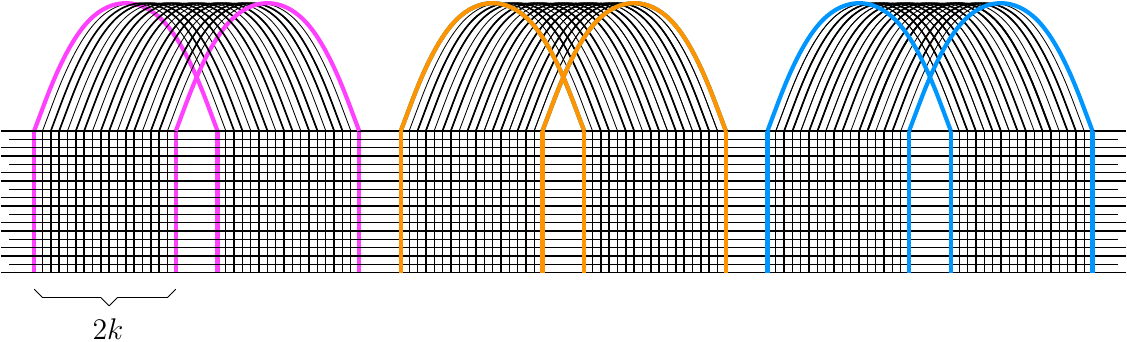}}
  \end{center}
  \caption{The result of the construction from \cref{interference}. The colors indicate the bundles of paths used to create the three consecutive crosscaps.}
  \label{celebrities}
\end{figure}

\subsection{Proof of \cref{diskussionen}}

We are finally ready for the proof of \cref{diskussionen}.

\begin{proof}[Proof of \cref{diskussionen}]
Let $H$ be a mixed surface grid of order $162^{\mathsf{h}_0+ \mathsf{c}+ \mathsf{h}}\cdot k$ with $\mathsf{h}$ handles and $\mathsf{c}$ crosscaps.

For each of the $\mathsf{c}$ crosscaps we need to call \cref{upanishads} at most $\mathsf{h}$ times to achieve a situation where all handles are consecutive and thus obtain a surface grid.
That is, if we aim for a surface grid of order $w$ with $\mathsf{h}$ handles and $\mathsf{c}$ crosscaps, we may start out with a mixed surface grid of order $9^{\mathsf{h}+\mathsf{c}}\cdot w.$
Moreover, this is enough to move any crosscap essentially anywhere along the cycle $C_1.$
That means that also the surface grid of order $9^{\mathsf{h}+\mathsf{c}}\cdot w$ contains every mixed surface grid of order $w$ with $\mathsf{h}$ handles and $\mathsf{c}$ crosscaps as a minor.

In particular, if $H'$ is the surface grid of order $18^{\mathsf{h}_0}\cdot k,$ then $H'$ is a minor of $H$ by the reasoning above.

Now, to obtain a Dyck-grid of order $w$ with $\mathsf{h}+\mathsf{h}_{0}$ handles and $\mathsf{c}_0$ crosscaps from a surface grid with $h$ handles and $\mathsf{c}$ crosscaps we may proceed as follows:
Choose three consecutive crosscaps at positions $i,$ $i+1,$ and $i+2$ such that the transaction at position $i-1$ is either a handle, or $i-1=0.$
Then apply \cref{classification} to these three crosscaps.
The resulting minor we find is still a surface grid, but now on $\mathsf{h}+1$ handles and $\mathsf{c}-2$ crosscaps and its order has decreased by a factor of $18.$
As we have to apply this exactly $\mathsf{h}_{0}$ times to reduce the number of crosscaps below three, by starting out with a surface grid of order $18^{\mathsf{h}_{0}}\cdot w$ with $\mathsf{h}$ handles and $\mathsf{c}$ crosscaps, we obtain the Dyck-grid of order $w$ and Euler-genus $2(\mathsf{h}+\mathsf{h}_{0})+\mathsf{c}_0=2\mathsf{h}+\mathsf{c}$ as a minor.
Notice that \cref{classification} states that reversing this operation also costs a factor of $18,$ that means that the (orientable in case $\mathsf{c}=0$ and otherwise non-orientable) Dyck-grid of order $18^{\mathsf{h}_{0}}\cdot w$ contains the surface grid of order $w$ with $h$ handles and $\mathsf{c}$ crosscaps as a minor.
In fact, for any $\mathsf{h}'\leq  \mathsf{h}_{0},$ it also contains the surface grid of order $w$ with $\mathsf{h}+(\mathsf{h}_0-\mathsf{h}')$ handles and $\mathsf{c}-2(\mathsf{h}_{0}-\mathsf{h}')$ crosscaps as a minor.

In particular, our graph $H'$ from above contains a Dyck-grid of order $k$ and Euler-genus $2\mathsf{h}+\mathsf{c}$ as a minor.

Now let $g\coloneqq 2\mathsf{h}+\mathsf{c}$ as in the statement of our theorem and notice the following inequalities:
\begin{align*}
  \mathsf{c} & \leq  g,\\
  2\mathsf{h} & \leq  g,\text{ and}\\
  2\mathsf{h}_{0} &\leq  \mathsf{c},
\end{align*}
It follows that
\begin{align*}
  162^{\mathsf{h}_{0}+\mathsf{h}+\mathsf{c}}\cdot k & \leq  162^{\frac{1}{2}g + g + \frac{1}{2}g}\cdot k\\
  &\leq  162^{2g}\cdot k.
\end{align*}
This is true for any choice of $\mathsf{h}'$ and $\mathsf{c}'$ such that $2\mathsf{h}'+\mathsf{c}'=g.$
Hence, by the discussion above, that any mixed surface grid of order $162^{2g}\cdot k$ with $\mathsf{h}'$ handles and $\mathsf{c}'$ crosscaps where $2\mathsf{h}'+\mathsf{c}'=g$ contains a Dyck-grid of order $k$ and Euler-genus $g$ as a minor.
Moreover, the (orientable in case $\mathsf{c}'=0$ and otherwise non-orientable) Dyck-grid of order $162^{2g}k$ contains every mixed surface grid of order $k$ with $\mathsf{h}'$ handles and $\mathsf{c}'$ crosscaps as a minor.
Hence, our proof is complete.
\end{proof}

\subsection{Universal graphs for graphs of bounded Euler-genus.}

An important consequence of \cref{diskussionen} is that Dyck-walls are indeed minor-universal for graphs of bounded Euler-genus.
Indeed, any mixed surface wall with $h$ handles and $\mathsf{c}$ crosscaps is minor-universal for graphs that embed into the surface obtained from the sphere by adding $h$ handles and $\mathsf{c}$ crosscaps. 
\red{We use notation $\Sigma^{(\mathsf{h},\mathsf{c})}$ for a surface with ${\sf h}$ handles and ${\sf c}$ crosscap.}

Let us now state the theorem of Gavoille and Hilaire.

\begin{proposition}[Gavoille and Hilaire \cite{gavoille2023minor}]\label{territorial}
For every pair $(\mathsf{h},\mathsf{c})\in\mathbb{N}\times[0,2],$ a graph $G$ is embeddable in $\Sigma^{(\mathsf{h},\mathsf{c})}$ if and only if there exists a\red{n integer $k=k(\mathsf{h},\mathsf{c},G)$} such that one of the following holds
\rev{38. P11, L-8: Delete “it is a”. $k(G)$ should be $k(h, c, G)$ because it also depends on $h$ and $c$. And you should mention that $k(G)$ is an integer.}
\ans{We deleted ``it is a'' and we made the changes.}
\begin{itemize}
\item in case $\mathsf{c}=0,$ $G$ is a minor of $\mathscr{D}^{(\mathsf{h},0)}_k,$ and
\item in case $\mathsf{c}\neq 0,$ $G$ is a minor of \red{a mixed surface} grid of order $k$ with $0$ handles and $2\mathsf{h}+\mathsf{c}$ crosscaps.
\end{itemize}
Moreover, we have that $k(G)\in \mathcal{O}((2\mathsf{h}+\mathsf{c})^2(2\mathsf{h}+\mathsf{c}+|V(G)|)^2).$
\end{proposition}

We now obtain the fact that Dyck-walls are minor-universal for graphs of bounded Euler-genus as a direct corollary of \cref{territorial} and \cref{diskussionen}.

\begin{corollary}\label{territorial2}
For every pair $(\mathsf{h},\mathsf{c})\in\mathbb{N}\times[0,2],$ a graph $G$ is embeddable in $\Sigma^{(\mathsf{h},\mathsf{c})}$
\rev{40. P11, L-2: What is $\Sigma^{(h,c)}$?} 
\ans{We added the definition before Proposition 3.6 in Subsection 3.3.}

if and only if it is a minor of $\mathscr{D}^{(\mathsf{h},\mathsf{c})}_{k},$ for some $k=k(G)\in\mathbb{N}.$
\rev{41. P11, L-1: $k(G)$ should be $k(h,c,G)$.}
\ans{Fixed.}
Moreover, $\red{k(\mathsf{h},\mathsf{c},G)}\in 2^{\Ocal(\mathsf{h}+\mathsf{c})}.$
\end{corollary}

\section{A brief introduction to the toolset of graph minors}\label{hypocritical}

We will be using a theorem developed by Kawarabayashi, Thomas, and Wollan in their proof of the GMST \cite{KawarabayashiTW20Quicklyexcluding} which makes the extraction of some surface grid a very achievable task.
However, to be able to even state this result we need to introduce a lot of machinery.
The purpose of this section is the introduction and collection of the various technical definitions from \cite{KawarabayashiTW20Quicklyexcluding} we need.

\subsection{Walls and tangles}\label{subsec_wallsandtangles}

\paragraph{Walls.}
An \emph{$(n\times m)$-grid} is the graph $\Gamma_{n,m}$ with vertex set $[n]\times[m]$ and edge set \[\{\{(i,j),(i,j+1)\} \mid i\in[n],j\in[m-1]\}\cup\{\{(i,j),(i+1,j)\} \mid i\in[n-1],j\in[m]\}.\]
We call the path $(i,1)(i,2)\dots(i,m)$ the \emph{$i$th row} and the path $(1,j)(2,j)\dots(n,j)$ the \emph{$j$th column} of the grid.
\red{An edge $\{(i,j)(i,j+1)\}$ of the $i$th row is an \emph{even} (\emph{odd}) edge of the $i$th row if $j$ is \emph{even} (\emph{odd}).
Similarly, an edge $\{(i,j)(i+1,j)\}$ of the $j$th column is an \emph{even} (\emph{odd}) edge of the $j$th column if $i$ is \emph{even} (\emph{odd}).}
An \emph{elementary $k$-wall}~$W_k$ for~${k \geq  3},$ is obtained from the ${(k\times 2k)}$-grid $\Gamma_{k,2k}$ by deleting every odd edge in every odd column and every even edge in every even column, and then deleting all degree-one vertices.
\rev{42. P12, L10-11: What are odd edges and even edges?
}
\ans{We added a definition.}
The \emph{rows} of~$W_k$ are the subgraphs of~$W_k$ induced by the rows of~$\Gamma_{k,2k},$ while the \emph{$j$th column} of~$W_k$ is the subgraph induced by the vertices of columns~${2j-1}$ and~${2j}$ of~$\Gamma_{k,2k}.$
We define the \emph{perimeter} of $W_k$ to be the subgraph induced by
{$${\{ (i,j) \in V(W_k) \mid j \in \{1,2,2k,2k-1\} \text{ and } i \in [k] \textnormal{, or } i \in \{1,k\} \text{ and } j \in [2k] \}}.$$}
A \emph{$k$-wall}~$W$ is a graph isomorphic to a subdivision of~$W_k.$

\begin{figure}
  \centering
\includegraphics[scale=.9]{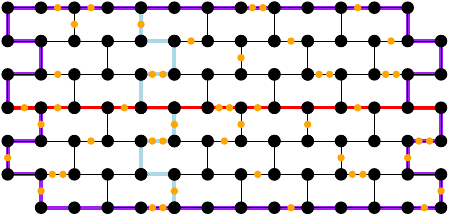}
    \caption{An illustration of a 7-wall $W.$ The perimeter is drawn in \textcolor{DarkLavender}{purple}, the 3rd column is drawn in \textcolor{CornflowerBlue}{light blue}, the 4th line is drawn in \textcolor{BostonUniversityRed}{red}, and the subdivision vertices in \textcolor{ChromeYellow}{yellow}.}
    \label{fig_overlay_d}
\end{figure}

The vertices of degree three in $W$ are called the \emph{branch vertices}.
In other words, $W$ is obtained from a graph $W'$ isomorphic to $W_k$ by subdividing each edge of~$W'$ an arbitrary (possibly zero) number of times.
The \emph{perimeter} of~$W,$ denoted by~$\Perimeter(W),$ is the subgraph isomorphic to the subgraph of~$W'$ induced by the vertices of the perimeter of~$W_k$ together with the subdivision vertices of the edges of the perimeter of~$W_k.$
We define rows and columns of $k$-walls analogously to their definition for elementary walls (see \cref{fig_overlay_d} for an example).
A \emph{wall} is a $k$-wall for some~$k.$

An $h$-wall $W'$ is a \emph{subwall} of some $k$-wall $W$ where $h\leq  k$ if every row (column) of $W'$ is contains in a row (column) of $W.$

The following is a statement of the Grid Theorem.
While we will not explicitly use the Grid Theorem, we will, later on, make use of the existence of a function that forces the existence of a large wall in graph with large enough treewidth.

\begin{proposition}[Grid Theorem \cite{chuzhoy2021towards}]\label{broadcasting}
	There exists a function $\mathsf{p}\colon\Nbbb\to\Nbbb$ with $\mathsf{p}(n)\in\mathcal{O}(n^9\operatorname{\mathsf{poly log}}(n))$ such that for every $k\in \Nbbb$ 
\rev{43. P12, L-15: You cannot cite [3] once you specify $p(n)\in O(n^9...)$ because this bound for p(n)was not proved in [3].}
	and every graph $G,$ if $\tw(G)>\mathsf{p}(k)$ then $G$ contains a $k$-wall as a subgraph.
\end{proposition}
\ans{We removed the reference. Later we also cite earlier results with non-polynomial bounds.}

We stress that the above result holds if we replace walls by grids as these two parametric graphs are linearly equivalent (in the sense this is equivalence is defined in the begnning of \cref{sec_dycks}). See also \cite{robertson1986graph,DiestelJGT99,RobSeymT94} for earlier proofs of \cref{broadcasting} with non-polynomial dependencies.
\rev{44. P12, L-13: You cannot replace walls by grids because Proposition 4.1 was stated in terms of subgraphs, not minors. Moreover, I do not understand what “these two parametric graphs are linearly equivalent” means. Parametric graphs are undeﬁned, and it is unclear what equivalence relation you are considering.}
\ans{In the new version we introduced the concept of parametric graph and the definition of (linear) parametric equivalence.}
 
In~\cite{robertson1991graph} Robertson and Seymour introduced tangles as the max-min counterpart of the parameter of branchwidth, that is linearly equivalent to treewidth.
Tangles have been important in the proofs of the Graph Minors series of Robertson and Seymour.
Also, they play an important role in an abstract theory of connectivity, see for example~\cite{Grohe16bTangledup,DiestelK20profi,DiestelO19tangl,DiestelEW19struc,Diestel18abstr}.

\paragraph{Tangles.}
\rev{45. P12, L-11 and L-7: The titles of these two paragraphs are the same.}
\ans{fixed.}
A \emph{separation} in a graph $G$ is a pair $(A,B)$ of vertex sets such that $A\cup B=V(G)$ and there is no edge in $G$ with one endpoint in $A\setminus B$ and the other in $B\setminus A.$
The \emph{order} of $(A,B)$ is $|A\cap B|.$
Let $G$ be a graph and $k$ be a positive integer. 
We denote by $\mathcal{S}_k(G)$ the collection of all tuples $(A,B)$ where $A,B\subseteq V(G)$ and $(A,B)$ is a separation of order smaller than $k$ in $G.$ 
\rev{46. P12, L-4: $< k$ should be replaced by plain texts. }
\ans{Done.}
An \emph{orientation} of $\mathcal{S}_k(G)$ is a set $\mathcal{O}$ such that for all $(A,B)\in\mathcal{S}_k(G)$ exactly one of $(A,B)$ and $(B,A)$ belongs to $\mathcal{O}.$ 
A \emph{tangle} of order $k$ in $G$ is an orientation $\mathcal{T}$ of $\mathcal{S}_k(G)$ such that 
\begin{eqnarray}
\mbox{for all $(A_1,B_1),(A_2,B_2),(A_3,B_3)\in\mathcal{T},$ it holds that $\red{G[A_1]\cup G[A_2]\cup G[A_3]\neq G}.$}\label{tangleax}
\end{eqnarray}
Let $\mathcal{T}$ be a tangle in a graph $G.$
A tangle $\mathcal{T}'$ in $G$ is a \emph{truncation} of $\mathcal{T}$ if $\mathcal{T}'\subseteq\mathcal{T}.$

\red{Notice that our definition differs slightly from the original definition by Robertson and Seymour \cite{robertson1991graph}, however, as pointed out -- for example -- in \cite{KawarabayashiTW20Quicklyexcluding}, the two notions are equivalent.}

Let $W$ be an $r$-wall in a graph $G$ and assume that the
\red{columns and the rows} are specified.
 
\rev{47. P13, L2: Horizontal paths and vertical paths are undefined.}
\ans{We now use columns and rows. They are all defined in \cref{subsec_wallsandtangles}. We also added a new figure (\cref{fig_overlay_d}) to visualize these concepts.}
Let $\mathcal{T}_W$ be the orientation of $\mathcal{S}_r\red{(G)}$
\rev{48. P13, L3: Does $\Scal_{r}$ equal $\Scal_{r}(W)$ or $\Scal_{r}(G)$?}
\ans{Fixed}
such that for every $(A,B)\in\mathcal{T}_W,$ $B\setminus A$ contains the vertex set of both a \red{row} and a \red{column} of $W,$ we call $B$ the \emph{$W$-majority side} of $(A,B).$
\rev{49. P13, L4: You have to show that $T_W$ is indeed a tangle. Your definition of tangles is different from the one defined by Robertson and Seymour. So even though this result for Robertson and Seymour’s tangles is known in the literature, you cannot use it for granted unless you argue how to transform the proof of the result for Robertson and Seymour’s tangles to the result for your tangles or cite a result in the literature that directly applies to your tangles. Similar comments apply to P13, L7.}
\ans{The definition was indeed wrong, we updated it to the one from  \cite{KawarabayashiTW20Quicklyexcluding}. We also added a reference and commented on the equivalence between the two definition.}
Then $\mathcal{T}_W$ is the tangle \emph{induced} by $W.$

Similarly, let $h\in\mathbb{N}$ and $c\in[0,2]$ and let $D$ be an $(\red{\mathsf{h},\mathsf{c}};t)$-Dyck-wall in $G.$
Then also $D$ defines a majority side for any separation of order less than $r.$
In this case the majority side is the side that contains a full cycle and a column of the annulus wall of $D.$
In the same way as above this allows us to define the tangle $\mathcal{T}_D$ \emph{induced by} $D.$

\paragraph{Well-linked-sets and tangles.} Let $\alpha  \in [2/3, 1),$ $G$ be a graph and $S \subseteq V(G)$ be a vertex set. 
A set $X \subseteq V(G)$ is said to be an \emph{$\alpha $-balanced separator} for $S$ if for every component $C$ of $G - X$ it holds that $|V(C) \cap S| \leq \alpha |S|.$ 
\rev{50. P13, L11: $C(C)$ should be $V(C)$.}
\ans{Fixed.}
\rev{51. P13, L13: $\Scal^{q+1}$ should be $\Scal_{q+1}(G)$.}
\ans{Fixed.}
Let $q \in \Nbbb.$ We say that $S$ is a \emph{$(q, \alpha )$-well-linked set} of $G$ if there is no $\alpha $-balanced separator of size at most $q$ for $S$ in $G.$ 
Given a $(q, \alpha )$ well-linked set of $G$ we define $$\mathcal{T}_{S} \coloneqq \{ (A, B) \in \mathcal{S}_{q+1} \mid |S \cap B| > \alpha |S| \}.$$ 
It is not hard to see that $\mathcal{T}_{S}$ is a tangle of order $q$ in $G.$
  
Notice that it is easier to encode and handle  algorithmically  a well-linked set $S$ that a tangle $\Tcal,$ as the first is a set of vertices while the second is a set of pairs of sets.
\rev{52. P13, L15: I do not understand what “represent well-linked sets of unbounded size with a bounded number of \red{memory}”. What is your presentation?}
\ans{We revised this paragraph. We make clear that linked sets are easier to encode than tangles. This makes it easier to deal with the algorithmically and this justifies why we opt for presenting our results using well-linked sets instead of tangles.}
For this reason, we opt  to work with well-linked sets instead of tangles throughout this paper.
\rev{53. P13, L17: In fact, you are only allowed to use a limited amount of results of graph minors. Your definition for tangles is different from the ones deﬁned by Robertson and Seymour. So you cannot use any result in Robertson and Seymour’s graph minor series or any results in the literature that use Robertson and Seymour’s tangles.}
\ans{Our definition was indeed erroneous, the new one we give is equivalent to the one of Robertson and Seymour (as explained for example in \cite{KawarabayashiTW20Quicklyexcluding}) so this is not an issue.
We added some comments next to the definition to elaborate.}

Please note that, \red{with the observation above on tangles from well-linked sets}, this still allows us to make use of all results of graph minors that need a large order tangle as part of the input.

Let $G$ be a graph and $v\in V(G)$ be a vertex of degree at least four.
We say that $G'$ is obtained from $G$ by \emph{splitting $v$} if there exist an edge $x_1x_2\in E(G')$ such that $G-v=G'-x_1-x_2$ and $N_G(v)=\red{(N_{G'}(x_1)\cup N_{G'}(x_2))\setminus \{ x_1,x_2\}}$ where $N_{G'}(x_1)\cap N_{G'}(x_2)\red{=}\emptyset.$
If there exists a vertex $v$ such that a graph $G'$ is obtained from a graph $G$ by splitting $v,$ we say that $G'$ is obtained from $G$ by \emph{splitting a vertex}.
\rev{54. P13, L20: Based on your definition, every complete graph with arbitrarily large size can be obtained from $K_2$ by repeatedly splitting vertices. Is it what you want?}
\ans{We changed the definition to impose that the neighborhood of $v$ is partitioned by $x_1$ and $x_2$.}
A pair $(H,T)$ is said to be an \emph{expansion} of a graph $G,$ or a \emph{$G$-expansion}, if $H$ is a subdivision of a graph $G'$ obtained from $G$ by a sequence of vertex-splits, every vertex $v$ with $\deg_H(v)\geq 3$ belongs to $T\subseteq V(H),$ and $G'$ is the graph obtained from $H$ by dissolving\footnote{The \emph{dissolution} of a vertex $v$ of degree two in a graph $G$ is the operation of removing $v$ and its incident edges from $G$ and the addition of a new edge joining the two neighbors of $v$.} every vertex of $V(H)\setminus T.$
\rev{55. P13, L25: What does “dissolving” mean?}
\ans{We added the definition.}
We call $T$ the \emph{branch vertices} of $(H,T).$
\rev{56. P13, L25: ... call $T$ the set of \emph{branch vertices} ...}
\ans{Fixed.}

\subsection{Growing a wall from a well-linked set}
\label{biographies}

In this subsection we describe how, given an $(s,\alpha)$-well-linked 
set $S$ for large enough $s$ we can find a large wall $W$ in {\sf FPT}-time such that $\mathcal{T}_W$ is a truncation of $\mathcal{T}_S.$
This will give us the right to equip our local structure theorems with an additional property that will allow us to localise it with respect to some tangle which will be encoded using the set $S$ instead of the collection of separations which can be very large.
Towards this goal, we present an algorithmic version of Lemma 14.6 from \cite{KawarabayashiTW20Quicklyexcluding}.
That is, we prove the following theorem.

\begin{theorem}\label{thm_algogrid}
Let $k\geq 3$ be an integer, $\alpha\in [2/3,1).$
There exist universal constants $c_1,c_2\in\mathbb{N}\setminus\{ 0\},$ and an algorithm that, given a graph $G$ and a 
\rev{57. P13, L-12: “an” should be “a”}
\ans{Fixed everywhere.}
$(36c_1k^{20}+3,\alpha)$-well-linked set $S\subseteq V(G)$ computes in time $2^{\mathcal{O}(k^{c_2})}|V(G)|\red{|E(G)|^2}\log(|V(G)|)$ a $k$-wall $W\subseteq G$ such that $\mathcal{T}_W$ is a truncation of $\mathcal{T}_S.$\ans{fixed running time here}
\end{theorem}

Our proof for \autoref{thm_algogrid} combines results from three different papers.
For the sake of completeness, since we need to modify the proofs of most of these results slightly, we present a complete and self contained proof here.

The first step is to show that, given our well-linked set $S$ we can compute from it a set with slightly stronger properties which also induces a tangle and whose tangle is a truncation of $\red{\mathcal{T}_S}.$
\rev{58. P13, L-6: What is $\Scal$?}
\ans{$\Scal$ should be $S$ here. We fixed this.}

To stress that all results presented here are algorithmic, we phrase them mostly using the well-linked set $S$ and only make use of abstract properties of $\mathcal{T}_S$ without requiring the tangle to be given.

Let $\alpha\in[2/3,1)$ be fixed, $k$ be a positive integer, and $S$ be a $(k,\alpha)$-well-linked set in a graph $G.$
A set $F$ is said to be \emph{$S$-free} in $G$ if for every separation $(A_1,A_2)\in\mathcal{S}_{|F|}\red{(G)}$ it holds that, if $F\subseteq A_i$
\rev{59. P13, L-2: $\Scal_{|F|}$ should be $\Scal_{|F|}(G)$. }
\ans{Fixed.}
for $i\in[2],$ then $|A_i\cap S|>\alpha|S|.$

The proof of the following lemma is an adaptation of Lemma 14.1 from \cite{KawarabayashiTW20Quicklyexcluding}.

\begin{lemma}\label{prohibiting}
Let $k$ be a positive integer, $\alpha\in[2/3,1),$ $G$ be a graph, and $S\subseteq V(G)$ be $(k,\alpha)$-well-linked.
Then there exists an $S$-free set $F$ of size $k-1.$

Moreover, there exists an algorithm that, given $G$ and $S,$ computes $F$ in time $\mathcal{O}(k\red{|V(G)||E(G)|}).$
\end{lemma}

\begin{proof}
First let us observe that
\rev{60. P14, L5: Delete “if $G$ is disconnected”. You do not need this assumption. And if you have this assumption, you do not explain how to construct an $S$-free set of size $i+ 1$ when $i=0$.}
\ans{Fixed}
there must exist some component $G'$ of $G$ such that $|V(G')\cap S|>\alpha|S|$ since $S$ is well-linked.
\rev{61. P14, L6: ... if we ﬁx ... and set $F_0\coloneqq \{x_0\}$...}
\ans{Fixed.}
Hence, if fix any vertex $x_0\in V(G')$ and set $F_0\coloneqq \{x_0\},$ then $F_0$ is $S$-free.

As $S$ is $(k,\alpha)$-well-linked it follows that $|S|\geq 3k+1$ since otherwise we could delete $k$ vertices from $S$ and any component in the resulting graph has at most $\frac{2}{3}|S|\leq \alpha|S|$ vertices of $S.$

Assume we have already constructed an $S$-free set $F_i$ of size $i+1$ for $i\in[0,k-2].$

Now let $(A,B)$ be a separation of order $i+1$ such that $F_i\subseteq A$ and $|B\cap S|>\alpha|S|.$
Such a separation exists since $(F_i,V(G))$ satisfies this property by the size of $S.$

\red{We now iteratively ``push'' $(A,B)$ towards $S$ by asking, in each step, for a separation of size at most $i+1$ between $A\cap B$ and $B\cap S.$
The goal of this process is to minimize $|B|.$
Once this minimization step concludes -- that is after at most $|V(G)|$ iterations, each of which takes $\mathcal{O}(k\cdot |V(G)|)$ time -- we select $x_{i+1}\in B\setminus A$ to be any vertex.}
\red{Indeed, let $(A,B)$ be a separation of order $i+1$ such that $F_i\subseteq A$, $|B\cap S| > \alpha|S|$, and $|B|$ is minimized.
Then let $x_{i+1}\in B\setminus A$ be any vertex.}

\rev{62. P14, L12-15: These sentences should be revised.}
\ans{We added some text in the end of this paragraph to clarify our final choice of $(A,B)$.}

Suppose, $F_{i+1}\coloneqq F_i\cup\{ x_{i+1}\}$ is not $S$-free.
Then there exists a separation $(X,Y)$ of order at most $i+1$ where $F_{i+1}\subseteq X$ and $|Y\cap S|>\alpha|S|.$
Moreover, such a separation can be found in time \red{$\mathcal{O}(k\cdot |E(G)|).$}
\rev{63. P14, L15 and L17: It seems to me that $O(k)|V(G)|$ should be $O(k)|E(G)|$. If it is the case, the running
time in this lemma should be changed.}
\ans{We added some text here for this.}
It follows from the assumption that $F_{i+1}$ is $S$-free that $|X\cap Y|=i+1.$

Now, consider the separation $(A\cap X,B\cup Y).$
Since $|Y\cap S|>\alpha|S|$ it follows that $|(B\cup Y)\cap S|>\alpha|S|.$
From the assumption that $F_i$ is $S$-free this means that $(A\cap X,B\cup Y)$ is of order at least $i+1.$
Thus, by the submodularity of separations in undirected graphs, it follows that $(A\cup X,B\cap Y)$ is of order at most $i+1.$
Sine $A$ is a proper subset of $A\cup X,$ our construction of $(A,B)$ implies that $|A\cup X|>\alpha|S|.$
\rev{64. P14, L22: $|A\cup  X|$.}
\ans{Fixed.}
However, this means that $A\cup X\cup (B\cap Y)=V(G)$ while $(A,B),(X,Y),(B\cap Y,A\cup X)\in\mathcal{T}_S.$
This is a clear violation of the tangle axioms and therefore impossible.
Hence, a separation like $(X,Y)$ cannot exist and this means that $F_{i+1}$ is, in fact, $S$-free.
\end{proof}

Free sets \red{of large order} are themselves \red{``highly connected''}.
\rev{65. P14, L26: “Free sets are themselves highly connected” does not make sense. “Connected” is used for describing graphs, not sets. Moreover, the empty set is free, and how can the empty set “highly connected”?}
\ans{We added that we assume the free set to be big. We also put highly connected in quotation marks.}
Indeed, they exhibit a slightly stronger property than just being $(k,\alpha)$-well-linked.

Let $G$ be a graph and $S\subseteq V(G)$ be a set of vertices.
We say that $S$ is \emph{strongly linked}\footnote{In \cite{KawarabayashiTW20Quicklyexcluding} this property is called ``well-linked'' but we decide to change it here to not overload the term.} if for every partition $\{S_1,S_2\}$ of $S$ into two sets there 
\rev{66. P14, L29: $\{S_1, S_2\}$ of $S$...}
\ans{Fixed.}
does not exist a separation $(A_1,A_2)$ of $G$ such that $S_i\subseteq A_i$ for both $i\in[2],$ and $|A_1\cap A_2|<\min\{ |S_1|,|S_2|\}.$
Notice that every strongly linked set of size $3k+1$ is also $(k,2/3)$-well-linked.

The following is a corollary of Lemma 14.2 from \cite{KawarabayashiTW20Quicklyexcluding} phrased in terms of well-linked set instead of tangles.

\begin{lemma}[\!\! \cite{KawarabayashiTW20Quicklyexcluding}]\label{differentiated}
Let $\alpha\in[2/3,1),$ $k\geq 1$ be an integer, $G$ be a graph and $S\subseteq V(G)$ be $(k,\alpha)$-well-linked.

If $F\subseteq V(G)$ is $S$-free with $|F|<k$ then $F$ is strongly linked.
\end{lemma}

We believe that it is possible to prove the following sequence of lemmas purely in terms of $(k,\alpha)$-well-linked sets, but for the sake of brevity we stick to the setting of Kawarabayashi, Thomas, and Wollan \cite{KawarabayashiTW20Quicklyexcluding}.

\begin{lemma}\label{sublimated}
Let $\alpha\in[2/3,1),$ $k\geq 1$ be an integer, $G$ be a graph and $S\subseteq V(G)$ be $(3k+1,\alpha)$-well-linked.

If $F\subseteq V(G)$ is $S$-free with $|F|=3k$ then $\mathcal{T}_F\coloneqq \{ (A,B)\in\mathcal{S}_k \mid |B\cap F|>2k \}$ is a tangle of order $k$ which is a truncation of $\mathcal{T}_S.$
\end{lemma}

\begin{proof}
First let us show that for every $(A,B)\in\mathcal{S}_k$ we have $\{ (A,B),(B,A)\}\cap\mathcal{T}_F\neq\emptyset.$
Since $|F|=3k$ this implies that $\FFcal$ contains exactly one of $(A,B)$ and $(B,A)$ for every $(A,B)\in\mathcal{S}_k.$

So let $(A,B)\in\mathcal{S}_k.$
This means $|A\cap B|<k.$
Hence, since $F$ is strongly linked by \cref{differentiated}, this means that $\min\{ |A\cap F|,|B\cap F| \}<k.$
With $|F|=3k$ this implies that $\max\{ |A\cap F|,|B\cap F|\}>2k.$

Moreover, let $(A_1,B_1),(A_2,B_2),(A_3,B_3)\in\mathcal{T}_F.$
Then $|A_i\cap F|\leq k-1$ for all $i\in[3].$
Therefore, $|(A_1\cup A_2\cup A_3)\cap F|\leq 3k-3<3k$ and thus, the union of three small sides cannot cover all of $G$ and so $\mathcal{T}_F$ is indeed a tangle of order $k.$

Finally, we have to show that $\mathcal{T}_F$ is a truncation of $\mathcal{T}_S.$
To see this suppose there is some $(A,B)\in\mathcal{T}_F$ such that $(B,A)\in\mathcal{T}_S.$
Notice that this means that $|A\cap B|<k$ and $|B\cap F|>2k$ by definition of $\mathcal{T}_F.$
Hence, $|A\cap F|\leq k-1.$
Thus, $(A,B\cup F)$ is a separation of order less than $2k.$
\rev{68. P15, L-15: $< k$ should be “less than $k$”.}
\ans{Fixed.}
As $F\subseteq B\cup F$ it follows from the definition of $S$-free sets that $|(B\cup F)\cap S|>\alpha|S|$ and thus, by definition of $\mathcal{T}_S$ we have $(A,B\cup F)\in\mathcal{T}_S.$
However, as $(B,A)\in\mathcal{T}_S$ by assumption, we now have that two small sides, namely $A$ and $B,$ of separations in $\mathcal{T}_S$ cover all of $G.$
This is clearly a contradiction to $\mathcal{T}_S$ being a tangle and thus we must have that $\mathcal{T}_F\subseteq \mathcal{T}_S.$
\end{proof}

An interesting, separate corollary of the proof of \cref{sublimated} which is reflected in Lemma 14.4 from \cite{KawarabayashiTW20Quicklyexcluding} is the following.

\begin{corollary}[\!\! \cite{KawarabayashiTW20Quicklyexcluding}]\label{meaninglessness}
Let $k\geq 1$ be an integer, $G$ be a graph and $F\subseteq V(G)$ be a strongly linked set of size $3k$ in $G.$
Then $\mathsf{tw}(G)\geq k.$
\end{corollary}

We will also need the following observation on grid-minors.
This observation will be helpful in determining if a grid we have found yields a tangle that agrees with $\mathcal{T}_S.$

\begin{observation}\label{penitentiary}
Let $k\geq 2$ be an integer, $G$ be a graph and $(M,T)$ be an expansion of the $(k\times k)$-grid in $G.$
\red{Let $\Gamma_k$ denote the $(k \times k)$-grid and let $\mathcal{X} = \{ X_v\}_{v\in V(\Gamma_k)}$ be a minor model of $\Gamma_k$ in $M$.
Notice that, by the fact that $(M,T)$ is an expansion of $\Gamma_k$, we may choose $\mathcal{X}$ such that $X_v$ is a tree for every $v\in V(\Gamma)_k$.
Let $Q$ be any row or column of $\Gamma_k$ and let $S$ be obtained by selecting a vertex $s_v \in V(X_v) \cap T$ for each $v\in V(Q)$ (in case $\mathsf{deg}_{\Gamma_k}(v)=2$ select $s_v$ to be an arbitrary vertex of $X_v$).
Then $S$ is strongly linked in $G$.}
\rev{67. P15, L19-21: This paragraph does not seem correct. What is $X_v$ ? Where is $X_v$ used?}
\ans{We have now defined the notion of a minor model and updated the statement accordingly.}
\end{observation}

If the set $S$ as above is taken from the $i$th row (column), we call $S$ a \emph{representative set} of the $i$th row (column) of $(M,T).$
Since every $k$-wall naturally correspondonds to an expansion of the $(k\times k)$-grid.

Next, we adapt Lemma 14.5 from \cite{KawarabayashiTW20Quicklyexcluding} for our purposes.
That is, we present almost the same proof as the one found in the work of Kawarabayashi, Thomas, and Wollan while stressing that it is algorithmic.

\begin{lemma}\label{evaluation}
Let $k\geq 1$ be an integer, $G$ be a graph and $X,Y\subseteq V(G)$ be strongly linked sets such that $|X|\geq k$ and $|Y|\geq 3k.$
Exactly one of the following two statement holds.
\begin{enumerate}[itemsep=-2pt]
\item There exist $k$ pairwise vertex-disjoint $X$-$Y$-paths in $G,$ or
\item if $(A,B)$ is a separation of order less that $k$ such that $X\subseteq A$ and $|Y\cap B|\geq k,$ then either
\begin{enumerate}[itemsep=-2pt]
  \item there exists a separation $(A',B')$ of order $<k$ such that $X\subseteq A',$ $|Y\cap B'|\geq k,$ $A\subseteq A',$ and $B'\subsetneq B,$ or
  \item there exists an edge $e$ of $G[B]$ such that $X$ is strongly linked in $G-e.$ 
\rev{69. P15, L-12: Change “it holds” to “such”.}
\ans{Fixed.}
\end{enumerate}
\end{enumerate}
Moreover, there exists an algorithm that, in time $2^{\mathcal{O}(k)}\cdot\red{|E(G)|},$ finds either the paths, the separation $(A',B'),$ or the edge $e$ as above.
\end{lemma}

\begin{proof}
Given the sets $X$ and $Y$ we may determine in time $\mathcal{O}(k\cdot |\red{E}(G)|)$
\rev{70. P15, L-9: It seems to me that $O(k|V(G)|)$ should be $O(k|E(G)|)$.}
\ans{Fixed.}
if there exist $k$ pairwise vertex-disjoint paths between $X$ and $Y$ or find some separation $(A,B)$ of order less than $k$ such that $X\subseteq A$ and $Y\subseteq B.$
In particular, it holds that $|B\cap Y|\geq k.$

So from now on we assume that we are given some separation $(A,B)$ of order at most $k-1$ such that $X\subseteq A,$ and $|Y\cap B|\geq k.$

Fix an edge $xy=e\in E(G[B]).$
We can test in time $2^{\mathcal{O}(k)}|\red{E}(G)|$ if there exists a partition of $X$ into sets $X_1$ and $X_2$ such that there exists a separation $(L,R)$ of $G-e$ with $X_1\subseteq L,$ $X_2\subseteq R,$ and $|L\cap R|<p\coloneqq \min\{|X_1|,|X_2|\}.$
If such a partition cannot be found we have successfully determined that $X$ is strongly linked in $G-e$ and can output the $e.$
Otherwise we continue to work with the separation $(L,R).$

Without loss of generality we may assume $x\in L.$
Indeed, it must now hold that $x\in L\setminus R,$ and $y\in R\setminus L$ since otherwise $(L,R)$ is witness that $X$ is not strongly linked in $G$ which contradicts our assumption.
Moreover, $x$ has at least one neighbour in $L\setminus R$ as otherwise it can be moved to $R$ to form a separation of $G$ which violates the strong linkedness of $X.$
Similarly, $y$ has some neighbour in $R\setminus L.$

Observe that $(A\cap L,B\cup R)$ is a separation of $G$ with $X_1\subseteq A\cap L$ and $X_2\subseteq B\cup R.$
Hence, it must be at least of order $p+1.$
Let $q\coloneqq |A\cap B|<k.$
It follows from the submodularity of separations in undirected graphs that $(A\cup L,B\cap R)$ has order at most $q-1.$
Notice that, by choice of $e,$ we have $x\notin B\cap R.$
Hence, $(A\cup L,(B\cap R)\cup\{x\})$ is a separation of $G$ of order at most $q<k.$
Moreover, $(B\cap R)\cup \{ x\}\subsetneq B.$

In the case where $|(B\cap R)\cup \{ x\}\cap Y|\geq k.$
We have found our desired separation and may terminate.
Hence, we may assume $|(B\cap R)\cup \{ x\}\cap Y|\leq k-1.$

Following a symmetric argument we may analyse the separations $(A\cap R,B\cup L)$ and $(A\cup R,B\cap L)$ and reach the same two possible outcomes.
Either $(A\cup R, (B\cap L)\cup\{y\})$ is the separation we were looking for, or we also obtain $|(B\cap L)\cup \{y\}|\leq k-1.$

Now if none of the two separations we have found is suitable for our needs we have $|(B\cap R)\cup \{ x\}\cap Y|\leq k-1$ and $|(B\cap L)\cup \{y\}|\leq k-1.$
It now follows that $|B\cap Y|\leq 2k-2.$
This, however, would imply that $|A\cap Y|\geq k+2$ which contradicts our initial assumption of $(A,B).$
This completes the proof.
\end{proof}

The last piece towards the proof of \cref{thm_algogrid} is an adaptation of a theorem of Fomin, Lokshtanov, Panolan, Saurabh, and Zehavi \cite{fomin2020hitting1,fomin2019hitting2}, that is Proposition 3.6 in \cite{fomin2019hitting2}.

Moreover, we need a way to compute a tree decomposition of bounded width if one exists.
To this end, we use Korhonen's \cite{korhonen2022single} recent result.

\begin{proposition}[\!\! \cite{korhonen2022single}]\label{conquerors}
There exists an algorithm that, given a graph $G$ and an integer $k$ outputs, in time $2^{\mathcal{O}(k)}|V(G)|,$ a tree decomposition for $G$ of width at most $2k+1$ or concludes that $\mathsf{tw}(G)>k.$
\end{proposition} 

\begin{lemma}\label{counterfeit}
There exists a universal constant $c\geq 1$ and an algorithm that, given a graph $G$ and a $(3\cdot c\cdot k^{10}+1,\alpha)$-well-linked set $S\subseteq V(G)$ for some integer $k$ and some $\alpha\in[2/3,1)$ computes in time $2^{\mathcal{O}(k^{10}\log k)}|V(G)|\log(|V(G)|)$ a $k$-wall $G.$
\rev{71. P17, L2: $O(\log |V(G)|)$.}
\ans{Fixed.}
\end{lemma}

\begin{proof}
Our goal is to find a subgraph $G'$ of $G$ whose treewidth is at least $c_1k^{10}$ for some constant $c_1$ and at most $c_2k^{10}$ for another constant $c_2\geq c_1.$
Once we have found this subgraph, we may use \cref{conquerors} to compute a tree decomposition of bounded width for $G$ and then proceed with dynamic programming to check if $G'$ contains the $(k\times k)$-grid as a minor (using i.e., the algorithm of \cite{AdlerDFST11Faster}).
Since $G'$ satisfies $\mathsf{tw}(G')\geq c_1k^{10},$ by choosing $c_1$ to be the constant  from \cref{broadcasting}, we are guaranteed to find the desired $k$-wall.

In the following we describe how to find the graph $G'.$
This idea is taken directly from \cite{fomin2019hitting2} and we only modify two details.
Note that the highly well-linked set ensures that we can never run into the case where we determine that $G$ has small treewidth.

Let us fix $c$ to be the constant hidden in the $\mathcal{O}$-notation of \cref{broadcasting} and set $t\coloneqq c\cdot k^{10}.$
We initialise as follows:
Set $A_0\coloneqq V(G)$ and $F_0\coloneqq \emptyset.$
Notice that the pair $(A_0,F_0)$ satisfies the following conditions for $i=0,$ in particular since $S$ is $(3t+1,\alpha)$-well linked:
\begin{itemize}[itemsep=-2pt]
  \item $\mathsf{tw}(G[F_i])\leq 2t,$ and
  \item $\mathsf{tw}(G[A_i\cup F_i])>t.$
\end{itemize}
The process will stop as soon as it determines that $\mathsf{tw}(G[A_i\cup F_i])\leq 2t.$

Suppose we have already constructed the pair $(A_{i-1},F_{i-1})$ for some $i\geq 1$ satisfying the two conditions above while $\mathsf{tw}(G[A_{i-1}\cup F_{i-1}])>2t.$
\red{Notice that we make use of the algorithm from \cref{conquerors} to ensure that those bounds on the treewidth hold.}
Let $S_i\subseteq A_{i-1}$ be an arbitrary set of size $\left\lceil \frac{|A_{i-1}|}{2}\right\rceil.$
There are two cases to consider.
\begin{description}
  \item[Case 1:] \textit{$\mathsf{tw}(G[F_{i-1}\cup S_i])\leq 2t$}, in this case we set $F_i\coloneqq F_{i-1}\cup S_i$ and $A_i\coloneqq A_{i-1}\setminus S_i.$
  Notice that this preserves the two conditions above.
  \item[Case 2:] \textit{$\mathsf{tw}(G[F_{i-1}\cup S_i])>t$}, here we set $F_i\coloneqq F_{i-1}$ and $A_i\coloneqq S_i$ and still satisfy our conditions.  
\end{description}
Clearly, this procedure must stop after at most $O(\log(|V(G)|))$ many steps as either we reach a situation where $G[A_i\cup F_i]$ has treewidth at most $2t,$ or $A_i=\emptyset.$
It follows from our conditions above that in the later case we also have that $t< \mathsf{tw}(G[F_{i}\cup A_i]) \leq 2t.$
Hence, this algorithm is always successful.
Moreover, in each step we call \cref{conquerors}, 
\rev{72. P17, L4: I do not understand “in each step we call Proposition 4.9”. It seems to me that you only call
 \cref{conquerors} once.}
\ans{We added a sentence to explain where the algorithm is called.}
 along with minor checking on bounded treewidth graphs which can be done in time
$2^{\Ocal(t)}(k^2)^{2t}2^{\Ocal(k^2)}\cdot |V(G)|,$
 using, e.g., the minor-checking algorithm of \cite{AdlerDFST11Faster}. This leads to the final running time.
\end{proof}

Notice that \cref{counterfeit} can be combined with Reed's algorithm for finding either a tree decomposition of boundd width or a $(k,2/3)$-well-linked set \cite{reed1992finding} to produce an updated version of the original result in \cite[Proposition 3.6]{fomin2019hitting2}.

\begin{corollary} \label{contribution}
There exists a universal constant $c>1$ and an algorithm that, given an integer $k$ and a graph $G,$ either computes a tree decomposition of width at most $ck^{10}$ for $G,$ or a $k$-wall in $G$ in time $2^{\mathsf{poly}(k)}|V(G)|\log(|V(G)|).$ 
\end{corollary}

We are now ready for the proof of \cref{thm_algogrid}.
This is an algorithmic version of Lemma 14.6 from \cite{KawarabayashiTW20Quicklyexcluding}.

\begin{proof}[Proof of \cref{thm_algogrid}]
Let us set $c_1$ to be the constant from \cref{counterfeit} 
\rev{73. P17, L15: “from Theorem 4.2” should be “from Lemma 4.10”.}
\ans{fixed}
and set $t\coloneqq 6^{\red{10}}c_1k^{20}.$
Then $S$ is a $(3t+3,\alpha)$-well-linked set in $G.$
By using \cref{prohibiting} we can find an $S$-free set $F$ of size $3t+1$ in time $\mathcal{O}(t\red{|V(G)||E(G)|}).$\ans{fixed running time here}
Moreover, by \cref{differentiated} $F$ is strongly linked, and by \cref{sublimated} $\mathcal{T}_F$ is a tangle of order $3t+1$ which is a truncation of $\mathcal{T}_S.$
Indeed, we can also observe that $F$ is $(3t+1,2/3)$-well-linked in $G.$ 

We will now start iterating the following process working on the graph $G_i,$ where $G_0\coloneqq G$ while maintaining the invariant that $F$ is strongly linked in $G_i$:
\begin{enumerate}[itemsep=-2pt]
\item We use \cref{counterfeit} to find a $6k^2$-wall $W$ in $G_i.$
\rev{74. P17, L21: Delete .}
\ans{Done.}
\rev{75. P17, L21:  If you want a $6k^2$-wall, you have to define $t$ to be $6^{10}c_1k^{20}$.}
\ans{fixed}
\item We then select a representative set $X$ of the first column of $W,$ let us denote \red{this column by}  
\rev{76. P17, : What is “this path”?} 
\ans{We now defined and use lines and columns. We updated the text accordingly.}
by $P.$
Indeed, we may make sure to select $X$ to only come from the first column of $W.$
Then, we ask for a linkage of size $2k^2$ from $F$ to $X$ in $G_i.$
This request has one of two initial outcomes:
\begin{enumerate}[itemsep=-2pt]
  \item We find the linkage, or
  \item we find a separation $(A,B)$ of order at most $2k^2-1$ such that $F\subseteq A$ and $X\subseteq B.$ 
\end{enumerate}
In the first outcome we are happy, return $W$ and terminate, but in the second outcome we may call for \cref{evaluation} to either ``push'' $(A,B)$ closer to $X,$ or find an edge $e$ of \red{$G_i-e$}
\rev{77. P17, L28: $G_i-e$.} 
\ans{Fixed.}
such that $F$ is still strongly linked in $G_i-e.$
By iterating at most $|E(G_i)|$
many times, \cref{evaluation} will eventually produce such an edge.
We now set $G_{i+1}\coloneqq G_i-e$ and start the next iteration.
\rev{78.  P17, L29: $|E(G_i)|$ should be $|V(G_i)|$.}
\ans{Each iteration deletes an edge, so we have at most as many iterations as there are edges in $G_i$.}
\end{enumerate}
Notice that the above procedure stops after at most $|E(G)|$ iterations, justifying our running time.

Next, we show that, if the above process produces a $2k^2$-linkage from $F$ to $X$ this certifies that the tangle $\mathcal{T}_W$ is indeed a truncation of $\mathcal{T}_F$ which, in turn, is a truncation of $\mathcal{T}_S$ as desired.

Let $\mathcal{P}$ be the $F$-$X$-linkage of order $2k^2$ and let $X'\subseteq X$ be the set of endpoints of the paths in $\mathcal{P}.$
Recall that $P$ is the first column of $W$ and $X'\subseteq X\subseteq V(P).$

We now partition $L_1\coloneqq P$ into $k$ pairwise disjoint sub-paths $R_1,\dots,R_k$ such that $|V(R_i)\cap X'|\geq 2k$ for each $i\in[k].$
For each $i\in[k]$ let $r_i$ be the first vertex of $R_i$ encountered when traversing along $L_1$ and let $Q_i$ be the unique row of $W$ containing $r_i.$
Finally we choose $L_2,\dots,L_k$ to be $k-1$ arbitrary and distinct columns of $W$ which are all distinct from $L_1.$
Let $W'$ be the $k$-subwall of $W$ with rows $Q_1,\dots,Q_k$ and columns $L_1,\dots,L_k.$

Now all we have to do is to show that $\mathcal{T}_{W'}$ is a truncation of $\mathcal{T}_F.$
Let $(A,B)\in\mathcal{T}_F$ be a separation of order at most $k-1.$
It follows that $|A\cap F|\leq |A\cap B|\leq k-1.$
Moreover, there exists some $j\in[k]$ such that $Q_j\cup R_j$ is vertex-disjoint from $A\cap B.$
Since there are $2k$ pairwise vertex-disjoint paths from $R_j$ to $F$ by the choice of $R_j$ it follows from $(A,B)\in\mathcal{T}_F$ that $V(Q_j\cup R_j)\subseteq B\setminus A.$
Moreover, there exists some $j'\in[k]$ such that $L_{j'}$ is vertex-disjoint from $A\cap B.$
Since $L_{j'}$ has a non-empty intersection with $Q_j$ we must have that $V(Q_j\cup L_{j'})\subseteq B\setminus A$ and thus it follows that $(A,B)\in \mathcal{T}_{W'}.$

To complete the proof, observe that eventually the process above must find such a linkage as otherwise we run out of edges to delete.
\end{proof}

\subsection{$\Sigma$-decompositions}\label{subsec_sigmadec}
In this subsection we introduce the core concept that we will use to describe graphs that ``almost'' embed onto some surface.
All of these definitions can be found in \cite{KawarabayashiTW20Quicklyexcluding}, also see \cite{ThilikosW22killi} for additional explanations and illustrations.
We start by introducing the notions we will be using to describe ``almost embeddings'' in surfaces.

\red{If $\Omega$ is a cyclic permutation of some set, we denote this set by $V(\Omega).$}

\begin{definition}[Society]\label{republicans}
\rev{79. P18, L11-12: $\Omega$ is deﬁned twice.}
\ans{fixed.}
A \emph{society} is a pair $(G,\Omega),$ where $G$ is a graph and $\Omega$ is a cyclic permutation with $V(\Omega)\subseteq V(G).$
A \emph{cross} in a society $(G,\Omega)$ is a pair $(P_1,P_2)$ of disjoint paths\footnote{When we say two paths are \emph{disjoint} we mean that their vertex sets are disjoint.} in $G$ such that $P_i$ has endpoints $s_i,t_i\in V(\Omega)$ and is otherwise disjoint from $V(\Omega),$ and the vertices $s_1,s_2,t_1,t_2$ occur in $\Omega$ in the order listed.
\end{definition}
  
\begin{definition}(Drawing in a surface)\label{interchangeable}
  A \emph{drawing} \emph{in a surface $\Sigma$}  (with crossings) is a triple $\Gamma=(U,V,E)$ such that
\rev{80. P18, L-10: in $\Sigma$ with crossings ...}
\ans{Fixed.}
  \begin{itemize} 
  \item $V$ and $E$ are finite, 
  \item $V\subseteq U\subseteq\Sigma,$ 
  \item $V\cup\bigcup_{e\in E}e=U$ and $V\cap (\bigcup_{e\in E}e)=\emptyset,$ 
  \item for every $e\in E,$  $e=h((0,1)),$ where $h\colon[0,1]_{\mathbb{R}}\to U$ is a homeomorphism onto its image with $h(0),h(1)\in V,$ and
  \item if $e,e'\in E$ are distinct, then $|e\cap e'|$ is finite.
  \end{itemize}
  We call the set $V,$ sometimes referred to by $V(\Gamma),$ the \emph{vertices of $\Gamma$} and the set $E,$ referred to by $E(\Gamma),$ the \emph{edges of $\Gamma$}.
  If $G$ is graph and $\Gamma=(U,V,E)$ is a drawing with crossings in a surface $\Sigma$ such that $V$ and $E$ naturally correspond to $V(G)$ and $E(G)$ respectively, we say that $\Gamma$ is a \emph{drawing of $G$ in $\Sigma$ (possibly with crossings)}. In the case where no two edges in $E(\Gamma)$ have a common point, we say that $\Gamma$ is a \emph{drawing of $G$ in $\Sigma$ without crossings}. In this last case, the connected components of $\Sigma\setminus U,$ are the \emph{faces} of $\Gamma.$
\end{definition}
  
\begin{definition}[$\Sigma$-decomposition]\label{crystallize} 
  Let $\Sigma$ be a surface. 
  A \emph{$\Sigma$-decomposition} of a graph $G$ is a pair $\delta =(\Gamma,\mathscr{D}),$ where $\Gamma$ is a drawing of $G$ is $\Sigma$ with crossings, and $\mathscr{D}$ is a collection of closed disks, each a subset of $\Sigma$ such that 
  \begin{enumerate}
  \item the disks in $\mathscr{D}$ have pairwise disjoint interiors, 
  \item the boundary of each disk in $\mathscr{D}$ intersects $\Gamma$ in vertices only, 
  \item if $\Delta_1,\Delta_2\in\mathscr{D}$ are distinct, then $\Delta_1\cap\Delta_2\subseteq V(\Gamma),$ and 
  \item every edge of $\Gamma$ belongs to the interior of one of the disks in $\mathscr{D}.$ 
  \end{enumerate} 
  Let $N$ be the set of all vertices of $\Gamma$ that do not belong to the interior of the disks in $\mathscr{D}.$ 
  We refer to the elements of $N$ as the \emph{nodes} of $\delta .$ 
  If $\Delta\in\mathscr{D},$ then we refer to the set $\Delta\setminus N$ as a \emph{cell} of $\delta.$
  \rev{81. P18, L-2: I do not see why at most 3 nodes on the boundary are removed.}
  \ans{Indeed, at this stage the number of nodes on the boundary of a cell is unbounded, we removed the sentence.}
  We denote the set of nodes of $\delta $ by $N(\delta )$ and the set of cells by $C(\delta ).$ 
  For a cell $c\in C(\delta )$ the set of nodes that belong to the closure of $c$ is denoted by $\widetilde{c}.$ 
  Notice that this means that the cells $c$ of $\delta $ with $\widetilde{c}\neq\emptyset$ form the edges of a hypergraph with vertex set $N(\delta )$ where $\widetilde{c}$ is the set of vertices incident with $c.$ 

  For a cell $c\in C(\delta )$ we define the graph $\sigma_{\delta }(c),$ or $\sigma(c)$ if $\delta $ is clear from the context, to be the subgraph of $G$ consisting of all vertices and edges drawn in the closure of $c.$ 
 
  We define $\pi_{\delta }\colon N(\delta )\to V(G)$ to be the mapping that assigns to every node in $N(\delta )$ the corresponding vertex of $G.$
\end{definition}

\begin{definition}[Vortex]\label{conspicuous}
  Let $G$ be a graph, $\Sigma$ be a surface and $\delta =(\Gamma,\mathscr{D})$ be a $\Sigma$-decomposition of $G.$
  A cell $c\in C(\delta )$ is called a \emph{vortex} if $|\widetilde{c}|\geq  4.$
  Moreover, we call $\delta $ \emph{vortex-free} if no cell in $C(\delta )$ is a vortex.
\end{definition}
  
\begin{definition}[Rendition]\label{designation}
  Let $(G,\Omega)$ be a society, and let $\Sigma$ be a surface with one boundary component $B$ \red{homeomorphic to a circle}.
  A \emph{rendition} in $\Sigma$ of $G$ is a $\Sigma$-decomposition $\rho$ of $G$ such that the image under $\pi_{\rho}$ of $N(\rho)\cap B$ is $V(\Omega),$ mapping one of the two cyclic orders of $B$ to the order of $\Omega.$
  \rev{82. P19, L10: I do not see why there are two cyclic orders of $B$. $B$ is possibly not homeomorphic to a circle. For example, $B$ can be a union of infinitely many line segments sharing a common point. You probably have to formally define what a surface is in your paper to exclude this situation.}
  \ans{In our applications we only consider boundary components $B$ homeomorphic to circles, we added this assumption.}

  Let $\rho=(\Gamma,\mathscr{D})$ be a rendition of $(G,\Omega)$ in a disk.
If there exists a cell $c_0\in C(\red{\rho})$ such that no cell in $C(\red{\rho})\setminus \{ c_0\}$ is a vortex we call the triple 
\rev{84. P19, L14: Should “tripe” be “triple”?}
\ans{Fixed.}
$(\Gamma,\mathscr{D},c_0)$ a
\rev{83. P19, L13-14: What is $\delta$?}
\ans{fixed}
\emph{cylindrical rendition} of $(G,\Omega).$
\end{definition}

Let $\rho=(\Gamma,\mathscr{D})$ be a rendition of a society $(G,\Omega)$ in a surface $\Sigma.$
For every cell $c\in C(\rho)$ with $|\widetilde{c}|=2$ we select one of the components of $\mathsf{bd}(c)-\widetilde{c}.$
This selection is called a \emph{tie-breaker in $\rho$}, and we assume every rendition to come equipped with a tie-breaker.
Let $Q\subseteq G$ be either a cycle or a path that uses no edge of $\sigma(c)$ for every vortex $c\in C(\rho).$
We say that $Q$ is \emph{grounded} in $\rho$ if either $Q$ is a non-zero length path with both endpoints in $\pi_{\rho}(N(\rho)),$ or $Q$ is a cycle, and \red{there exist at least two distinct cells $c_1,c_2\in C(\rho)$ such that $Q$ uses edges of $\sigma(c_1)$ and $\sigma(c_2)$.}
\rev{85. P19, L19: Do you mean exactly two distinct cells or at least two distinct cells?}
\ans{We changed the phrasing to explain that we mean ``at least two''.}
A $2$-connected subgraph $H$ of $G$ is said to be \emph{grounded} if every cycle in $H$ is grounded.

If $Q$ is grounded we define the \emph{trace}\footnote{Please notice that this is called the ``track'' in \cite{KawarabayashiTW20Quicklyexcluding}.} of $Q$ as follows.
Let $P_1,\dots,P_k$ be distinct maximal sub-paths of $Q$ such that $P_i$ is a subgraph of $\sigma(\red{c_i})$ for some cell $\red{c_i}$ and $Q=\bigcup_{i\in[k]}P_k.$
Fix an index $i.$
\rev{86. P19, L22: I believe that the cell $c$ is diﬀerent for diﬀerent $i$. So you should use $c_i$ or something else to address it.}
\ans{Fixed.}
The maximality of $P_i$ implies that its endpoints are $\pi(n_1)$ and $\pi(n_2)$ for distinct nodes $n_1,n_2\in N(\rho).$
If $|\widetilde{\red{c_i}}|=2,$ define $L_i$ to be the component of $\mathsf{bd}(\red{c_i})-\{n_1,n_2\}$ selected by the tie-breaker, and if $|\widetilde{\red{c_i}}|=3,$ define $L_i$ to be the component of $\mathsf{bd}(\red{c_i})-\{n_1,n_2\}$ that is disjoint from $\widetilde{\red{c_i}}.$
Finally, we define $L_i'$ by slightly pushing $L_i$ to make it disjoint from all cells in $C(\rho).$

We define such a curve $L_i'$ for all $i$ while ensuring that the curves intersect only at a common endpoint \red{or in both endpoints in the special case where $Q$ is a cycle and $k=2$.}
\rev{87. P19, L-20: I do not understand “intersecting only at a common end- point”. If $k= 2,$ $L'_1$ and $L'_{2}$ intersect in at least two points.}
\ans{We added a remark catching this special case.}

The \emph{trace} of $Q$ is defined to be $\bigcup_{i\in[k]}L_i'.$
So the trace of a cycle is the homeomorphic image of the unit circle, and the trace of a path is an arc in $\Sigma$ with both endpoints in $N(\rho).$
  
\paragraph{Paths, linkages and transactions.}
If~$P$ is a path and~$x$ and~$y$ are vertices on~$P,$ we denote by~${xPy}$ the sub-path of~$P$ with endpoints~$x$ and~$y.$
Moreover, if~$s$ and~$t$ are the endpoints of~$P,$ and we order the vertices of~$P$ by traversing~$P$ from~$s$ to~$t,$ then~${xP}$ denotes the path~${xPt}$ and~${Px}$ denotes the path~${sPx}.$
Let~$P$ be a path from~$s$ to~$t$ and~$Q$ be a path from~$q$ to~$p.$
If~$x$ is a vertex in~${V(P) \cap V(Q)}$ such that~$Px$ and~$xQ$ intersect only in $x$ then~${PxQ}$ is the path obtained from the union of~$Px$ and~$xQ.$
\rev{88. P19, L-12: , should be .}
\ans{Fixed.}
Let~${X,Y \subseteq V(G)}.$
A path $P$ is an \emph{$X$-$Y$-path} if it has one endpoint in $X$ and the other in $Y$ and is internally disjoint from~${X \cup Y}$\red{, or $P$ consists precisely of one vertex contained in $X\cap Y$.}
\rev{89. P19, L-12: In your deﬁnition of an $X-Y$ path, the endpoints must be distinct. Do you really mean it? And if $P$ is between two distinct vertices in $X\cap Y$, is it an $X-Y$ path?}
\ans{We added a comment explaining  the special case.
}
Whenever we consider~$X$-$Y$-paths we implicitly assume them to be ordered starting in $X$ and ending in $Y,$ except if stated otherwise.
An \emph{$X$-path} is an $X$\nobreakdash-$X$\nobreakdash-path of length at least one.
In a society~$(G,\Omega),$ we write~$\Omega$-path as a shorthand for a~$V(\Omega)$-path.

Let~$(G,\Omega)$ be a society.
A \emph{segment} of~$\Omega$ is a set~${S \subseteq V(\Omega)}$ such that there do not exist~${s_1,s_2 \in S}$ and~${t_1,t_2 \in V(\Omega) \setminus S}$ such that~${s_1,t_1,s_2,t_2}$ occur in~$\Omega$ in the order listed. 
A vertex~${s \in S}$ is an \emph{endpoint} of the segment~$S$ if there is a vertex~${t \in V(\Omega) \setminus S}$ which immediately precedes or immediately succeeds~$s$ in the order~$\Omega.$
For vertices~${s,t\in V(\Omega)},$ if~$t$ immediately precedes~$s$ we define~$s\Omega t$ to be the \emph{trivial segment}~$V(\Omega),$
and otherwise we define~$s\Omega t$ to be the uniquely determined segment with first vertex~$s$ and last vertex~$t.$
\medskip

Let~$G$ be a graph.
A \emph{linkage} in~$G$ is a set of pairwise vertex-disjoint paths.
In slight abuse of notation, if~$\mathcal{L}$ is a linkage, we use~$V(\mathcal{L})$ and~$E(\mathcal{L})$ to denote~${\bigcup_{L\in\mathcal{L}}V(L)}$ and~${\bigcup_{L\in\mathcal{L}}E(L)}$ respectively.
Given two sets~$A$ and~$B$ we say that a linkage~$\mathcal{L}$ is an \emph{$A$-$B$-linkage} if every path in~$\mathcal{L}$ has one endpoint in~$A$ and one endpoint in~$B.$ 
  
We call $\abs{\mathcal{L}}$ the \emph{order} of $\mathcal{L}.$
\red{We say that two linkages $\mathcal{L}$ and $\mathcal{R}$ are \emph{vertex disjoint} if $V(\mathcal{L})\cap V(\mathcal{R})=\emptyset$.}
\medskip

Let~${(G,\Omega)}$ be a society. 
A \emph{transaction} in~${(G,\Omega)}$ is an $A$-$B$-linkage for disjoint segments~$A,B$ of~$\Omega.$ 
We define the \emph{depth} of~${(G,\Omega)}$ as the maximum order of a transaction in~${(G,\Omega)}.$
  
Let $(G,\Omega)$ be a society and let $\mathcal{P}$ be a transaction on $(G,\Omega).$
We say that $\mathcal{P}$ is a 
\begin{itemize}
\item \emph{crosscap transaction} if the members of $\mathcal{P}$ can be numbered $P_1,P_2,\dots,P_n$ and the endpoints of $P_i$ denoted by $u_i$ and $v_i$ in such a way that $$u_1, u_2,\dots, u_n, v_1, v_2,\dots,v_n$$ appear in $\Omega$ in the order listed.
We say that $\mathcal{P}$ is a crosscap transaction of \emph{thickness $n$}.
\item  \emph{handle transaction} on $(G,\Omega)$ if the members of $\mathcal{P}$ can be numbered $P_1,P_2,\dots,P_{2n}$ and the endpoints of $P_i$ can be denoted by $u_i$ and $v_i$ in such a way that 
\begin{align*}
u_1,u_2,\dots,u_{2n}v_n,v_{n-1},\dots,v_1,v_{2n},v_{2n-1},\dots,v_{n+1}
\end{align*}
appear in $\Omega$ in the order listed.
We say that $\mathcal{P}$ is a handle transaction of \emph{thickness $n$}.
\end{itemize}\medskip

Notice that a handle transaction of thickness $1$ is also a crosscap transaction of thickness $1.$
We call such a transaction a \emph{cross}.

\begin{proposition}[Two Paths Theorem,~\cite{jung1970verallgemeinerung,seymour1980disjoint,shiloach1980polynomial,thomassen19802,robertson1990graph}]\label{thm_twopaths}
  A society $(G,\Omega)$ has no cross if and only if it has a vortex-free rendition in a disk.
\end{proposition}

\paragraph{Linear decompositions.}
In the original Graph Minors Series by Robertson and Seymour, the depth of vortices was not described in terms of their largest transaction, but in terms of another type of decomposition that allows to break the graph inside a vortex into smaller pieces, all of which attach only via a small set of vertices.

Let $(G,\Omega)$ be a society.
A \emph{linear decomposition} of $(G,\Omega)$ is a sequence  $$\mathcal{L}= \langle X_1,X_2,\dots,X_n,v_1,v_2,\dots,v_n \rangle$$ such that \red{the vertices}
\rev{90. P20, L-12: $n=|V(\Omega)|$? Are $v_1, v_2, \ldots, v_n$ distinct?}
\ans{Indeed, they should be. We now specify this.}
$v_1,v_2,\dots,v_n\in V(\Omega)$ \red{are pairwise distinct and} occur in that order on $\Omega$ and $X_1,X_2,\dots,X_n$ are subsets of $V(G)$ such that
  \begin{itemize}[itemsep=-2pt]
  \setlength\itemsep{0em}
  \item  for all $i\in[n],$ $v_i\in X_i,$
  \item $\bigcup_{i\in[n]}X_i=V(G),$ 
  \item for every $uv\in E(G),$ there exists $i\in [n]$ such that $u,v\in X_i,$ and 
  \item for all $x\in V(G),$ the set $\{ i\in[n] \mid x\in X_i \}$ forms an interval in $[n].$
  \end{itemize}
The \emph{adhesion} of a linear decomposition is $\max_{i\in [n-1]}|X_i\cap X_{i+1}|,$ its \emph{width} is $\max_{i\in [n]}|X_i|.$

\begin{proposition}[\!\! \cite{KawarabayashiTW20Quicklyexcluding}]
\label{prefascist}
If a society $(G,\Omega)$ has depth at most $θ,$ then it has 
a linear decomposition of adhesion at most $θ.$
Moreover there exists an algorithm that, given  $(G,\Omega),$ either finds a transaction of order $>θ$ or outputs such a decomposition in time $\Ocal(θ|V(G)|^2).$
\end{proposition}

\begin{definition}[Vortex societies and breadth and depth of a $\Sigma$-decomposition]\label{upholstering}
  Let~$\Sigma$ be a surface and~$G$ be a graph.
  Let~${\delta  = (\Gamma,\mathscr{D})}$ be a $\Sigma$-decomposition of~$G.$
  Every vortex~$c$ defines a society~${(\sigma(c),\Omega)},$ called the \emph{vortex society} of~$c,$ by saying that~$\Omega$ consists of the vertices~$\pi_{\delta }(n)$ for~${n \in \widetilde{c}}$ in the order given by~$\Gamma.$
  (There are two possible choices of~$\Omega,$ namely~$\Omega$ and its reversal. 
  Either choice gives a valid vortex society.)
  The \emph{breadth} of~$\delta $ is the number of cells~${c \in C(\delta )}$ which are a vortex and the \emph{depth} of~$\delta $ is the maximum {depth} of the vortex societies~${(\sigma(c),\Omega)}$ over all vortex cells~${c \in C(\delta )}.$
\end{definition}

\vspace{-0mm}\subsection{Structure with respect to a wall}
To proceed, we need to combine our definition of flat walls with the idea of $\Sigma$-decompositions.
This is a necessary step to be able to relate flat renditions of walls to the drawings provided by $\Sigma$-decompositions and thus to impose additional structure onto these drawings.
Using this, we will introduce the general concept of a local structure theorem.

Let~$G$ be a graph and~${W}$ be a wall in~$G.$
We say that~$W$ is \emph{flat in a $\Sigma$-decomposition~$\delta $ of~$G$} if there exists a closed disk~${\Delta \subseteq \Sigma}$ such that

\begin{itemize} 
  \item the boundary of~$\Delta$ does not intersect any cell of~$\delta ,$ 
  \item ${\pi(N(\delta )\cap\Boundary{\Delta})\subseteq V(\Perimeter(W))},$ 
  \item if $u$ and $v$ are distinct vertices, both of which are either corners or branch vertices of $W,$ such that \red{neither corresponds to a node of $\delta$} then $u,v$ belong to distinct cells of $\delta,$
  \rev{91. 21, L14: How can $\pi^{-1}(u)\not\in N(\delta)$? The domain of $\pi$ is $N(\delta)$. Maybe you mean that $u$ is not in the image of $\pi$. But what you wrote does not suggest that.}
  \ans{Yes that's correct. We changed the phrasing.}
  \item no cell~${c \in C(\delta )}$ with~${c \subseteq \Delta}$ is a vortex, and 
  \item ${W - \Perimeter(W)}$ is a subgraph of~${\bigcup \{ \sigma(c) \mid c \subseteq \Delta \}}.$
\end{itemize}

\begin{definition}[$W$-central $\Sigma$-decomposition]\label{def:Wcentral} 
  Let~$G$ be a graph, let ${r \geq  3}$ be an integer and let~$W$ be an ${r}$-wall in~$G.$ 
  If~${(A,B)}$ is a separation of~$G$ of order at most~${r-1},$ then exactly one of the sets~${A \setminus B}$ and~${B \setminus A}$ includes the vertex set of a column and a row of~$W.$ 
  If it is the set~$A\setminus B,$ we say that~$A$ is the \emph{$W$-majority side of the separation~${(A,B)}$}; otherwise, we say that~$B$ is the $W$-majority side. 
  
  Let~$\Sigma$ be a surface and~${\delta  = (\Gamma,\mathscr{D})}$ be a $\Sigma$-decomposition of~$G.$ 
  We say that~$\delta $ is \emph{$W$-central} if there is no cell~${c \in C(\delta )}$ such that~${V(\sigma(c))}$ includes the $W$-majority side of a separation of~$G$ of order at most~${r-1}.$ 
  
  More generally, let~${Z \subseteq V(G)}$ with~${\abs{Z} \leq  r-1},$ let~$\Sigma'$ be a surface, and let~$\delta '$ be a $\Sigma'$-decomposition of~${G-Z}.$ 
  Then~$\delta '$ is a \emph{$W$-central decomposition} of~${G-Z}$ if for all separations~${(A,B)}$ of~${G-Z}$ of order at most~${r-|Z|-1}$ such that~${B \cup Z}$ is the majority side of the separation~${(A\cup Z,Z\cup B)}$ of~$G,$ there is no cell~${c \in C(\delta ')}$ such that~${V(\sigma_{\delta '}(c))}$ contains~$B.$
\end{definition}

\section{The local structure of a surface}\label{configuration}
The final piece of technology from \cite{KawarabayashiTW20Quicklyexcluding} we will need is to describe how Kawarabayashi, Thomas, and Wollan  capture a surface in their version of the local structure theorem~\cite{KawarabayashiTW20Quicklyexcluding}.

The purpose of the following definitions is  to introduce the concept of \textit{$\Sigma$-configurations} which are meant as a strengthening of $\Sigma$-decompositions in a way that makes it possible to describe how a clique minor might appear during the inductive process of constructing the surface in the proof of \cref{assemblies} below.

\begin{definition}[Nest]\label{generating} 
  Let $\rho=(\Gamma,\mathscr{D})$ be a rendition of a society $(G,\Omega)$ in a surface $\Sigma$ and let $\Delta\subseteq\Sigma$ be an arcwise connected set. 
  A \emph{nest in $\rho$ around $\Delta$ of order $s$} is a sequence $\mathcal{C}=(C_1,C_2,\dots,C_s)$ of disjoint cycles in $G$ such that each of them is grounded in $\rho,$ and the trace of $C_i$ bounds a closed disk $\Delta_i$ in such a way that $\Delta\subseteq\Delta_1\subsetneq\Delta_2\subsetneq\dots\subsetneq\Delta_s\subseteq\Sigma.$
\end{definition}

Let $(G,\Omega)$ be a society \red{with $V(\Omega)\neq\emptyset$} and $\rho=(\Gamma,\mathscr{D})$ be a rendition of $(G,\Omega)$ in a surface $\Sigma.$
Let $C$ be a cycle in $G$ that is grounded in $\rho$ such that the trace $T$ of $C$ bounds a closed disk $D$ in $\Sigma.$
Let $L$ be the subgraph of $G$ with vertex set $\pi_{\rho}(N(\rho)\cap T)$ and no edges.
We define the \emph{outer graph of $C$ in $\rho$} as the graph
\begin{align*}
  L\cup\bigcup_{c\in C(\rho)\text{ and }c\not\subseteq D}\sigma(c). 
\end{align*}
We define the \emph{inner graph of $C$ in $\rho$} as the graph
\begin{align*}
  L\cup\bigcup_{c\in C(\rho)\text{ and }c\subseteq D}\sigma(c).
\end{align*}
Finally, we define the inner society of $C$ in $\rho$ as follows.
Let $\Omega'$ be the cyclic permutation of $V(L)$ obtained by traversing $T$ clockwise.
If $G'$ is the inner graph of $C$ in $\rho,$ then the \emph{inner society of $C$ in $\rho$} is the society $(G',\Omega').$

Let $\mathcal{C}=\{ C_1,\dots,C_s\}$ be a collection of pairwise vertex-disjoint cycles and $\mathcal{P}=\{ P_1,\dots,P_k\}$ be a linkage.
We say that $\mathcal{P}$ is \emph{orthogonal} to $\mathcal{C}$ if every $P_i$ has a non-empty intersection with each cycle of $\mathcal{C}$ such that there exists an endpoint $p_i$ of $P_i$ for which the cycles $C_1,\dots,C_s$ appear in the order listed when traversing along $P_i$ starting from $p_i,$ and $P_i\cap C_j$ is a (possibly trivial) sub-path of $P_i$ for all $i\in[k],$ $j\in[s].$

\begin{figure}
  \centering
\includegraphics[scale=1]{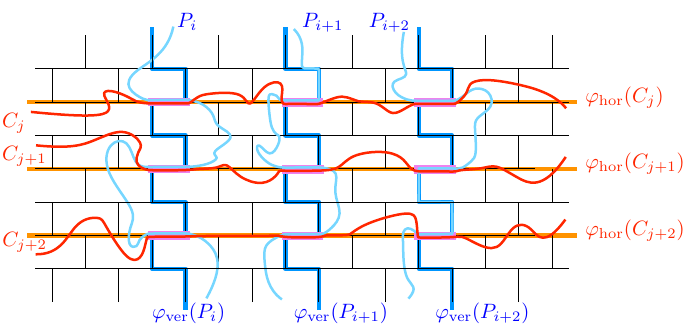}
    \caption{An illustration of a pair $(\mathcal{C},\mathcal{P})$ overlaying a wall $W.$}
    \label{fig_overlay}
\end{figure}

Let $G$ be a graph, $W$ a wall in $G,$ $\mathcal{C}$ a nest in $G,$ and $\mathcal{P}$ a linkage which is orthogonal to $\mathcal{C}.$

We say that $(\mathcal{P},\mathcal{C})$ \emph{overlays $W$} if we can fix the rows of 
$W$ to be $H_1,\dots, H_k$ and the columns $V_1,\dots, V_{\ell}$ such that there exist injective maps $\phi_{\text{hor}}\colon\mathcal{C}\to\{ H_1,\dots, H_k\}$ and $\phi_{\text{vert}}\colon\mathcal{P}\to\{ V_1,\dots, V_{\ell}\}$ such that for all $P\in\mathcal{P}$ and $C\in\mathcal{C},$ $\phi_{\text{hor}}(C)\cap\phi_{\text{vert}}(P)\subseteq P\cap C.$
See \cref{fig_overlay} for an illustration.

Let $C$ be some cycle grounded in a $\Sigma$-decomposition $\delta $ of a graph $G.$
We say that a linkage $\mathcal{L}$ is \emph{coterminal with a linkage $\mathcal{P}$ up to level $C$} if there exists $\mathcal{P}'\subseteq \mathcal{P}$ such that $H\cap \red{\bigcup}\mathcal{L}=H\cap \mathcal{P'}$ where $H$ denotes the \red{outer graph} of $C.$
\rev{92. 22, L-9: I do not understand the notation $H\cap \Lcal$. $H$ is a graph and $\Lcal$ is a set of paths.}
\ans{ We now specify that we mean $\bigcup\mathcal{L}$, i.e., the 
 union of the paths of the linkage. }
\rev{93. P22, L-9: The outer graph of $C_i$ s ambiguous when $\Sigma$ is a sphere with no boundary. }
\ans{This only happens if $V(\Omega)=\emptyset$, otherwise by the definition of renditions, $\Sigma$ is required to have a boundary homeomorphic to a circle. We added the non-empty assumption to the definition.}
When $\mathcal{C}=\{ C_1,\dots,C_{\ell}\}$ is a nest in $\delta ,$ and it is clear from the context to which nest we are referring, we will abbreviate ``coterminal with $\mathcal{P}$ up to level $C_i$'' by ``coterminal with $\mathcal{P}$ up to level $i$''.

\begin{definition}[$\Sigma$-Configuration]
  Let $G$ be a graph, let $\Sigma=\Sigma^{(\mathsf{h},\mathsf{c})}$ be a surface, for some $\mathsf{h},\mathsf{c}\in \Nbbb,$ and let $\delta =(\Gamma,\mathscr{D})$ be a $\Sigma$-decomposition of $G.$
  Let $W$ be a wall in $G,$ $W_1$ a subwall of $W,$ and $W_0$ an $r$-subwall of $W_1.$
  Assume that there exist cells $c_0,c_1\in C(\delta )$ and a closed disk $\Delta_0\subseteq\Sigma$ such that:
  \begin{itemize}
  \item No cell in $C(\delta )\setminus\{ c_0,c_1\}$ is a vortex.
  \item The vortex society $(G_0,\Omega_0)$ of $c_0$ in $\delta $ has a vortex-free rendition, $W_0\subseteq G_0,$ and there exists a vortex free rendition of $(G_0,\Omega_0)$ for which $W_0$ is flat.
  \item The cell $c_0\subseteq\Delta_0\subseteq \Sigma\setminus c_1$ and the boundary of $\Delta_0$ intersect no cell of $\delta .$
  \item The complementary society $(G_2,\Omega_2)$ to the vortex society of $c_0$ has a cylindrical rendition $\rho_2=(\Gamma_2,\mathscr{D}_2,c_2)$ with a nest $\mathcal{C}_2,$ where $\Gamma_2=\Gamma,$ $C(\rho_2)\setminus c_2=\{ c\in C(\delta ) \mid c\subseteq\Delta_0 \}\setminus\{ c_0\},$ $N(\rho_2)=N(\delta )\cap\Delta_0,$ $\sigma_{\rho_2}(c_2)=\bigcup_{c\not\subseteq\Delta_0}\sigma_{\delta }(c),$ and $\sigma_{\delta }(c)=\sigma_{\rho_2}(c)$ for every $c\in C(\rho_2)\setminus\{ c_2\}.$
  \item The vortex society of $c_1$ in $\delta $ has a cylindrical rendition $\rho_1=(\Gamma_1,\mathscr{D}_1,c_1')$ with a nest $\mathcal{C}_1.$
  \end{itemize}
  Let $\widetilde{\Delta_0}$ denote the set of vertices of $N(\rho)$ on the boundary of $\Delta_0.$
  Let $X_1,X_2$ be two complementary segments of $\Omega_2.$
  Assume that there exists a set $\gamma=\{ (\mathcal{L}_i,X_i^2) \mid i\in[\mathsf{h}+\mathsf{c}] \}$ and \red{vertex-}disjoint linkages
  \rev{94.  t P23, L7: What does “disjoint linkages $P_1$ and $P_2$” mean? A linkage is a set. So two linkages are disjoint mean that no path in $P_1$ is a path in $P$. But I guess you mean that every path in $P_1$ is disjoint from every path in $P$.}
  \ans{We added a definition for vertex-disjoint linkages to the definition of linkages and changed the quantification to requiring the linkages to be vertex disjoint.}
  $\mathcal{P}_1,\mathcal{P}_2$ in $(G_2,\Omega_2)$ such that:
  \begin{itemize}
  \item $\mathcal{P}_1$ is a linkage from $X_1$ to $\sigma_{\rho_1}(c_1')$ which is orthogonal to both $\mathcal{C}_1$ and $\mathcal{C}_2$ and $(\mathcal{P}_1,\mathcal{C}_1)$ {overlays} $W_1.$ 
  \item The linkage $\mathcal{P}_2$ is a linkage from $X_2$ to $\widetilde{\Delta_0}$ which is orthogonal to $\mathcal{C}_2$ and $(\mathcal{P}_1\cup\mathcal{P}_2,\mathcal{C}_2)$ overlays $W_1.$
  \item For $i\in[\mathsf{h}+\mathsf{c}],$ $\mathcal{L}_i$ is a transaction on $(G_2,\Omega_2)$ and $X^2_i\subseteq X_2$ is a segment of $\Omega_2.$
  The transactions $\mathcal{L}_1,\dots,\mathcal{L}_{\mathsf{h}+\mathsf{c}}$ are pairwise disjoint with endpoints in $X_i^2$  {and},  {moreover,}  $X_1^2,\dots,X^2_{\mathsf{h}+\mathsf{c}}$ are also pairwise disjoint.
  For all $i,$ $\mathcal{L}_i$ is coterminal with $\mathcal{P}_2$ up to level $1$ of $\mathcal{C}_2$ in $\rho_2$ and $\mathcal{L}_i$ is disjoint from $\sigma_{\delta }(c_1)$ and disjoint from $\mathcal{P}_1.$
  \item The transactions $\mathcal{L}_1,\dots,\mathcal{L}_{\mathsf{h}}$ are handle transactions, which are all of thickness $θ,$ and  the transactions $\mathcal{L}_{\mathsf{h}+1},\dots,\mathcal{L}_{\mathsf{h}+\mathsf{c}}$ are crosscap transactions, each of thickness $θ.$
  \item Let $\mathcal{P}_1=\{ P_1,\dots,P_m\},$ for $i\in[m]$ let $x_i$ be the endpoint of $P_i$ in $X_1$ and let $y_i$ be the last vertex of $P_i$ in $\widetilde{c_1}.$
  If $x_1,\dots,x_m$ appear in $\Omega_2$ in the order listed, then $y_1,\dots,y_m$ appear on $\widetilde{c_1}$ in the order listed.
  \end{itemize}
  In those circumstances (neglecting the wall $W$) we say that $(W_1,W_0,\delta ,c_0,c_1,\Delta_0,\mathcal{C}_1,\mathcal{C}_2,\mathcal{P}_1,\mathcal{P}_2,\gamma)$ is a \emph{$\Sigma$-configuration} with parameters $(r,|\mathcal{C}_1|,|\mathcal{C}_2|,|\mathcal{P}_1|,θ,\mathsf{h},\mathsf{c}).$
  The cell $c_1$ is called the \emph{exceptional cell of $\delta $}.
\end{definition}

For an illustration of a $\Sigma$-configuration consider (a), (b), and (c) in \cref{autoritarer}.

\begin{figure}[ht]
  \begin{center}
  \scalebox{.75}{\includegraphics{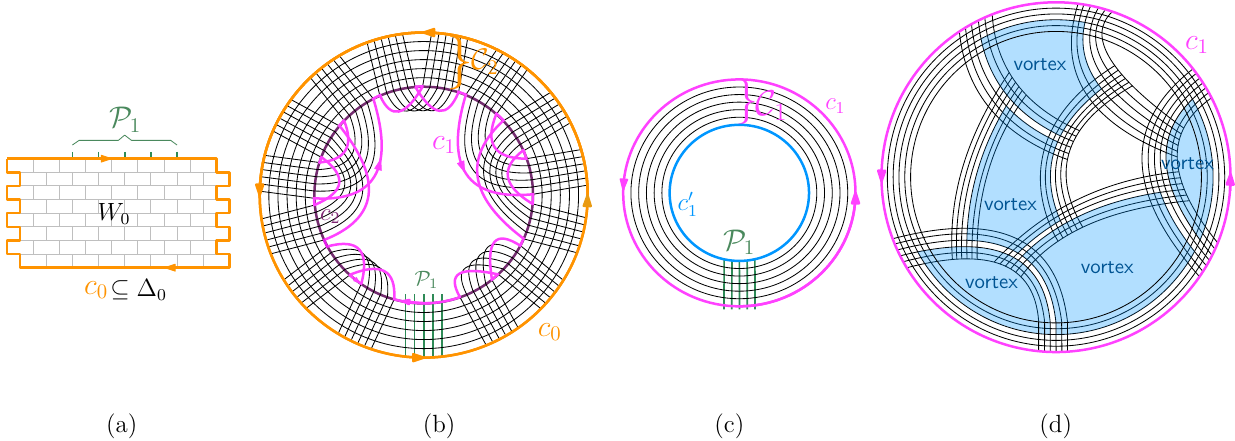}}
  \end{center}
  \captionsetup{singlelinecheck=off}
  \caption[]{A $\Sigma$-configuration of a graph (a), (b), and (c)) together with a rendition of bounded breadth and depth of its exceptional cell (d)).
  \begin{enumerate}
    \item[(a)] The vortex-free {rendition} of the vortex society of $c_0$ together with the wall $W_0$ which is grounded in this rendition. On the boundary of $c_0$ one can find one endpoint of every path in $\mathcal{P}_1.$ Notice that each such endpoint coincides with a branch vertex of $W_0.$
   \item[(b)] A rendition of the complementary society $(G_2,\Omega_2)$ to the vortex society of $c_0.$ This rendition contains the nest $\mathcal{C}_2$ of the special cell $c_2,$ a part of the linkage $\mathcal{P}_1$ which connects $W_0$ to the boundary of $c_1',$ the transactions $\mathcal{L}_i$ which are either handle or crosscap transactions on $(G_2,\Omega_2)$ and reach into the interior of $c_2.$ The linkage $\mathcal{P}_2$ ({not highlighted}) {consists} of the sub-paths of the members of the $\mathcal{L}_i$ connecting the boundary of $c_1$ to the boundary of $c_2.$
  Finally, the figure depicts the exceptional cell $c_1.$
  \item[(c)] The cylindrical rendition of the vortex society of $c_1,$ its vortex $c_1',$ its nest $\mathcal{C}_1,$ and the remaining part of $\mathcal{P}_1.$
  \item[(d)] An alternative rendition of the vortex society of $c_1,$ this time a rendition of bounded breadth and depth where each vortex is surrounded by a private nest and each such nest together with its vortex lives in a private disk. This rendition is guaranteed to exist due to \cref{assemblies}.
  \end{enumerate}}
  \label{autoritarer}  
\end{figure}\smallskip

Let $G$ be a graph and $W$ be a $n$-wall in $G.$
we say that $G$ has a \emph{$K_t$-minor grasped by $W$} if there exist connected and pairwise vertex-disjoint subgraphs $X_1,\dots X_t$ such that for all $i\neq j\in[t]$ there exists an edge with one endpoint in $X_i$ and the other in $X_j$ and for every $i\in[t],$ there exist $t$ distinct rows $R_{x^i_1},\dots,R_{x^i_t}$ and columns $C_{y^i_1},\dots,C_{y^i_t},$ $x^i_j,y^i_j\in[n]$ for all $j\in[t]$ such that $R_{x^i_j}\cap C_{y^i_j}\subseteq X_i$ for all $j\in[t].$

\begin{proposition}[Theorem 11.4, \cite{KawarabayashiTW20Quicklyexcluding}]\label{assemblies}
  Let $r,s_2,M,t\geq  1$ be integers with $M\geq  t,$ $s_2\geq  2t.$
  Let $\mu =\Floor{\frac{1}{10^5t^{26}}t^{10^7t^{26}}\cdot M}.$
  Let
  \begin{align*}
  R=49152t^{24}(r+2s_2)+\mu .
  \end{align*}
  Let $G$ be a graph, and let $W$ be an $R$-wall in $G.$
  Then either $G$ contains a $K_t$-minor grasped by $W,$ or there exist integers $\mathsf{h},\mathsf{c}\geq  0$ with $\mathsf{h}+\mathsf{c}\leq \Choose{t+1}{2},$ a set $A\subseteq V(G)$ of size at most $\mu ,$ a surface $\Sigma = \Sigma^{(\mathsf{h},\mathsf{c})},$ and a $\Sigma$-configuration $(W_1,W_0,\delta ,c_0,c_1,\Delta_0,\mathcal{C}_1,\mathcal{C}_2,\mathcal{P}_1,\mathcal{P}_2,\gamma)$ of $G-A$ with parameters $(r,M,s_2,M,M,\mathsf{h},\mathsf{c}).$

  Moreover, the vortex society of $c_1$ has a rendition $\rho_1=(\Gamma_1,\mathscr{D}_1)$ of breadth at most $2t^2$ and depth at most $\mu $ and such that
  \begin{itemize}
  \item for each vortex cell $c\in C(\rho_1)$ there exists a nest $\mathcal{C}_c$ of order $M$ in $\rho_1$ around the unique disk $\Delta\in\mathscr{D}$ corresponding to $c,$ and 
  \item for each vortex cell $c\in C(\rho_1)$ let $\Delta_c\subseteq\Sigma$ be the disk bounded by the trace of $C_M\in\mathcal{C}_c,$ then for each pair of distinct vortex cells $c,c'\in C(\rho_1)$ we have that $\Delta_c\cap \Delta_{c'}=\emptyset.$
\end{itemize}
Moreover, there exists a function $g\colon\mathbb{N}^3\to\mathbb{N}$ and an algorithm that finds one of these outcomes in time $g(t,r,s_2)|V(G)|^2,$ where $g(t,r,s_2)\in 2^{\poly(t+r+s_2)}.$
\end{proposition}

For an illustration of the outcome of \cref{assemblies} see (a), (b), and (d) in \cref{autoritarer}.
Please notice that the structure of the nests for the vortices of $\rho_1$ is not made explicit in \cite{KawarabayashiTW20Quicklyexcluding}.
However, the existence of these nests follows from the proofs as explained, for example, in \cite{ThilikosW22killi}.

\section{The local structure theorem}\label{quintessence}
We are finally ready to prove the local version, that is a theorem with respect to a large wall, of our structure theorem.
To do this, we first show how to extract a Dyck-wall of chosen order from a $\Sigma$-configuration.
Then we use this to derive a variant of the local Graph Minors Structure Theorem in the case where we exclude Dyck-grid for some fixed surface as a minor.

\subsection{Extracting a Dyck-wall}

This section is dedicated to the extraction of a Dyck-wall from a given $\Sigma$-configuration.

\begin{lemma}\label{pindarischen}
  Let $\mathsf{h},\mathsf{c},t,r,s_2$ be non-negative integers.
  In case $\mathsf{c}=0$ let $\mathsf{c}_0=\mathsf{h}_{0}=0,$ if $\mathsf{c}=2\red{p}+1$ \red{for some integer $p$}
  \rev{95. P25, L1: What is $p$? Is $p$ an integer?}
  \ans{Yes, we now clarify this.}
  let $\mathsf{c}_0=1$ and $\mathsf{h}_{0}=p,$ and if $\mathsf{c}=2p\neq 0$ let $\mathsf{c}_0=2$ and $\mathsf{h}_{0}=p-1.$
  Finally, let $g=2\mathsf{h}+\mathsf{c}.$
  Assume that
  \begin{align*}
  r\geq  16t\text{, }M \geq  (162^{2g}+4)t,\text{ and }s_2\geq  (162^{2g}+4)t.
  \end{align*}
  Let $G$ be a graph that has a $\Sigma$-configuration $(W_1,W_0,\delta ,c_0,c_1,\Delta_0,\mathcal{C}_1,\mathcal{C}_2,\mathcal{P}_1,\mathcal{P}_2,\gamma)$ where the parameters are $(r,M,s_2,M,M,\mathsf{h},\mathsf{c})$ and where $\Sigma=\Sigma^{(\mathsf{h},\mathsf{c})}$ then $G$ contains an $\mathscr{D}^{(\mathsf{h}+\mathsf{h}_{0},\mathsf{c}_0)}_t$-expansion which is grounded in $\delta .$
\end{lemma}
  
\begin{proof}
  The claim follows quickly from the assumptions on the numbers and the definition of $\Sigma$-configurations.
  
  Let us first describe how to find a mixed surface grid in the $\Sigma$-configuration. Consider the $\Sigma$-configuration $(W_1,W_0,\delta ,c_0,c_1,\Delta_0,\mathcal{C}_1,\mathcal{C}_2,\mathcal{P}_1,\mathcal{P}_2,\gamma).$
  \rev{96. P25, L8: Should $o_1$ and $o_2$ be $c_1$ and $c_2$?}
  \ans{Indeed! We changed it.}

  In case $\mathsf{h}+\mathsf{c}=0$ we may conclude immediately as $W_0$ is an $r$-wall where $r\geq  16t.$
  Hence, $W_0$ contains the $(t,4t)$-cylindrical grid, and therefore $\mathscr{D}_{t}^{(0,0)},$ as a minor.
  Moreover, $W_0$ is grounded in $\delta .$

  So from now on we may assume that $\mathsf{h}+\mathsf{c}\geq  1.$
  We first collect the cycles necessary for the cylindrical grid which will host the crosscap and handle transactions.
  Let us number the cycles of $\mathcal{C}_2$ as $C_1,C_2,\dots,C_{s_2}$ as in the definition of a nest.
  Let $H_0$ be the subgraph of $G$ consisting of the union of the cycles $C_1,\dots,C_{162^{2g}t}.$
  For each $i\in[\mathsf{h}+\mathsf{c}]$ let $Q_i$ be the subgraph of $G$ obtained as follows.
  For each $L\in\mathcal{L}_i$ there exist paths $P_1,P_2\in\mathcal{P}_1$ such that $L$ is coterminal with $\{ P_1,P_2\}$ up to level $162^{2g}t.$
  For each $j\in[2],$ let $x_j$ be the last vertex of $P_j$ on $C_{162^{2g}t}$ when traversing along $P_j$ starting on $V(\Omega_2).$
  Moreover, let $y_j$ be the last vertex of $P_j$ in $C_1\cap L$ when traversing along $P_j$ starting on $x_j$ and moving towards $\sigma_{\rho_1}(c_1').$
  Let $L'$ be the sub-path of $L$ with endpoints $y_1$ and $y_2.$
  Finally, we define $\hat{L}$ to be the path $x_1P_1y_1L'y_2P_2x_2.$
  Observe, that both $P_1$ and $P_2$ are orthogonal with $\{ C_1,\dots,C_{162^{2g}t}\},$ and $L'$ is internally disjoint from $H_0.$
  Then $Q_i$ is defined as the union of the paths $\hat{L}$ for all $L\in\mathcal{L}_i.$
  Now, let $H_1$ be the union of $H_0$ with the graphs $Q_i$ for all $i\in[\mathsf{h}+\mathsf{c}].$

  Observe that $H_1$ contains a mixed surface grid of order $162^{2g}t$ with $\mathsf{h}$ handles and $\mathsf{c}$ crosscaps  as a minor since each of the linkages $\mathcal{L}_i$ has order at least $(162^{2g}+4)t$ which allows forgetting $4t$ of the consecutive $L'$-paths.

  Finally, notice that every cycle in $\mathcal{C}_2$ is grounded in $\delta .$
  Moreover, any path $\hat{L}$ as above closes two cycles with $C_{162^{2g}t},$ each of which is grounded in $\delta .$
  It follows that $H_1$ is grounded in $\delta $ and thus, the claim now follows from \cref{diskussionen}.
\end{proof}

As a direct corollary of \cref{pindarischen}, if one applies the lemma with parameter $t=2k,$ then inside the $\mathscr{D}_{2k}^{\mathsf{h}+\mathsf{h}_{0},\mathsf{c}_0}$-expansion one can find a corresponding Dyck-wall of order $k$ as a subgraph.
  
We now have everything in place to prove the local structure theorem for excluding Dyck-grids.
The statement we prove here is more general than the one required for our main theorem as we want to be able to have more fine-tuned control over the excluded surfaces.
That is, in the following we fix an Euler-genus $g_1=2\mathsf{h}_1$ and an Euler-genus $g_2=\mathsf{c}_2+2\mathsf{h}_2$ where $\mathsf{c}_2\in[2]$ and $g_1$ represents a surface of orientable genus $\mathsf{h}_1$ while $g_2$ represents a non-orientable surface.
To ensure that the two surfaces are incomparable, we also require $0\leq  \mathsf{h}_2<\mathsf{h}_1$ since otherwise we could find a Dyck-grid with $\mathsf{h}_1$ handles as a minor in any Dyck-grid with $\mathsf{h}_2$ handles and $\mathsf{c}_2$ crosscaps.

\begin{theorem}\label{transposition}
  Let $k,\mathsf{h}_1,\mathsf{h}_2,c_2$ be non-negative integers where $\mathsf{c}_2\in[0,2]$ and $0\leq  \mathsf{h}_2<\mathsf{h}_1.$
  \rev{97.  P25, L-6: Is “if $c_{2}\neq 0$” only used when deﬁning $g_2$ or is it also used when deﬁning $g_1$?}
  \ans{This only matters for $g_2$, we slightly rephrased to make it more clear.}
  Let $g_1\coloneqq 2\mathsf{h}_1$.
  \red{Also, if $\mathsf{c}_2\neq 0$ let $g_2\coloneqq \mathsf{c}_2+2\mathsf{h}_2$ and let $g_2\coloneqq 0$ otherwise.}
  Moreover, let $g\coloneqq\max\{ g_1,g_2\}$ and fix $t\coloneqq (162^{2g}+4)k.$
  
  Define
  \begin{align*}
  \mathbf{a}_{g_1,g_2} &\coloneqq  t^{10^7t^{26}}\\
  \mathbf{v}_{\mathsf{b},{g_1,g_2}} &\coloneqq 2t^2\\
  \mathbf{v}_{\mathsf{d},{g_1,g_2}} &\coloneqq t^{10^7t^{26}}\\
  \mathbf{R}_{g_1,g_2} &\coloneqq 49152t^{24}(r+16t)+\left\lceil\frac{5}{2}t^{10^7t^{26}}\right\rceil.
  \end{align*}

Then, given an integer $r\geq t,$ graph $G,$ and an $R$-wall $W$ with $R\geq \mathbf{R}_{g_1,g_2},$ one of the following holds.
\begin{enumerate}
  \item \red{If $g_2 \neq 0$ then $G$ contains a $\mathscr{D}^{(\mathsf{h}_1,0)}_k$\red{-expansion} or a $\mathscr{D}^{(\mathsf{h}_2,\mathsf{c}_2)}_k$-expansion $D$ such that $\mathcal{T}_D\subseteq \mathcal{T}_W$.
  If otherwise $g_2=0$ and $G$ contains a $\mathscr{D}^{(\mathsf{h}_1,0)}_k$\red{-expansion} $D$ such that $\mathcal{T}_D\subseteq \mathcal{T}_W$.}
  \item \red{Otherwise there exist non-negative integers $\mathsf{h},\mathsf{c}$ where $\mathsf{c}\in[0,2]$ and a surface $\Sigma=\Sigma^{(\mathsf{h},\mathsf{c})}$ such that there exists a positive integer $i$ for which neither $\mathscr{D}^{(\mathsf{h}_1,0)}_i$ nor, in case $\mathsf{c}_2\neq0,$ $\mathscr{D}^{(\mathsf{h}_2,\mathsf{c}_2)}_i$ is embeddable in $\Sigma$.
  Moreover, there exists a set $A\subseteq V(G)$ with $|A|\leq \mathbf{a}_{g_1,g_2},$ a $W$-central $\Sigma$-decomposition $\delta$ of $G-A$ of breadth at most $\mathbf{v}_{\mathsf{b},{g_1,g_2}}$ and depth at most $\mathbf{v}_{\mathsf{d},{g_1,g_2}},$ and an $r$-subwall $W'$ of $W-A$ which is flat in $\delta.$}
\end{enumerate}
Moreover, there exists a function $g\colon\mathbb{N}^4\to\mathbb{N}$ and an algorithm that finds one of these outcomes in time $g(k,\mathsf{h}_1,r)|V(G)|^2)$ where $g(k,\mathsf{h}_1,r)=O(r\cdot 2^{2^{\Ocal(\mathsf{h}_1)}\poly(k)}).$
\end{theorem}

  \rev{98. P26, L2: Do you mean $D^{\mathsf{h}_1,0}_k$  or $D^{\mathsf{h}_1,0}_k$-expansion?}
  \ans{In outcome 1. we mean expansions, we clarified this now.}
\rev{99. P26, L2-3: This sentence should be split into two sentences, such as “if $g_2\neq 0$, then $G$ contains ..., and if $g_2= 0$, then $G$ contains ...”.}
\ans{We split the sentence.}
\rev{100. The structure of the sentence in Outcome 2 is too complicated. It should be split into more sentences. For example, I do not understand “neither $D^{h_1,0}_{i}$ nor, in case ...”. And “embeds” might be “be embeddable in $\Sigma$}
\ans{We split the sentences in an attempt to make it more readable.}

To prove \cref{transposition} we start by showing that we can always either find a Dyck-grid of order $k$ which corresponds to one of the two fixed Euler-genus $g_1$ and $g_2,$ or we find a $\Sigma$-configuration whose parameters depend only on $k,$ $g_1,$ and $g_2$ where the Euler-genus of $\Sigma$ is strictly less than $g_1$ if $\Sigma$ is orientable, and strictly less than $g_2$ if $\Sigma$ is non-orientable and $g_2\neq 0$ or at most $g_1$ if $\Sigma$ is non-orientable but $g_2=0.$

\begin{lemma}\label{intermingling}
  Let $k,\mathsf{h}_1,\mathsf{h}_2,\mathsf{c}_2$ be non-negative integers where $\mathsf{c}_2\in[0,2]$ and $0\leq  \mathsf{h}_2<\mathsf{h}_1.$
  Let $g_1\coloneqq 2\mathsf{h}_1.$
  \red{Also, if $\mathsf{c}_2\neq 0$ let $g_2\coloneqq \mathsf{c}_2+2\mathsf{h}_2$ and let $g_2\coloneqq 0$ otherwise.}
  Moreover, \red{let $g\coloneqq\max\{ g_1,g_2\}$ and} let $t\geq  1$ be an integer.

  Finally, let $r,s_2,M\geq  1$ be integers with $r\geq  16t,$ $s_2\geq  2M\geq \max\{ 2(162^{2g}+4)k, 2t\},$ and set $\mu \coloneqq \left\lfloor \frac{1}{10^5t^{26}}t^{10^7t^{26}}\cdot M\right\rfloor.$
  Fix
  \begin{align*}
  R\coloneqq 49152t^{24}(r+2s_2)+\mu.
  \end{align*}

  Let $G$ be a graph, and let $W$ be an $R$-wall in $G.$
  
  Then one of the following is true.
  \begin{enumerate}
  \item $G$ contains a $K_t$-minor grasped by $W,$

  \item $G$ contains an expansion $D$ of \textsl{(a)} $\mathscr{D}^{(\mathsf{h}_1,0)}_k$ or \textsl{(b)} $\mathscr{D}^{(\mathsf{h}_2,\mathsf{c}_2)}_k,$ where outcome (b) is allowed if and only if $\mathsf{c}_2\neq 0,$ such that $\mathcal{T}_D\subseteq \mathcal{T}_W,$ or
  
  \item there exist integers $\mathsf{h},\mathsf{c}\geq  0$ and a surface $\Sigma=\Sigma^{(\mathsf{h},\mathsf{c})}$ where neither $\mathscr{D}^{(\mathsf{h}_1,0)}_i$ nor, in case $\mathsf{c}_2\neq0,$ $\mathscr{D}^{(\mathsf{h}_2,\mathsf{c}_2)}_i$ embeds for all $i\in\mathbb{N},$ such that there is a set $A\subseteq V(G)$ of size at most $\mu $ and a $\Sigma$-configuration $(W_1,W_0,\delta ,o_0,o_1,\Delta_0,\mathcal{C}_1,\mathcal{C}_2,\mathcal{P}_1,\mathcal{P}_2,\gamma)$ of $G-A$ with parameters $(r,M,s_2,M,M,\mathsf{h},\mathsf{c}).$

  Moreover, the vortex society of $o_1$ has a rendition $\rho_1$ of breadth at most $2t^2$ and depth at most $\mu$ such that
  \begin{itemize}
  \item for each vortex cell $c\in C(\rho_1)$ there exists a nest $\mathcal{C}_c$ of order $M$ in $\rho_1$ around the unique disk $\Delta\in\mathscr{D}$ corresponding to $c,$ and 
  \item for each vortex cell $c\in C(\rho_1)$ let $\Delta_c\subseteq\Sigma$ be the disk bounded by the trace of $C_M\in\mathcal{C}_c,$ then for each pair of distinct vortex cells $c,c'\in C(\rho_1)$ we have that $\Delta_c\cap \Delta_{c'}=\emptyset.$
\end{itemize}
  \end{enumerate}
Moreover, there exists a function $g\colon\mathbb{N}^4\to\mathbb{N}$ and an algorithm that finds one of these outcomes in time $g(k,\mathsf{h}_1,r)\cdot |V(G)|^2)$ where $g(k,\mathsf{h}_1,r)=O(r\cdot 2^{2^{\Ocal(\mathsf{h}_1)}\poly(k)}).$
\end{lemma}

\begin{proof}
We use \cref{assemblies} to either obtain a $K_t$-minor grasped by $W$ or a set $A\subseteq V(G)$ of size at most $\mu $ together with integers $\mathsf{h},\mathsf{c}\leq  \Choose{t+1}{2},$ a surface $\Sigma$ of Euler-genus $2\mathsf{h}+\mathsf{c},$ and a $\Sigma$-configuration $(W_1,W_2,\delta ,o_0,o_1,\Delta_0,\mathcal{C}_1,\mathcal{C}_2,\gamma)$ of $G-A$ with parameters $(r,M,s_2,M,M,\mathsf{h},\mathsf{c})$ such that the vortex society of $c_1$ in $\delta $ has depth at most $\mu $ and breadth at most $2t^2$ together with the desired properties for the nests.

In case we find the $K_t$-minor, we are done immediately.
\medskip

In case $\mathsf{c}=2n\neq 0,$ let $\mathsf{h}'\coloneqq \mathsf{h}+n-1$ and $\mathsf{c}'\coloneqq 2,$ similarly, if $\mathsf{c}=0$ we set $\mathsf{h}'\coloneqq \mathsf{h}$ and $\mathsf{c}'\coloneqq \mathsf{c}=0.$
Moreover, if $\mathsf{c}=2n+1,$ let $\mathsf{h}'\coloneqq \mathsf{h}+n$ and $\mathsf{c}'\coloneqq 1.$
By \cref{pindarischen} we find $D\coloneqq\mathscr{D}^{(\mathsf{h}',\mathsf{c}')}_{k}$ as a minor of $G.$

Observe that, by \cref{territorial2}, if $K_t$ embeds in $\Sigma^{(\mathsf{h},\mathsf{c})},$ then we find our $K_t$-minor grasped by $W$ within $D.$
Moreover, if $\mathscr{D}^{(\mathsf{h}_1,0)}_k$ embeds in $\Sigma^{(\mathsf{h},\mathsf{c})},$ then $D$ contains $\mathscr{D}^{(\mathsf{h}_1,0)}_k$ as a minor in a way that lets their tangles agree.
So the case where $\mathsf{c}'\neq 0$ but $\mathsf{c}_2=0$ is closed.

What follows is an analysis on the numbers $\mathsf{h}', \mathsf{c}',\mathsf{c},$ and $\mathsf{h}$ in relation to $g_1=2\mathsf{h}_1$ and $g_2=2\mathsf{h}_2+\mathsf{c}_2.$

Suppose $\mathsf{h}'\geq  \mathsf{h}_1,$ then $D$ contains $\mathscr{D}_k^{(\mathsf{h}_1,0)}$ as a minor, and we are done.
So we may assume $\mathsf{h}'<\mathsf{h}_1$ in the following.
This settles the cases where $\mathsf{c}=0$ or $\mathsf{c}_2=0$ since both situations imply that $2\mathsf{h}'+\mathsf{c}' \leq 2\mathsf{h} < g_1.$

From now on we may assume $\mathsf{c},\mathsf{c}_2\neq 0$ which implies that $\mathsf{c}',\mathsf{c}_2\in [2].$

Let us first assume that $\mathsf{c}'=1$ and $2\mathsf{h}'+\mathsf{c}'\geq  2\mathsf{h}_2+\mathsf{c}_2.$
If $\mathsf{c}_2=1$ we immediately obtain $\mathsf{h}'\geq  \mathsf{h}_2.$
Hence, we may find $\mathscr{D}^{(\mathsf{h}_2,\mathsf{c}_2)}_k$ as a minor in $D$ as desired.
So we may further assume that $\mathsf{c}_2=2.$
Since the two sums have different parity, we obtain $\mathsf{h}'>\mathsf{h}_2.$
We now use \cref{classification} to obtain the surface grid $D'$ of order $k$ with $\mathsf{h}'-1$ handles and $\mathsf{c}'+2$ crosscaps as a minor {within} $D.$ 
As $\mathsf{h}'-1\geq  \mathsf{h}_2$ and certainly $\mathsf{c}'+2>\mathsf{c}_2$ we find $\mathscr{D}^{(\mathsf{h}_2,\mathsf{c}_2)}_k$ as a minor in $D'$ as desired.

For the final case we assume $\mathsf{c}'=2$ and again $2\mathsf{h}'+\mathsf{c}'\geq  2\mathsf{h}_2+\mathsf{c}_2.$
Since this means $\mathsf{c}'\geq 2 \geq \mathsf{c}_2\geq 1$ we also obtain $\mathsf{h}'\geq \mathsf{h}_2$ which immediately implies the presence of a $\mathscr{D}^{(\mathsf{h}_2,\mathsf{c}_2)}_k$-minor in $D$ as desired.\label{o6yut7}\ans{Removed a sentence about centrality of $W$ which is not required by the statement.}
\end{proof}

\subsection{The proof of \cref{transposition}}
With \cref{intermingling} at hand our proof of \cref{transposition} becomes an analogue of the proof for Theorem 2.11 from \cite{KawarabayashiTW20Quicklyexcluding}.
However, we split the proof in half to make an intermediate structure accessible for future applications.
The reason for this is that we require our final $\Sigma$-decomposition to be $W$-central with respect to a large wall $W.$
The way $W$-centrality is achieved in \cite{KawarabayashiTW20Quicklyexcluding} is by removing additional vertex sets from the vortices which potentially damages the surrounding nests.
Some applications might need access to the nests around vortices after the fact, and thus we include an intermediate step here.

Moreover, we will need the following seminal result of Grohe.
We say that a graph $G$ is \emph{quasi-$4$-connected} if for every separation $(A_1,A_2)$ of order at most three in $G$ there exists a unique $i\in[2]$ such that $|A_i\setminus A_{3-i}|\leq  1.$

\begin{proposition}[\!\!\cite{grohe2016quasi}]\label{enthroning}
Every graph $G$ has a tree-decomposition $(T,\beta)$ of adhesion at most three such that for all $t\in V(T)$ the torso $G_t$ of $G$ at $t$ is a minor of $G$ that is either quasi-$4$-connected or a complete graph of order at most four.
\end{proposition}

We call a graph $H$ that appears as the torso of a node in a tree-decomposition of $G,$ as in \cref{enthroning}, a \emph{quasi-$4$-connected component} of $G.$

We now provide the following consequence of \cref{intermingling} that reveals 
several features of the local structure of a graph where a Dyck-grid is excluded. 
\cref{efficiency} is of independent interest and may serve as the departure point of future developments.

\begin{lemma}\label{efficiency}
  Let $\mathsf{h}_1,\mathsf{h}_2,\mathsf{c}_2$ be non-negative integers where $\mathsf{c}_2\in[0,2]$ and $0\leq  \mathsf{h}_2<\mathsf{h}_1.$  We set 
 \begin{eqnarray*} 
g_1 & \coloneqq & 2\mathsf{h}_1,\\ 
g_2 & \coloneqq & 
\begin{dcases*}
2\mathsf{h}_2+\mathsf{c}_2
 & if $\mathsf{c}_2\neq 0$\,, \\[1ex]
0
 &otherwise\,  
\end{dcases*}, and\\
g & \coloneqq & \max\{ g_1,g_2\}.
\end{eqnarray*}
Let  also $t\geq 1,$ $k,$ $r,$ and $q$ be non-negative integers, where, $r\geq  16\cdot (162^{2g}+4)k$ and $q\geq \max\{k,t\},$ and set 
  \begin{eqnarray*}
  M& \coloneqq &  2(162^{2g}+4)q,\\
 \mu & \coloneqq & \left\lfloor \frac{1}{10^5t ^{26}}t^{10^7t^{26}}\cdot M\right\rfloor, \mbox{~and~}\\
 R & \coloneqq & 49152\cdot t^{24}(r+4M)+\mu.
  \end{eqnarray*}

\noindent  If a graph $G$ contains an $R$-wall $W,$
  then either, 
  \begin{enumerate}
  \item $H$ contains a $K_t$-minor grasped by $W,$
  \item $G$ contains an expansion $D$ of \textsl{(a)} $\mathscr{D}^{(\mathsf{h}_1,0)}_k$ or \textsl{(b)} $\mathscr{D}^{(\mathsf{h}_2,\mathsf{c}_2)}_k,$ where outcome (b) is allowed if and only if $\mathsf{c}_2\neq 0,$  such that $\mathcal{T}_D\subseteq \mathcal{T}_W,$ or
  \item there exist 
  \begin{itemize}
  \item a surface $\Sigma\coloneqq\Sigma^{(\mathsf{h},\mathsf{c})}$ where $2\mathsf{h} < 2\mathsf{h}_1$ in case $c=0,$ and $2\mathsf{h}+\mathsf{c} < 2\mathsf{h}_2+\mathsf{c}_2$ if $\mathsf{c}_2\neq 0.$
  \item a set $A\subseteq V(G)$ of size at most $\mu,$ and 
  \item a $\Sigma$-decomposition $\delta =(\Gamma,\mathscr{D})$ of $G-A$ with the following properties:
  \begin{enumerate}
  \item $\delta $ is of breadth at most $2t^2$ and depth at most $\mu,$
  \item if we define $\mathsf{c}_0$ and $\mathsf{h}_{0}$ as follows: 
  \begin{itemize}
 \item if $\mathsf{c}=0,$ let $\mathsf{c}_0\coloneqq 0$ and $\mathsf{h}_{0}\coloneqq \mathsf{h},$
 \item  if $\mathsf{c}>0$ and $\mathsf{c}+2\mathsf{h}$ is even, let $\mathsf{c}_0\coloneqq 2$ and $\mathsf{h}_{0}\coloneqq \frac{\mathsf{c}+2\mathsf{h}}{2}-1,$ and 
 \item otherwise, let $\mathsf{c}_0\coloneqq 1$ and $\mathsf{h}_{0}\coloneqq \frac{\mathsf{c}+2\mathsf{h}-1}{2},$ 
 \end{itemize}
 then there exists an $(\mathsf{h}_{0},\mathsf{c}_0;q)$-Dyck-wall $D$ in $G-A$ such that $D$ is grounded in $\delta$ and $\mathcal{T}_D\subseteq \mathcal{T}_W$
  \item there exists an $r$-subwall $W'$ of $W$ such that $W'$ is flat in $\delta,$ and $W'$ is drawn in the interior of the disk defined by the trace of the boundary of the simple face of $D,$
  \item for each vortex cell $v\in C(\delta)$ there exists a nest $\mathcal{C}_v$ of order $M$ in $\delta $ around the unique disk $\Delta\in\mathscr{D}$ corresponding to $v,$
  \item for each vortex cell $v\in C(\delta )$ let $\Delta_v\subseteq\Sigma$ be the disk bounded by the trace of $C_M\in\mathcal{C}_v,$ then for each pair of distinct vortex cells $v,v'\in C(\delta )$ we have that $\Delta_v\cap \Delta_{v'}=\emptyset,$
  \item let $\widetilde{\Delta}$ be the disk in $\Sigma$ which is bounded by the trace of the exceptional face of $D$ and does not contain any vertex of $W',$ for every every vortex $v \in C(\delta )$ we have $\Delta_v\subseteq \widetilde{\Delta},$
  \item let $X$ be the set of all vertices of $G$ drawn in the interior of cells of $\delta $ and let $G_{\delta }$ be the torso of the set $V(G)\setminus(X\cup A)$ in $G-A,$ then $G_{\delta }$ is a quasi-$4$-connected component of $G-A,$ and
  \item finally if holds that $\mathcal{T}_{W'}\subseteq \mathcal{T}_W$ and $\mathcal{T}_D\subseteq \mathcal{T}_W.$
  \end{enumerate}
  \end{itemize}
  \end{enumerate}
\end{lemma}

\begin{proof}
By our choice of parameters we may immediately call upon \cref{intermingling}.
The first possible outcome of \cref{intermingling} yields a $K_t$-minor grasped by $W$ as desired.
The second outcome of \cref{intermingling} provides us with a minor model of $\mathscr{D}^{(\mathsf{h}_1,0)}_k$ or, in case ${\sf c}_2\neq 0,$ of $\mathscr{D}^{(\mathsf{h}_2,\mathsf{c}_2)}_k.$

The third outcome provides us with integers $\mathsf{h},\mathsf{c}\geq  0$ and s surface $\Sigma=\Sigma^{(\mathsf{h},\mathsf{c})}$ where none of the graphs $K_t,$ $\mathscr{D}^{(\mathsf{h}_1,0)}_k,$ and, if ${\sf c}_2\neq 0,$ $\mathscr{D}^{(\mathsf{h}_2,\mathsf{c}_2)}_k$ embeds.
Moreover, we obtain a set $A\subseteq V(G)$ of size at most $\mu $ and a $\Sigma$-configuration $(W_1,W_0,\delta ,o_0,o_1,\Delta_0,\mathcal{C}_1,\mathcal{C}_2,\mathcal{P}_1,\mathcal{P}_2,\gamma)$ of $G-A$ with parameters $(r,M,2M,M,M,\mathsf{h},\mathsf{c}).$

Additionally, we have that the vortex society of $o_1$ has a rendition $\rho_1=(\Gamma_1,\mathscr{D}_1)$ of breadth at most $2t^2$ and depth at most $\mu $ such that
\begin{itemize}
  \item for each vortex cell $c\in C(\rho_1)$ there exists a nest $\mathcal{C}_c$ of order $M$ in $\rho_1$ around the unique disk $\Delta\in\mathscr{D}$ corresponding to $c,$ and 
  \item for each vortex cell $c\in C(\rho_1)$ let $\Delta_c\subseteq\Sigma$ be the disk bounded by the trace of $C_M\in\mathcal{C}_c,$ then for each pair of distinct vortex cells $c,c'\in C(\rho_1)$ we have that $\Delta_c\cap \Delta_{c'}=\emptyset.$
\end{itemize}

We extend $\delta $ to a $\Sigma$-decomposition $\delta '=(\Gamma',\mathscr{D})$ of $G-A$ by fixing a vortex free rendition of the vortex society of $o_0$ such that the wall $W_0$ is flat in the rendition and fixing a rendition of breath at most $2t^2$ and depth at most $\mu $ of the vortex society of $o_1.$
Thus, $\delta '$ has breath at most $2t^2$ and depth at most $\mu .$

Next, \cref{pindarischen} provides us with a minor-model $D'$ of $\mathscr{D}^{(\mathsf{h}_{0},\mathsf{c}_0)}_{2q}$ which is grounded in $\delta '.$
Notice that $D'$ contains an $(\mathsf{h}_{0},\mathsf{c}_0;q)$-Dyck-wall $D$ as a subgraph which still is grounded in $\delta '.$
In particular, $D'$ is also grounded in $\delta ,$ implying that $D$ is as well.
From the construction of $D'$ from \red{the mixed surface} grid in \cref{pindarischen} it follows that $o_0$ sits in the disk bounded by the trace of the simple face of $D$ while $o_1$ belongs to the disk bounded by the exceptional face of $D.$

We now start to further modify $\delta '$ by iterating the following operation exhaustively.
Let $(Y_1,Y_2)$ be a separation of order at most three in $G-A$ such that $Y_1\cap Y_2$ consists of nodes of $\delta '$ and assume that $(Y_1,Y_2)$ does not \textit{cross} any separation induced by a non-vortex cell of $\delta'.$
That is, we assume that there does not exist a non-vertex cell $x\in C(\delta')$ such that $$V(\sigma'(x))\cap (Y_1\setminus Y_2)\neq\emptyset \text{~and~} V(\sigma'(x))\cap (Y_2\setminus Y_1)\neq\emptyset.$$
By the submodularity of separations we may always find such a separation $(Y_1,Y_2)$ if some separation of order at most three with only nodes in its separator exists.
We start with separations of order one, then consider separations of order two, and finally separations of order three.
Without loss of generality we may assume that $Y_1$ contains a row and a column of $W_0.$
Moreover, we assume that $Y_2$ is maximal with respect to set containment.
This implies that $Y_1\setminus Y_2$ is connected.
Indeed, it follows from the definition of $\Sigma$-configurations and the $3$-connectivity of Dyck-grids that also at most one branch vertex of $D$ may be contained in $Y_2\setminus Y_1.$
Now suppose $Y_2\setminus Y_1$ contains a node of $\delta '.$
Let $C_{(Y_1,Y_2)}$ be the collection of all cells $x$ of $\delta '$ such that $V(\Sigma'(x))\cap (Y_2\setminus Y_1)\neq\emptyset.$
It follows that $V(\sigma'(x))\subseteq Y_2$ for all $x\in C_{(Y_1,Y_2)}$ by our choice of $(Y_1,Y_2).$
We may now define a new cell $c_{(Y_1,Y_2)}$ replacing all cells in $C_{(Y_1,Y_2)}$ as follows.
Let $\Delta_{(Y_1,Y_2)}$ be the disk which is the closure of $\bigcup_{c\in C_{(Y_1,Y_2)}}c.$
Now let $c_{(Y_1,Y_2)}$ be the cell obtained by removing the points $\pi'^{-1}(y),$ $y\in Y_1\cap Y_2,$ from $\Delta_{(Y_1,Y_2)},$ and notice that all vertices of $Y_2\setminus Y_1$ are drawn into the interior of $c_{(Y_1,Y_2)}.$
Then forget all cells in $C_{(Y_1,Y_2)}.$
Let $\delta ''=(\Gamma'',\mathscr{D}'')$ be the $\Sigma$-decomposition obtained once the above operation cannot be applied any more.
Notice that $D$ is still grounded in $\delta ''$ and so is $W_0.$
The same is true for all cycles from the nests of the vortices of $\delta ''.$
This implies that the torso $G_{\delta''}$ contains a Dyck-wall $D''$ of the same order as $D$ which can be obtained by contracting some edges of $D$ and is quasi-$4$-connected.

If follows that the set $S$ of vertices of $G-A$ corresponding to the nodes of $\delta ''$ cannot be split further by separations of order at most three and $S$ is minimal with this property.
Hence, by applying \cref{enthroning} there must exist a bag of the resulting tree-decomposition whose vertex set is exactly $S.$

So far we can observe that $\delta ''$ satisfies condition a) since it is of breadth at most $2t^2$ and depth at most $\mu .$
It also holds that $D$ is grounded in $\delta ''$ and none of its vertices are drawn in $\Delta_c$ for any vortex of $\delta '',$ hence property b) is satisfied.
The wall $W_0$ is an $r$-subwall of $W$ which is drawn in the disk bounded by the trace of the simple face of $D$ and rooted in $\delta ''$ which implies property c).
Moreover, every vortex of $\delta ''$ is contained in the disk bounded by the exceptional face of $D$ as required by d) and all vortices are equipped with nests of order $M$ which is property e).

These nests are pairwise disjoint and belong to their private disks, hence f) is satisfied.
Finally, the construction above ensures that $G_{\delta ''}$ is quasi-$4$-connected and so we also get property g).

So we are left with discussing the tangles.
For this observe that, by the definition of $\Sigma$-configurations, $W_0$ is a subwall of $W_1,$ $W_1$ is a subwall of $W,$ and $(\mathcal{P}_1\cup\mathcal{P}_2,\mathcal{C}_2)$ overlays $W_1.$
Recall that $D$ consists entirely of cycles from $\mathcal{C}_2$ and the paths from $\mathcal{P}_1$ and $\mathcal{P}_2.$
Since $r,q< M$ it follows that the $W$-majority side of any separation of order less than $r$ must also be the $W_0$-majority side and the $W$-majority side of any separation of order less than $q$ must also be the $D$-majority side.
Hence, also property h) holds as desired and our proof is complete.
\end{proof}

An advantage of \cref{efficiency} over \cref{transposition} is that here the parameter $M$ is freely choosable, as long as it is at least  $t.$
This means that one may increase the order of the Dyck-wall and the nests around the vortices arbitrarily and independent of the excluded minors.
We continue with the final step in the proof of \cref{transposition}.

\paragraph{Deriving \cref{thm_GMST3}.}
With regards to improving the bound on the Euler-genus from \cite{KawarabayashiTW20Quicklyexcluding} let us fix some graph $H.$
We set $g_1$ to be the smallest Euler genus of an orientable surface where $H$ embeds and $g_2$ to be the smallest Euler genus of a non-orientable surface where $H$ embeds. 
We then set $t\coloneqq |V(H)|.$
Then, one needs to set $k$ to be $\Theta\big( 162^{2\max\{ g_1,g_2\}} (\max\{g_1,g_2 \})+t^2 \big).$
This ensures that \cref{territorial2} witnesses the presence of an $H$-minor, even if the surface resulting from \cref{assemblies} has an Euler genus below the bound guaranteed there.
This concludes the proof of \cref{thm_GMST3}.
\medskip

The proof of \cref{efficiency} can be turned into an algorithm
that given a graph $G,$ $k,\mathsf{h}_1,\mathsf{h}_{2},\mathsf{c}_{2}, r, M$ and an $R$-wall as above, computes the structures implied by \cref{efficiency}
in time $$\poly(r+q)\cdot  2^{\cdot 2^{\Ocal(\mathsf{h}_1)} \poly(t+k)}|V(G)|^{2}+\Ocal(|V(G)|^3).$$
This running time comes from two different sources.
The quadratic, in $|V(G)|,$ part is due to us calling \cref{intermingling} which in turn calls \cref{pindarischen} from \cite{KawarabayashiTW20Quicklyexcluding}. Here the only additional work that needs to be done is the extraction of the Dyck-wall as described in the proof of \cref{intermingling}.
Since this is purely handling a bounded size object, the main contribution stems indeed from \cref{pindarischen}.
The second term, cubic in $|V(H)|,$ comes from the refinement of our $\Sigma$-decomposition.
In particular from the requirement that $G_{\delta''}$ is quasi-$4$-connected in the end which makes it necessary to consider all possible separations of order at most three.

\begin{proposition}[\!\!\cite{robertson1990graph,KawarabayashiTW20Quicklyexcluding}]\label{resembling}
  Let $d\geq  0$ be an integer.
  Every society of depth at most $d$ has a linear decomposition of adhesion at most $d.$
\end{proposition}

\begin{proof}[Proof of \cref{transposition}]
Let us start by discussing the numbers.
We set
\begin{align*}
  M & \coloneqq t,\\
  s_2 & \coloneqq 2t,\\
  r' & \coloneqq r + 16t + \left\lfloor \frac{1}{10^5t^{24}}t^{10^7t^{26}} \right\rfloor\text{, and}\\
  \mu  & \coloneqq \left\lfloor \frac{1}{10^5t^{25}}t^{10^7t^{26}} \right\rfloor.
\end{align*}
Let $R\coloneqq 49152t^{24}(r'+2s_2)+\mu $ be the constant from \cref{intermingling} and notice that
\begin{align*}
R=49152t^{24}(r'+2s_2)+\mu  & = 49152t^{24}(r'+4t) + \mu \\
& \leq  49152t^{24}\left(r + 16t + \left\lfloor \frac{1}{10^5t^{24}}t^{10^7t^{26}} \right\rfloor + 4t\right)+\left\lfloor \frac{1}{10^5t^{25}}t^{10^7t^{26}} \right\rfloor\\
& \leq  49152t^{24}(r + 16t)+\frac{49152t^{24}}{10^5t^{24}}t^{10^7t^{26}} + 2\cdot 10^5t^{25} + \left\lfloor \frac{1}{10^5t^{25}}t^{10^7t^{26}} \right\rfloor\\
& \leq  49152t^{24}(r +16t)+\frac{1}{2}t^{10^7t^{26}}+2\cdot\frac{10^5t^{25}}{10^5t^{25}}t^{10^7t^{26}}\\
& \leq  49152t^{24}(r +16t)+\left\lceil \frac{5}{2}t^{10^7t^{26}}\right\rceil = \mathbf{R}_{g_1,g_2}.
\end{align*}
Moreover, we have ensured that $r'\geq  16t.$
Let $W$ be an $\mathbf{R}_{g_1,g_2}$-wall in a graph $G.$

We may call upon \cref{efficiency} with the numbers $k,\mathsf{h}_1,\mathsf{h}_2,\mathsf{c}_2,$ $r',s_2, M$ and $\mu $ and the wall $W.$
In case we find $K_t,$ $\mathscr{D}^{(\mathsf{h}_1,0)}_k,$ or $\mathscr{D}^{(\mathsf{h}_2,\mathsf{c}_2)}_k$ as a minor we are done.
This is particularly true by our choice of $t$ as in case we find the $K_t$-minor, it will contain both $\mathscr{D}^{(\mathsf{h}_1,0)}_k$ and $\mathscr{D}^{(\mathsf{h}_2,\mathsf{c}_2)}_k$ as minors.
So we may assume the final outcome of \cref{efficiency} to hold.

This means we find a set $A\subseteq V(G)$ of size at most $\mu $ and $\Sigma$-decomposition $\delta =(\Gamma,\mathscr{D})$ of $G-A$ satisfying properties a) to h) from \cref{efficiency}.
What remains is to ensure that $\delta $ is $W$-central.

Let $y_1,\dots,y_{n}\in C(\delta )$ be the vortices of the $\Sigma$-decomposition $\delta .$
Let $P_1,\dots,P_{\mathbf{R}_{g_1,g_2}}$ be the rows and $Q_1,\dots,Q_{\mathbf{R}_{g_1,g_2}}$ be the columns of $W.$
Moreover, let 
\begin{align*}
\mathcal{W}\coloneqq\{ P_i\cup Q_j \mid i,j\in[\mathbf{R}_{g_1,g_2}]\text{ where }(V(P_i)\cup V(Q_j))\cap A=\emptyset\}.
\end{align*}
Observe that for every $T\in\mathcal{W}$ there exist at least $r'-\mu $ vertex disjoint paths in $W-A$ that link $T$ and $W_0.$
Thus, the only cells $y\in C(\delta ')$ such that an element of $\mathcal{W}$ is a subgraph of $\sigma_{\delta '}(y)$ are the vortices $y_i,$ $i\in[n].$
\bigskip

\textbf{Claim:} \textit{Let $i\in[n]$ and let $(G_i,\Omega_i)$ be the vortex society of $y_i.$
There exists a set $A_i\subseteq V(G_i)$ with $|A_i|\leq  2\mu +1$ such that $G_i-A_i$ does not contain any element of $\mathcal{W}$ as a subgraph.}

\smallskip
By \cref{resembling} there exists a linear decomposition $\langle X_1,\dots,X_m,v_1,\dots,v_m\rangle$ of $(G_i,\Omega_i)$ which has adhesion at most $\mu .$
For every $T\in\mathcal{W}$ with $T\subseteq G_i$ let $a_T\in[m]$ be the smallest integer such that $V(T)\cap X_{a_T}\neq\emptyset$ and let $b_T\in[m]$ be the largest integer with $V(T)\cap X_{b_T}\neq\emptyset$
Notice that the fact that each $T\in\mathcal{W}$ with $T\subseteq G_i$ is connected in $G_i$ implies that $V(T)\cap X_j\neq\emptyset$ for all $j\in[a_T,b_T].$
Let $b$ be the smallest value in the set $\{ b_T \mid T\in\mathcal{W}\text{ and }T\subseteq G_i \}.$
It follows from the fact that $T\cup T'$ is connected for all $T',T\in\mathcal{W}$ that $V(T)\cap X_b\neq \emptyset$ for every $T\in\mathcal{W}$ with $T\subseteq G_i.$

Let $A_i\coloneqq (X_b\cap X_{b-1})\cup(X_b\cap X_{b+1})\cup\{ v_b\}.$
Since the linear decomposition is of adhesion at most $\mu $ we obtain that $|A_i|\leq  2\mu +1.$

Moreover, $G_i-A_i$ does not contain any member of $\mathcal{W}$ as a subgraph as desired.
To see this let us assume, towards a contradiction, that there exists some $T\in\mathcal{W}$ with $T\subseteq G_i$ such that $T\subseteq G_i-A_i.$
Since $A_i$ separates the vertices of $X_b\setminus A_i$ from the rest of $G$ and $T$ must intersect $X_b,$ the only way how $T$ could avoid $A_i$ is if $V(T)\subseteq X_b\setminus A_i.$
Recall that there exist $r'-\mu $ pairwise vertex disjoint paths in $W-A$ that that link $T$ to $W_0$ where $V(W_0)\cap V(G_i)=\emptyset.$
Furthermore, we have
\begin{align*}
r'-\mu  & = r + 16t + \left\lfloor\frac{1}{10^5t^{24}}t^{10^7t^{26}}\right\rfloor-\left\lfloor\frac{1}{10^5t^{25}}t^{10^7t^{26}} \right\rfloor\\
& \geq  \frac{t}{10^5t^{25}}t^{10^7t^{26}}-\frac{1}{10^5t^{25}}t^{10^7t^{26}}-1\\
& = \frac{t-1}{10^5t^{25}}t^{10^7t^{26}}-1\\
& > 3\mu  > |A_i|.
\end{align*}
So $A_i$ cannot separate $T$ from $W_0,$ and thus we have reached a contradiction and the proof of the claim is complete.
\bigskip

Set $\hat{A}\coloneqq A\cup\bigcup_{i\in[n]}A_i$ and observe that
\begin{align*}
  |\hat{A}| & \leq  2t^2(2\mu +1)+\mu \\
  & \leq  \frac{5t^2}{10^7t^{25}}t^{10^7t^{26}}\\
  & \leq  t^{10^7t^{26}}.
\end{align*}
Moreover, from the construction of $\hat{A}$ we obtain that $\hat{A}\setminus A\subseteq \bigcup_{i\in n}V(\sigma_{\delta '}(y_i)),$ that is, every vertex of $\hat{A}$ that does not belong to $A$ belongs to some vortex of $\delta .$
In particular, this means that $\hat{A}$ is disjoint from $W_0.$
Finally, let $\delta '$ be the restriction of $\delta $ to the graph $G-\hat{A}.$
Now, $\delta '$ is a $\Sigma$-decomposition of $G-\hat{A}$ which is $W$-central, has breadth at most $2t^2$ and depth at most $\mu \leq  t^{10^7t^{26}}.$ 
Moreover, $W_0\subseteq W-\hat{A}$ contains an $r$-wall which is flat in $\delta '$ and thus, our proof is complete.
\end{proof}

\section{From local to global structure theorems}\label{locglob}

Finally, we have everything in place to prove a \red{global}\ans{used to say ``local''} version of our, so far local, structural theorems.
To close the gap in the Euler-genus left by Kawarabayashi, Thomas, and Wollan \cite{KawarabayashiTW20Quicklyexcluding} we would need to provide a global version of \cref{efficiency}.
However, the main goal of this paper is to understand the resulting structure when the universal graphs of some orientable and some non-orientable surface are excluded.
Since the general strategy of the proof \red{--} we follow the proof of \cite{KawarabayashiTW20Quicklyexcluding}, an earlier version of this had appeared in \cite{robertson1991graph} \red{--}\ans{made a slight change to the formatting to improve readability} stays the same for all such ``globali\red{z}ation%
-type arguments, we only present here the global version of \cref{transposition}.
Proving a global version of \cref{efficiency} follows along the same line of arguments.

Let $k,$ $\mathsf{h}_1,$ $\mathsf{h}_2,$ $\mathsf{c}_2,$ and $r$ be non-negative integers with $\mathsf{c}_2\in[0,2]$ and $0\leq \mathsf{h}_2<\mathsf{h}_1.$
Let $g_1\coloneqq 2\mathsf{h}_1,$ $g_2\coloneqq 2\mathsf{h}_2+\mathsf{c}_2$ if $\mathsf{c}_2\neq 0$ and $g_2\coloneqq 0$ otherwise.
Moreover, let $g\coloneqq\max\{ g_1,g_2\}$ and fix $t\coloneqq (162^{2g}+4)k.$
Recall the definitions of the functions $\mathbf{a}_{g_1,g_2}=t^{10^7t^{26}},$ $\mathbf{v}_{\mathsf{b},g_1,g_2}=2t^2,$ and $\mathbf{v}_{\mathsf{d},g_1,g_2}=t^{10^7t^{26}}$ from \cref{transposition}.
We define the following function.
Let $c_1$ be the constant from \cref{thm_algogrid}.
We set
\begin{align*}
  \mathbf{a}(g_1,g_2,k,r) &\coloneqq 36c_1\Big( 49152t^{24}(r+16t)+\left\lceil\frac{5}{2}t^{10^7t^{26}}\right\rceil\Big)^{20}+4\red{,}\\
  \red{\text{where }} & \red{g\coloneqq\max\{ g_1,g_2\} \text{ and }t\coloneqq (162^{2g}+4)k.}
\end{align*}
\ans{$g$ and $t$ used to be undefined here}

\begin{theorem}\label{thm_globalstructure}
Let $k,$ $\mathsf{h}_1,$ $\mathsf{h}_2,$ $\mathsf{c}_2,$ and $r$ be non-negative integers with $\mathsf{c}_2\in[0,2]$ and $0\leq \mathsf{h}_2<\mathsf{h}_1.$
Let $g_1\coloneqq 2\mathsf{h}_1,$ $g_2\coloneqq 2\mathsf{h}_2+\mathsf{c}_2$ if $\mathsf{c}_2\neq 0$ and $g_2\coloneqq 0$ otherwise.
Moreover, let $g\coloneqq\max\{ g_1,g_2\}$ and fix $t\coloneqq (162^{2g}+4)k.$

Then for every graph $G$ and every integer $r\geq 3$ one of the following is true.
\begin{enumerate}
  \item $G$ contains a \textsl{(a)} $\mathscr{D}^{(\mathsf{h}_1,0)}_k$ or a \textsl{(b)} $\mathscr{D}^{(\mathsf{h}_2,\mathsf{c}_2)}_k$-expansion $D$ where case \textsl{(b)} is allowed as an outcome if and only if $g_2\neq 0,$ or
 
  \item $G$ has a tree decomposition $(T,\beta)$ of adhesion less than $\mathbf{a}(g_1,g_2,k,r)$ such that for every $d\in V(T),$ either $|\beta(d)|\leq 4\mathbf{a}(g_1,g_2,k,r),$ or there exists a set $A_d\subseteq \beta(d)$ of size at most $4\mathbf{a}(g_1,g_2,k,r)$ and integers $\mathsf{h}_d,\mathsf{c}_d,$ $\mathsf{c}_d\in[0,2],$ together with a surface $\Sigma=\Sigma^{(\mathsf{h}_d,\mathsf{c}_d)}$ where neither $\mathscr{D}^{(\mathsf{h}_1,0)}_i$ nor, in case $\mathsf{c}_2\neq0,$ $\mathscr{D}^{(\mathsf{h}_2,\mathsf{c}_2)}_i$ embeds for all $i\in\mathbb{N},$ such that for the torso $G_d$ of $G$ at $d,$ the graph $G_d-A_d$ has a $\Sigma$-decomposition $\delta$ of width at most $\mathbf{a}(g_1,g_2,k,r)$ and breadth at most $\mathbf{v}_{\mathsf{b},g_1,g_2}$ such that
  \begin{enumerate}
    \item there exists an $r$-wall $W_d$ which is flat in $\delta,$
    \item there is no vertex of $G_d-A_d$ which is drawn in the interior of a non-vortex cell of $\delta,$
    \item $G_d-A_d$ is a minor of $G,$ and
    \item for every neighbour $d'$ of $d$ in $T$ it holds that $|\big(\beta(d)\cap\beta(d')\big)\setminus A_d|\leq 3.$
  \end{enumerate}
\end{enumerate}
Moreover, there exists a function $g\colon\mathbb{N}^4\to\mathbb{N}$ and an algorithm that finds one of these outcomes in time $g(k,\mathsf{h}_1,r)\red{|V(G)|^2|E(G)|^2}\log(|V(G)|))$ where $g(k,\mathsf{h}_1,r)=2^{2^{2^{\Ocal(\mathsf{h}_1)}\poly(k+r)}}.$\ans{fixed the running time}
\end{theorem}

  \begin{proof}
  We prove that the following, slightly stronger, version of the first claim holds by induction on the number of vertices in $G.$
  
  \textsl{Strengthened claim:} \textsl{Let $G$ be a graph which does not have a minor isomorphic to $\mathscr{H}^{(j)}_{t_j}$ for any $j\in[k]$ and let $Z\subseteq V(G)$ be a set of size at most $3\alpha.$
  Then $G$ has a tree decomposition $(T,\beta)$ with root $f,$ where $Z\subseteq \beta(f),$ and which is of adhesion less than $\mathbf{a}(g_1,g_2,k,r)$ such that for every $d\in V(T)$ either
  \begin{itemize}
  \item[1.] $|\beta(d)|\leq  4\mathbf{a}(g_1,g_2,k,r),$ or
  \item[2.] there exists a set $A_d\subseteq \beta(d)$ of size at most $4\mathbf{a}(g_1,g_2,k,r)$ and integers $\mathsf{h}_d,\mathsf{c}_d,$ $\mathsf{c}_d\in[0,2],$ together with a surface $\Sigma=\Sigma^{(\mathsf{h}_d,\mathsf{c}_d)}$ where $\mathscr{D}^{(\mathsf{h}_1,0)}_k$ does not embed and, if $\mathsf{c}_2\neq 0,$ where also $\mathscr{D}^{(\mathsf{h}_2,\mathsf{c}_2)}_k$ does not embed such that for the torso $G_d$ of $G_d$ at $d,$ the graph $G-A_d$ has a $\Sigma$-decomposition $\delta$ of width at most $\mathbf{a}(g_1,g_2,k,r)$ and breadth at most $\mathbf{v}_{\mathsf{b},g_1,g_2}$ such that, if $X_d$ is the set of all vertices of $G_d-A_d$ drawn in the interior of the vortices of $\delta$ together with the vertices in $A_d,$
  \begin{enumerate}
    \item there exists an $r$-wall $W_d$ which is flat in $\delta,$
    \item there is no vertex of $G_d-A_d$ which is drawn in the interior of a non-vortex cell of $\delta,$
    \item $G_d-X_d$ is a minor of $G,$ and
    \item for every neighbour $d'$ of $d$ in $T$ it holds that $|\big(\beta(d)\cap\beta(d')\big)\setminus\big(A_d\cup X_d\big)|\leq 3.$
  \end{enumerate}
  \end{itemize}}
  
  Assume towards a contradiction that the \textsl{strengthened claim} is false.
  We choose $G,$ $Z,$ and $k,\mathsf{h}_1,\mathsf{h}_2,\mathsf{c}_2,r$ to form a counterexample minimizing $|V(G)|+|V(G)\setminus Z|.$
  It follows that $|V(G)|>4\mathbf{a}(g_1,g_2,k,r)$ since otherwise we could choose $(T,\beta)$ to be a trivial tree decomposition with a single node and a unique bag containing all of $V(G).$
  Moreover, we have $|Z|=3\mathbf{a}(g_1,g_2,k,r)$ since otherwise we could add an arbitrary vertex of $V(G)\setminus Z$ to $Z$ and thereby contradict our choice for the counterexample.

  Suppose there is a set $X\subseteq V(G)$ of order less than $\mathbf{a}(g_1,g_2,k,r)$ such that $|V(J)\cap Z|\leq \lfloor\frac{2}{3}|Z|\rfloor\leq\mathbf{a}(g_1,g_2,k,r).$
  In this case let us introduce a node $f$ and set $\beta(f)\coloneqq X\cup Z.$
  Now, for each component $J$ of $G-X$ let $G_J\coloneqq G[V(J)\cup X]$ and $Z_J'\coloneqq Z\cap V(J).$
  Notice that $|Z_J'|<3\mathbf{a}(g_1,g_2,k,r).$
  In case $V(G_J)=Z_J'$ there is nothing to show here, otherwise we may pick an arbitrary vertex $v_J\in V(G_J)\setminus Z_J'$ and set $Z_J\coloneqq Z_J'\cup \{ v_J\}.$
  Observe that $G_J,$ $Z_J,$ and $k,\mathsf{h}_1,\mathsf{h}_2,\mathsf{c}_2,r$ do not form a counterexample by our minimality assumption.
  Hence, there exists a tree decomposition $(T_i,\beta_i)$ with root $f_i$ and $Z_J\subseteq \beta(f_i)$ which satisfies all requirements of the \textsl{strengthened claim}.
  We may form a tree decomposition $(T,\beta)$ by joining $f$ to the vertices $f_1$ and $f_2,$ thereby creating the tree $T.$
  It is straight forward to check that $(T,\beta)$ meets all requirements of the \textsl{strengthened claim} which is a direct contradiction to the choice of our counterexample.
  Observe that it is possible to determine if such a decomposition exists, at least with a slightly worse approximation factor, by using the techniques from \cite{Bodlaender96aline}.

  Hence, it follows that $Z$ is $(\mathbf{a}(g_1,g_2,k,r),\frac{2}{3})$-well-linked.
  We may now use \cref{thm_algogrid} to find a large wall whose associated tangle is a truncation of $\mathcal{T}_Z$ in time $2^{\mathcal{O}(\mathsf{poly}(t)^{c_2})}\red{|V(G)||E(G)|^2}\log(|V(G)|).$
  \ans{fixed the running time here.}

With $G$ being a counterexample to the \textsl{strengthened claim} we know that $G$ cannot contain $\mathscr{D}^{(\mathsf{h}_1,0)}_k$ as a minor.
Moreover, if $\mathsf{c}_2\neq 0$ $G$ may also not contain $\mathscr{D}^{(\mathsf{h}_2,\mathsf{c}_2)}_k$ as a minor.
Hence, an application of \cref{transposition} provides us with non-negative integers $\mathsf{h},\mathsf{c}$ where $\mathsf{c}\in[0,2]$ and a surface $\Sigma=\Sigma^{(\mathsf{h},\mathsf{c})}$ where $\mathscr{D}^{(\mathsf{h}_1,0)}_k$ does not embed and, in case $\mathsf{c}_2\neq 0,$ $\mathscr{D}^{(\mathsf{h}_2,\mathsf{c}_2)}_k$ does not embed such that there is a set $A\subseteq V(G)$ with $|A|\leq \mathbf{a}_{g_1,g_2},$ a $W$-central $\Sigma$-decomposition $\delta$ of $G-A$ of breadth at most $\mathbf{v}_{\mathsf{b},{g_1,g_2}}$ and depth at most $\mathbf{v}_{\mathsf{d},{g_1,g_2}},$ and an $r$-subwall $W'$ of $W-A$ which is flat in $\delta.$
  
What remains is to show how, for each cell $c$ of $\delta $ which has more than two vertices on its boundary and at least one vertex in its interior, we obtain a smaller graph $G_c$ with a bounded size set that will take the place of $Z.$
For each such graph, as it does not form a counterexample, we then obtain a decomposition meeting the requirements of our \textsl{strengthened claim}.
By combining all of these tree decompositions we will finally find a decomposition for $G$ and thereby contradict the existence of a counterexample.
  
Before we do this, however, let $d\coloneqq 2\mathbf{v}_{\mathsf{d},{g_1,g_2}}+1$ and $a\coloneqq \mathbf{a}_{g_1,g_2}.$
Notice, that $a+d+3< \mathbf{R}_{g_1,g_2} < \mathbf{a}(g_1,g_2,k,r).$
Hence, for \textsl{any} set $Y\subseteq V(G)$ of size at most $a+d+3$ there is a unique component $J$ of $G-Y$ such that $|V(J)\cap Z|>\frac{1}{3}|Z|.$
Moreover, this component $J$ is exactly the component that contains the vertex set of an entire row and an entire column of $W$ since $\mathcal{T}_W$ is a truncation of $\mathcal{T}_Z.$
This means, in particular, that for any non-vortex cell $c$ of $\delta,$ at most $\frac{2}{3}|Z|$ vertices of $Z$ can be drawn in the interior of $Z.$
Moreover, by \cref{prefascist} each vortex of $\delta$ has a linear decomposition of adhesion at most $\frac{1}{2}d.$
Hence, also the interior of any bag of any such linear decomposition for any vortex of $\delta$ cannot contain more than $\frac{2}{3}|Z|$ vertices of $Z.$

These observations will allow us to obtain tree decompositions as desired for all parts of the graph we are about to separate from the graph drawn on the boundaries of the cells of $\delta .$

From now on let $X$ be the collection of all vertices of $G-A$ which are drawn in some vortex of $\delta$ together with the vertices of $A.$
 
Let $\mathcal{V}$ be the collection of all vortices of $\delta .$
For each $c\in\mathcal{V}$ let $(G_c,\Omega_c)$ be the corresponding vortex society, and let $n_c\coloneqq |V(\Omega_c)|.$
Let us fix a linear order $v^c_1,\dots,v^c_{n_c}$ of $V(\Omega_c)$ such that there exists a linear decomposition $\langle L^c_1,\dots,L^c_{n_c},v^c_1,\dots,v^c_{n_c}\rangle$ of adhesion at most $\mathbf{v}_{\mathsf{d},{g_1,g_2}}$ of $(G_c,\Omega_c).$
We set $Y^c_i\coloneqq (L^c_i\cap L^c_{i-1})\cup(L^c_i\cap L^c_{i+1})\cup\{ v^c_i\}$ where $L^c_0=L^c_{i+1}=\emptyset.$
Observe that every vertex drawn on the closure of $c$ must belong to some $L^c_i,$ moreover, we have that $\bigcup_{i=1}^{n_c}L^c_i\subseteq X_f.$

For every $c\in\mathcal{V}$ and $i\in[n_c]$ let $H^c_i\coloneqq G[L^c_i\cup A]$ and let $Y^c_i\subseteq L^c_i$ consisting of all vertices in $L^c_i$ that appear in at least one of the two adhesion sets of $L^c_i.$
Notice that $\big(V(G)\setminus(L^c_i\setminus Y^c_i)\big)\cap\big(L^c_i\cup A\big)$ is a set of order at most $d+a+1.$
Hence, it follows from the discussion above that $|Z_{c,i}|<\mathbf{a}(g_1,g_2,k,r)$ where $Z_{c,i}\coloneqq A\cup (L^c_i\cap Z)\cup Y^c_i.$
Notice that this argument also yields the, bound $\alpha$ on the width of the linear decomposition $\langle L^c_1,\dots,L^c_{n_c},v^c_1,\dots,v^c_{n_c}\rangle.$

By the minimality of our counterexample, we obtain a tree decomposition $(T^c_i,\beta^c_i)$ with root $f^c_i$ satisfying the requirements of the \textsl{strengthened claim} for every $H^c_i,$ $Z_{c,i},$ and $k,\mathsf{h}_1,\mathsf{h}_2,\mathsf{c}_2,r.$

Now let $\mathcal{C}$ be the collection of all non-vortex cells of $\delta .$
For every $c\in\mathcal{C}$ let $H^c$ be the subgraph of $G$ with vertex set $V(\sigma(c))\cup A$ and edge set $E(\sigma(c))$ together with all edges that have at least one endpoint in $A$ and the other in $V(H^c).$
Let $N^c$ be the set of vertices of $G$ which are drawn on the boundary of the closure of $c$ together with the vertices in $A$ and the vertices in $Z\cap V(H^c).$
Let $\Omega^c$ be the cyclic permutation of the vertices of $N^c\setminus A$ obtained by following along the boundary of $c$ clockwise.
By deleting $Z$ in addition to $A$ it is possible that we slightly changed the way $H^c$ attaches to the rest of $G.$
To fix this, we apply \cref{thm_twopaths} to $(H^c-Z,\Omega^c)$ and enhance $\delta $ by possibly refining $c$ into smaller cells wherever possible.
Note that the application of \cref{thm_twopaths} guarantees that the formation of a torso by removing the interiors of non-vortex cells yields a minor of $G.$
From now on, we will assume that all cells have been refined in this way.

Observe that $|N^c\setminus (Z\setminus A)|\leq a+3\leq a+d+3.$
By the discussion above this implies that $|Z\cap V(H^c-A)|\leq\frac{2}{3}|Z|.$
Hence, $|N^c|<3\mathbf{a}(g_1,g_2,k,r)$

Let $H^c_1,\dots H^c_{n_c}$ be the components of $H^c-Z-A.$
Let $Z^c_i\coloneqq N^c\cap V(H^c_i)$ for each $i\in[n_c].$
For each $c\in\mathcal{C}$ and $i\in[n_c]$ we obtain, from the minimality of our counterexample, a tree decomposition $(T^c_i,\beta^c_i)$ with root $f^c_i$ for $H^c_i,$ $Z^c_i,$ and $k,\mathsf{h}_1,\mathsf{h}_2,c_2,r,$ which satisfies the requirements of the \textsl{strengthened claim}.

We may now construct a tree decomposition for $G$ as follows.
Let $f$ be a vertex, let $$X_f\coloneqq Z\cup A\cup\bigcup_{c\in\mathcal{V}}\bigcup_{i=1}^{n_c}Y^c_i$$ and set $\beta(f)$ to be the collection of all ground vertices of $\delta$ together with the set $X_f.$
It follows from the construction of the $Z^c_i$ and the definition of $\delta $ that, for the torso $G_f$ of $G$ at $f,$ $G_f-A-A_f$ has a $\Sigma$-decomposition $\delta '$ of width at most $\mathbf{a}(g_1,g_2,k,r)$ and breadth at most $\mathbf{v}_{\mathsf{b},g_1,g_2}.$
Moreover, $G_f-X_f$ is a minor of $G$ since $X_f$ contains all vertices which are drawn in the vortices of $\delta '$ together with the apex set $A\cup A_f$ and anything that attaches to this graph does so in separations of order at most three which are obtained from applications of \cref{thm_twopaths}.
  
For every cell $c$ of $\delta$ and every $i\in[n_c]$ introduce the edge $f^c_{i}f.$
Enhance $\beta$ by $\beta^c_i$ and finally, let $(T,\beta)$ be the resulting tree decomposition with root $f.$
We conclude by observing that $(T,\beta)$ satisfies the requirements of the \textsl{strengthened claim} and thus contradicts the choice of our counterexample. 
\end{proof}

\paragraph{Acknowledgements:} We are thankful to Laure Morelle, Christophe Paul, Evangelos Protopapas,  and \label{jkhsfxv}Giannos Stamoulis, for their helpful comments on the containments and the presentation of this paper.

\red{Moreover, we are grateful to an anonymous\ans{we thank the reviewer for all the help in improving this manuscript.} referee for  many detailed and sharp remarks, greatly improving the writing and presentation of the paper.}

\bibliographystyle{plainurl}
\bibliography{literature}

\end{document}